\documentclass[10pt,reqno]{amsart}
\usepackage{amsmath, amssymb, amsthm}
\usepackage{thmtools}
\usepackage[colorlinks=true,allcolors=blue,backref=page]{hyperref}
\usepackage[noabbrev,capitalize,nameinlink]{cleveref}
\usepackage{color}

\usepackage{mathrsfs}
\usepackage{mathtools}

\crefname{equation}{}{}
\usepackage{fullpage}
\usepackage[noadjust]{cite}
\usepackage{graphics}
\usepackage{pifont}
\usepackage{tikz}
\usepackage{bbm}
\usepackage[T1]{fontenc}

\usetikzlibrary{arrows.meta}

\usepackage{environ}
\usepackage{framed}
\usepackage{url}
\usepackage[linesnumbered,ruled,vlined]{algorithm2e}
\usepackage[noend]{algpseudocode}
\usepackage[labelfont=bf]{caption}
\usepackage{cite}
\usepackage{framed}
\usepackage[framemethod=tikz]{mdframed}
\usepackage{appendix}
\usepackage{graphicx}
\usepackage[textsize=tiny]{todonotes}
\usepackage{tcolorbox}
\usepackage{enumerate}
\usepackage[shortlabels]{enumitem}
\allowdisplaybreaks[1]

\apptocmd{\sloppy}{\hbadness 10000\relax}{}{} % magic BibTeX spacing fix

\crefname{algocf}{Algorithm}{Algorithms}

\crefname{equation}{}{} %remove ``Equation''
\crefname{conjecture}{Conjecture}{Conjectures} %add ``Conjecture''
\AtBeginEnvironment{appendices}{\crefalias{section}{appendix}} %appendices

\crefformat{enumi}{#2#1#3}
\crefrangeformat{enumi}{#3#1#4 to~#5#2#6}
\crefmultiformat{enumi}{#2#1#3}%
{ and~#2#1#3}{, #2#1#3}{ and~#2#1#3}

\usepackage[color,final]{showkeys} %add in 'final' into parameter to remove showkeys

\colorlet{refkey}{orange!20}
\colorlet{labelkey}{blue!30}

\crefname{algocf}{Algorithm}{Algorithms}

\numberwithin{equation}{section}
\newtheorem{theorem}{Theorem}[section]
\newtheorem{proposition}[theorem]{Proposition}
\newtheorem{lemma}[theorem]{Lemma}

\crefname{claim}{Claim}{Claims}

\newtheorem*{question*}{Question}

\theoremstyle{definition}
\newtheorem{definition}[theorem]{Definition}

\newtheorem*{definition*}{Definition}
\newtheorem{assumption}[theorem]{Assumption}
\theoremstyle{remark}
\newtheorem{remark}[theorem]{Remark}

\newcommand{\mb}{\mathbb}

\newcommand{\on}{\operatorname}

\allowdisplaybreaks

\title{Rising GUE Eigenvalue Process from a Fixed Level}

\author[A1]{Zoe Himwich}
\address{Department of Statistics, University of California, Berkeley, Berkeley, CA 94705}
\email{himwich@berkeley.edu}

\begin{document}

\begin{abstract} We construct the multilevel correlation kernel for the rising $\textsc{GUE}$ eigenvalue process starting from a fixed initial configuration $x^{(m)}$, and show that it converges on short time scales (as quickly as $\text{polylog}(m)$) to the extended semi-discrete sine kernel. As an application, we show fixed-energy universality of bulk local statistics of complex Hermitian Wigner matrices matching the covariance structure of $\textsc{GUE}$ and with a finite $4+\varepsilon$ moment for $\varepsilon>0$. This application demonstrates that it is possible to obtain universality of bulk local statistics under near-optimal moment assumptions without using a Dyson Brownian motion relaxation step, which had been a key ingredient in many results on this topic.%, such as those by Erd\H{o}s and Yau \cite{EY15} and Bourgade, Erd\H{o}s, Yau, and Yin \cite{BEYY16}.  
\end{abstract}

\maketitle
\setcounter{tocdepth}{1}{
  \hypersetup{linkcolor=black}
  \tableofcontents
}

\section{Introduction}\label{sec:introduction}

\subsection{Preface}
In this article, we construct the determinantal kernel of the rising $\textsc{GUE}$ eigenvalue process started from a fixed configuration at level $m$ (\cref{t:kernel}). This kernel did not previously exist in the literature, though a closely related kernel, associated with the corner-cutting process starting from a fixed configuration, was constructed by Metcalfe \cite{Met13} (see \cref{p:closedkernel}). The construction of the kernel for the rising eigenvalue process starting from a fixed configuration presents distinct challenges. We study the asymptotics of this kernel, showing that under assumptions on the starting configuration which are satisfied with high probability for Wigner matrices, we can obtain uniform-on-compact convergence to the extended semi-discrete sine kernel (or Boutillier bead kernel \cite{B09}) after $\text{polylog}{(m)}$ steps of the rising process (\cref{t:convrate}). We are able to use these integrable results to obtain a new proof of universality of bulk local statistics at fixed-energy for Wigner matrices that match the covariance structure of the $\textsc{GUE}$ and have a finite $4+\varepsilon$ moment for $\varepsilon>0$ (\cref{t:gaps}).

Bulk universality of local statistics (see the Wigner-Dyson-Gaudin-Mehta conjecture \cite[Conjectures 1.2.1 and 1.2.2]{Meh04}) for specific types of Wigner matrices have been studied and resolved by many researchers. A non-exhaustive collection of works studying local universality and variants of the general Wigner-Dyson-Gaudin-Mehta conjecture in various settings includes \cite{GM60,Dys70,DKMVZ99,PS97,BI99,J12,John01,Shcher09,EKYY12,BEYY16,ELPRSY10,ERSTVY10,ESY11,ESYY12,EY12comm,EYY12,HLY15,LSY19,TV12,TV11,TV11b,Agg19,EY15,ELRJSY10,T13}. There are many approaches which appear in this literature: early results began by proving this conjecture in classically integrable models; Johansson \cite{John01} proved bulk universality for certain Gaussian-divisible Hermitian Wigner matrices; and a program initiated by Erd\H{o}s, Schlein, and Yau \cite{ESY09,ESY09a} established local semicircle laws for eigenvalues and delocalization for eigenvectors. Subsequent papers by Erd\H{o}s, P\'ech\'e, Ramirez, Schlein, and Yau \cite{ELPRSY10}; Erd\H{o}s, Schlein, and Yau \cite{ESY11}; and Erd\H{o}s, Ramírez, Schlein, Tao, Vu, and Yau \cite{ERSTVY10} employed Dyson Brownian motion as a Gaussian perturbation of a general Wigner matrix which can be shown to equilibrate over short time scales (including through a determinantal kernel for eigenvalues which are allowed to evolve for a short time under Dyson Brownian motion dynamics, which is employed in the case of \cite{ELPRSY10}).  Tao and Vu \cite{TV11} pioneered an approach (also making use of the local semicircle law) which uses a Lindeberg replacement argument for random matrices which have moments matching classical integrable ensembles up to a certain order.

The optimal assumptions under which universality is believed to hold are that $2+\varepsilon$ moments exist for some $\varepsilon>0$ and that the covariance structure matches the symmetry class \cite{Agg19}. For the complex Hermitian Wigner matrices considered in this paper, it is possible to obtain bulk local universality for matrices which match the covariance structure of $\textsc{GUE}$ and which have a finite $4+\varepsilon$ moment for some $\varepsilon>0$, which can be achieved by combining \cite[Theorem 2.2]{EY15} or \cite[Theorem 2.2]{BEYY16} with \cite[Lemma 7.6]{EKYY12}. The papers of G\"otze, Naumov, Tikhomirov, and Timushev \cite{GNTT18} and G\"otze, Naumov, and Tikhomirov \cite{GNT18} developed local semicircle laws, eigenvalue rigidity, and eigenvector delocalization for Wigner matrices under a finite $4+\varepsilon$ moment assumption.

We use a standard Lindeberg replacement argument, adapted to pairs of Wigner
matrices sharing an $m\times m$ principal submatrix, to compare the fixed-energy
bulk local statistics of the original Wigner matrix with those of the matrix
obtained by adding Gaussian rows and columns. Combining this result with \cref{t:convrate}, we are able to show universality of bulk local statistics at fixed-energy for complex Wigner matrices that match the covariance structure of $\textsc{GUE}$ and have a finite $4+\varepsilon$ moment for $\varepsilon>0$. 

Arguments for fixed-energy universality and fixed-index gap universality under various conditions typically employ a three-step strategy: $(1)$ proving a local semicircle law; $(2)$ adding a Gaussian perturbation and proving universality of the perturbed matrix; and $(3)$ comparing the original and perturbed matrices. Our argument is analogous to this strategy. We assume the local semicircle law and eigenvalue rigidity of \cite{GNTT18,GNT18} (see \cref{thm:rigid}), which is analogous to step $(1)$.
In this case, step $(2)$ corresponds to adding Gaussian rows and columns to the $m\times m$ Wigner matrix and showing convergence of the resulting kernel to the extended semi-discrete sine kernel, and step $(3)$ corresponds to the Lindeberg comparison strategy. A distinction between the step $(2)$ in this paper and that in other papers is that the Gaussian perturbation is localized in the final few rows and columns of the matrix, rather than a Gaussian noise applied to each entry. Other papers, such as \cite{G17}, have taken the approach of conditioning on a large part of a point process configuration and then using kernel methods to analyze the remaining piece, and our approach uses the same idea.

Finally, along the way to the construction of the kernel in \cref{t:kernel}, we also prove
    \begin{enumerate}
        \item A version of the Eynard-Mehta theorem adapted to rising processes which start at a fixed level (\cref{a:em}),
        \item That the kernel of \cref{t:kernel} can be obtained as a limit of the kernel of lozenge tilings of a suitable sequence of polygon domains (\cref{s:leolimit}). This result is related to an earlier result of Aggarwal and Gorin \cite{AG22}, which does the same for the unconditioned rising eigenvalue process. 
        \item An alternative derivation of the kernel constructed by Metcalfe \cite{Met13} (\cref{s:compare}).
    \end{enumerate}
\subsection{Random Matrix Ensembles, Eigenvalue Processes}
In this section, we make several definitions which we use throughout the paper. There will be additional setup in \cref{s:prelim}.
\begin{definition}[Wigner Ensemble]\label{d:wigner} For each $n\in\mathbb{N}$, let $H^{(n)}=\{h_{ij}^{(n)}\}_{i,j=1}^{n}$ be a random $n\times n$ Hermitian matrix. Assume that the random variables $\{h_{ij}^{(n)}\}_{1\leq i\leq j\leq n}$ are independent, that the off-diagonal entries $\{h_{ij}^{(n)}\}_{1\leq i<j\leq n}$ are identically distributed, that the diagonal entries $\{h_{ii}^{(n)}\}_{1\leq i\leq n}$ are identically distributed and real-valued, and that $h_{ij}^{(n)}=\overline{h_{ji}^{(n)}}$. For all $1\leq i< j\leq n$, the $h_{ij}^{(n)}$ satisfy 
\begin{align*}
    \mathbb{E}\left[h_{ij}^{(n)}\right] = 0, & & \mathbb{E}[\Re(h_{ij}^{(n)})^{2}]=\mathbb{E}[\Im(h_{ij}^{(n)})^{2}]=\frac{1}{2}, & & & & \mathbb{E}[\Re(h_{ij}^{(n)})\Im(h_{ij}^{(n)})] = 0, & &  \mathbb{E}\left[|h_{ij}^{(n)}|^{4+\varepsilon}\right]<C, 
\end{align*}  and for all $1\leq i\leq n$, 
\begin{align*}
    \mathbb{E}[h_{ii}^{(n)}]=0, & & \mathbb{E}[(h_{ii}^{(n)})^{2}] =1, & & \mathbb{E}[|h_{ii}^{(n)}|^{4+\varepsilon}] <C,  
\end{align*} for some $\varepsilon>0$ and constant, $C>0$ which do not depend on $n$. 
We refer to a random matrix $H^{(n)}$ which satisfies these properties as an $n\times n$ Wigner matrix. We refer to the collection $\{H^{(n)}\}_{n\in\mathbb{N}}$ as a Wigner ensemble. 
\end{definition}
By Wigner's semicircle law, if $H^{(n)}$ has eigenvalues $\{\lambda^{(n)}_{i}\}_{i\in[[n]]}$, then the empirical spectral measure of $\{\lambda^{(n)}_{i}/\sqrt{n}\}_{i\in[[n]]}$ converges weakly in probability to the semicircle distribution, $\rho_{sc}(x):=\mathbf{1}_{|x|\leq 2}(2\pi)^{-1}\sqrt{4-x^{2}}.$ Eigenvalues near a ``bulk'' energy $X\in(-2,2)$ correspond to eigenvalues with the normalization as in \cref{d:wigner} which are close to $X\sqrt{n}$.

There are many natural notions of stochastic processes on eigenvalues of random matrices; we will work with the rising eigenvalue process, defined below. 
\begin{definition}[Rising Eigenvalue Process] \label{d:gueprocess}  We define $\mathrm{Conf}_{n}(\mathbb{R})$ to be the collection of ordered $n$-tuples of the form $x^{(n)}=(x_{1}^{(n)},\cdots,x_{n}^{(n)})$ with $x_{i}^{(n)}\geq x_{j}^{(n)}$ when $i<j$. We then define $\mathcal{F}_{n}=\mathcal{B}(\mathrm{Conf}_{n}(\mathbb{R}))$. We also define $\mathrm{Conf}_{n}^{0}(\mathbb{R})$ to be the collection of ordered $n$-tuples with strict inequality, $x^{(n)}=(x_{1}^{(n)},\cdots, x_{n}^{(n)})$ with $x_{i}^{(n)}>x_{j}^{(n)}$ when $i<j$. 

Let $H=\{h_{ij}\}_{i,j\ge 1}$ be an infinite Hermitian random matrix with entries satisfying the same assumptions as those in \cref{d:wigner}. For each $n\in\mathbb{N}$, let $H^{(n)}$ be the $n\times n$ principal submatrix of $H$ and let $\lambda^{(n)}$ be the $\mathrm{Conf}_{n}(\mathbb{R})$-valued random variable given by its eigenvalue configuration, with law on $(\mathrm{Conf}_{n}(\mathbb{R}),\mathcal{F}_{n})$ denoted by $\mathbb{P}^{W}_{n}$. The rising Wigner eigenvalue process is the law $\mathbb{P}^{W}$ of the sequence defined by $\{\lambda^{(n)}\}_{n=1}^{\infty}$ on $\left(\prod_{n\in\mathbb{N}}\mathrm{Conf}_{n}(\mathbb{R}),\bigotimes_{n\in\mathbb{N}}\mathcal{F}_{n}\right)$. If $\pi_{n}$ denotes projection onto the $n$-th level, then $\mathbb{P}_{n}^{W}=\mathbb P^{W}\circ\pi_{n}^{-1}.$ 

By Cauchy's interlacing theorem, $\mathbb{P}^{W}$ is supported on elements of $\prod_{n\in\mathbb{N}}\mathrm{Conf}_{n}(\mathbb{R})$ such that $x^{(n)}\prec x^{(n+1)}$ for all $n\in\mathbb{N}$, where the notation $x^{(n)}\prec x^{(n+1)}$ denotes interlacing configurations, meaning that $x^{(n)},x^{(n+1)}$ satisfy $x^{(n+1)}_{i}\geq x^{(n)}_{i}\geq x^{(n+1)}_{i+1}$ for all $1\leq i\leq n$. 
\end{definition}
\begin{remark}[Gaussian Unitary Ensemble, Rising Process] \label{r:gue} 
Let $G^{(n)}=\{g_{ij}^{(n)}\}_{i,j=1}^{n}$ be an $n\times n$ Wigner matrix as in \cref{d:wigner}. We say that $G^{(n)}$ is an $n\times n$ Gaussian Unitary Ensemble matrix if, for $1\leq i<j\leq n$, $g_{ij}^{(n)}$ takes the form $2^{-1/2}(\xi_{ij}+i\eta_{ij})$, and for $1\leq i\leq n$, $g_{ii}^{(n)}$ takes the form $\xi_{ii}$ where $\xi_{ij},\eta_{ij}$ are independent, real-valued, and distributed as $\mathcal{N}(0,1)$ for all $1\leq i\leq j\leq n$.

The rising $\textsc{GUE}$ eigenvalue process is defined analogously to that for the general rising  Wigner eigenvalue process, starting instead from an infinite $\textsc{GUE}$ matrix, and is supported on $\mathrm{Conf}^{0}(\mathbb{R})=\prod_{n\in\mathbb{N}}\mathrm{Conf}_{n}^{0}(\mathbb{R})$. We use $\mathbb{P}^{\textsc{GUE}}$ to denote its law, and $\mathbb{P}^{\textsc{GUE}}_{n}$ to denote the projection to the $n$-th level. In the $\textsc{GUE}$ case, the transition densities with respect to the Lebesgue measure on $\mathrm{Conf}^{0}_{n+1}(\mathbb{R})$ from $x^{(n)}\in\mathrm{Conf}^{0}_{n}(\mathbb{R})$ to $x^{(n+1)}$, $p_{n,n+1}(x^{(n)};x^{(n+1)})$, are given by
    \begin{align*}
        p_{n,n+1}(x^{(n)};x^{(n+1)}) & = \frac{1}{\sqrt{2\pi}} \mathbf{1}_{x^{(n)}\prec x^{(n+1)}}\frac{\prod_{k<j}(x_{k}^{(n+1)}-x_{j}^{(n+1)})}{\prod_{k<j}(x_{k}^{(n)}-x_{j}^{(n)})} \exp{\left(-\sum_{k=1}^{n+1}\frac{(x_{k}^{(n+1)})^{2}}{2} + \sum_{k=1}^{n}\frac{(x_{k}^{(n)})^{2}}{2}\right)}.
    \end{align*} 
\end{remark} 

We also define what it means for two random matrices to match moments up to a given order.
\begin{definition}[Matching Moments]\label{d:matching} We say two complex random variables $\xi,\zeta$ match moments to order $k\geq 1$ if $\mathbb{E}[\Re(\xi)^{a}\Im(\xi)^{b}]=\mathbb{E}[\Re(\zeta)^{a}\Im(\zeta)^{b}]$ for all $a,b\in\mathbb{Z}_{\geq 0}$ such that $a+b\leq k$. Similarly, we say that two random matrices match moments to order $k\geq 1$ if all entries in both matrices match moments to order $k$. 
\end{definition}
\begin{remark} Therefore, all Wigner matrices considered in this paper  match $\textsc{GUE}$ to two moments. 
\end{remark}

The rising GUE eigenvalue process in probability literature goes back to Baryshnikov \cite{B01}. 
Johansson and Nordenstam \cite{JN06} showed that the rising $\textsc{GUE}$ eigenvalue process is determinantal. Adler, Nordenstam, and van Moerbeke \cite[Theorem 1.3 and Corollary 1.4]{ANV14} showed that the kernel of the rising $\textsc{GUE}$ eigenvalue process converges, under appropriate rescaling around a bulk point, and uniformly on compact sets, to a limiting kernel (see \cite[Theorem 1.3]{ANV14} with $t=t'=0$), which we will refer to as the extended semi-discrete sine kernel (it is also known as the Boutillier bead kernel).
\begin{definition}\label{d:extendedsine}
    The extended semi-discrete sine process with parameter $a\in \mathbb{C}$ such that $\Im(a)>0$ is the determinantal point process on $\mathbb{N}\times\mathbb{R}$ with kernel given by 
    \begin{align*}
        K_{a}^{\text{sine}}(n_{1},x_{1};n_{2},x_{2}) & = \begin{cases}
             \frac{1}{2\pi i}\int_{\overline{a}}^{a}z^{n_{1}-n_{2}}e^{z(x_{2}-x_{1})}dz & n_{1}\geq n_{2}
             \\ -\frac{1}{2\pi i}\int_{\Re(a)+i(\mathbb{R}\setminus[-\Im(a),\Im(a)])}z^{n_{1}-n_{2}}e^{z(x_{2}-x_{1})}dz & n_{2}>n_{1}
        \end{cases},
    \end{align*} with respect to the counting measure on $\mathbb{N}$ and the Lebesgue measure on $\mathbb{R}$. Both contours lie in the set $\{z:\Re(z)=\Re(a)\}$ and are oriented in the positive imaginary direction. The integral in the second expression is interpreted as \begin{align*}
        \lim_{L\to\infty}
\left(
\int_{\Re(a)-iL}^{\Re(a)-i\Im(a)}z^{n_{1}-n_{2}}e^{z(x_{2}-x_{1})} dz
+
\int_{\Re(a)+i\Im(a)}^{\Re(a)+iL}
z^{n_{1}-n_{2}}e^{z(x_{2}-x_{1})} dz\right).
    \end{align*}
    This kernel coincides with that of the Boutillier bead process \cite{B09}. When $n_{1}=n_{2}=n,$ up to a conjugation by a non-vanishing function (see \cref{r:conj}), $K_{a}^{\text{sine}}(n,x_{1};n,x_{2})$ is equal to the sine kernel with slope $\phi = \Im(a)$,
\begin{align}\label{e:sinekernel}
        K_{\phi}^{\text{sine}}(x_{1},x_{2}) & := \frac{\sin{\left(\phi(x_{2}-x_{1})\right)}}{\pi(x_{2}-x_{1})}.
    \end{align}  
    When $\phi=\pi$, we will refer to this kernel by the notation $K^{\text{sine}}(x_{1},x_{2})$. 
\end{definition}    

\subsection{Main Results}\label{s:mainresults}

Our main results are an explicit multilevel kernel for the rising $\textsc{GUE}$ eigenvalue process starting from a fixed configuration $x^{(m)}\in\mathrm{Conf}_{m}^{0}(\mathbb{R})$ (\cref{t:kernel}) and showing that it converges on $\text{polylog}(m)$-time scales to the extended semi-discrete sine kernel. 

For $x^{(m)}\in\mathrm{Conf}_{m}^{0}(\mathbb R)$, we define the rising $\textsc{GUE}$ eigenvalue process started at $x^{(m)}$ to be the Markov process on $\prod_{r=m+1}^{\infty}\mathrm{Conf}^{0}_{r}(\mathbb R)$
whose transition densities are the $\textsc{GUE}$ transition densities $p_{r,r+1}$ (see \cref{r:gue}), with initial condition
$x^{(m)}$.
\begin{theorem}\label{t:kernel} The $\textsc{GUE}$ rising eigenvalue process started at $x^{(m)}\in \mathrm{Conf}_{m}^{0}(\mathbb{R})$ is determinantal with respect to the reference measure on $\mathbb{N}\times\mathbb{R}$ defined by the counting measure on $\mathbb{N}$ and the Lebesgue measure on $\mathbb{R}$. For $n_{1},n_{2}\in\mathbb{N}_{>m}$ and  $x_{1},x_{2}\in\mathbb{R}$, the kernel of this point process is given by 
\begin{align*}
     K^{\textsc{GUE}}_{x^{(m)}}(n_{1}, x_{1} ;n_{2},x_{2})  & = \frac{(n_{1}-m)!}{(n_{2}-m-1)!}
    \frac{-1}{(2\pi i)^{2}}
    \int\limits_{\Gamma} d z\oint\limits_{c(x_{1})}dw
    \frac{(z-x_{2})^{n_{2}-m-1}}{(w-x_{1})^{n_{1}-m+1}}
    \frac{e^{z^{2}/2-w^{2}/2} }{w-z}
    \prod_{r=1}^{m}
    \frac{w-x_{r}^{(m)}}{z-x_{r}^{(m)}}.
\end{align*} In this expression, $\Gamma=b+i\mathbb{R}$ for $b\in\mathbb{R}$ such that $x_{2}<b$, and, for all $j\in [[m]]$ which satisfy $x_{j}^{(m)}>x_{2}$ $b$ also satisfies $b<x_{j}^{(m)}$. Similarly, if $x_{2}<x_{1}$, then $b<x_{1}$. The value of the integral is independent of any such choice of $b$. The contour $\Gamma$ is oriented in the positive imaginary direction, and the contour $c(x_{1})$ is a small positively oriented contour around $x_{1}$ containing none of the $x^{(m)}_{j}$ for $j\in[[m]]$ unless $x^{(m)}_{j}=x_{1}$. The contour $c(x_{1})$ is also chosen so as not to intersect $\Gamma$. 

We also define the notation $\widetilde{K}_{x^{(m)}}^{\textsc{GUE}}(n_{1},x_{1};n_{2},x_{2}):=K_{x^{(m)}}^{\textsc{GUE}}(n_{1}+m,x_{1};n_{2}+m,x_{2})$ for $n_{1},n_{2}\in\mathbb{N}$ and $x_{1},x_{2}\in\mathbb{R}$ to emphasize that this process starts at level $m$. 
\end{theorem}
In \cref{s:leolimit}, we also demonstrate that, under an appropriate choice of parameters, we can obtain this kernel as a limit of Petrov's \cite{Pet14} kernel of the polygon lozenge tiling. This is analogous to, but distinct from, an existing result of Aggarwal and Gorin \cite[Theorem 4]{AG22} which shows that the law of the rising $\textsc{GUE}$ eigenvalue process arises as a scaling limit of uniformly random lozenge tilings of large
domains with three adjacent straight boundary segments meeting at $120$ degrees, under the appropriate scaling.

A result of Metcalfe \cite{Met13} fixes the final eigenvalue configuration of rising process with $m$ steps (as opposed to fixing the initial state, in our case, which presents distinct challenges). We also demonstrate in \cref{s:compare} that we can obtain Metcalfe's kernel using the methods from \cref{s:kernel}. It is also possible to obtain this kernel as the limit of the kernel which appears in \cite[Theorem 1]{Pet14}, by a process similar to that outlined in \cref{s:leolimit}. 

The next result demonstrates convergence of the kernel of the rising $\textsc{GUE}$ eigenvalue process started from a fixed configuration $x^{(m)}\in\mathrm{Conf}_{m}^{0}$. The uniform-on-compact convergence which we show is strong enough to obtain convergence of gap statistics in bounded regions. The constants in this theorem may depend on the $\varepsilon,C>0$ which appear in \cref{d:wigner}.
\begin{theorem}\label{t:convrate} Let $T:=T(m)$ be an integer-valued sequence satisfying $ \log^{4}{(m)}\ll T(m)\ll m^{2/3}$. There exists a sequence of events $\Omega_{m}\in\mathcal{F}_{m}$ with $\mathbb{P}_{m}^{W}(\Omega_{m})\geq 1-m^{-25}$ such that for any compact set $K\subset \mathbb{N}\times \mathbb{R}$ and $\alpha\in (0,2)$ there exist $C:=C(K,\alpha)>0,m_{0}:=m_{0}(K,\alpha)\in\mathbb{N}$ such that for all $(n_{1},x_{1}),(n_{2},x_{2})\in K$, $m>m_{0},$ $X\in(-2+\alpha,2-\alpha)$, and $x^{(m)}\in\Omega_{m}$, 
    \begin{align}\label{e:convratebound}
        \big| \frac{1}{\sqrt{m}}\widetilde{K}_{x^{(m)}}^{\textsc{GUE}}\left(n_{1}+T, X\sqrt{m}+\frac{x_{1}}{\sqrt{m}}; n_{2}+T, X\sqrt{m}+\frac{x_{2}}{\sqrt{m}}\right) - K^{\text{sine}}_{a(X)}(n_{1},x_{1};n_{2},x_{2})\big|\leq CT^{-1/2},
    \end{align} where $a(X)=X/2+i\pi\rho_{sc}(X)$. 
\end{theorem} 
\begin{remark}
    We can relax the scaling assumptions on $T,m$ to $\log^{2}{(m)}\ll T(m)\ll m$ at the cost of obtaining a worse rate in \eqref{e:convratebound}. 
\end{remark} 

There is also a more general version of \cref{t:convrate}, 
which holds in the case where we do not have \cref{t:local}, but instead have some weaker assumptions on the configurations $x^{(m)}$, analogous to the ``local density'' and ``intermediate scale'' assumptions imposed in \cite{GP19}. We state this result, and the weaker assumptions which are required to prove it, in \cref{a:fullgen}, and outline its proof.

A related, earlier work of Huang \cite[Theorem 1.4]{H22} proves that the bulk scaling limit of the Wigner corner process over any fixed finite number of adjacent levels is the bead process associated with $\text{Sine}_{\beta}$. In the complex Hermitian $\beta=2$ case, this corresponds to the Boutillier bead process. The argument in that paper uses single-level bulk universality, eigenvalue rigidity, local interval-count variance bound, and eigenvector delocalization as inputs (whereas we use universality of a rising process as an input to prove bulk local universality).  

Finally, as an application of \cref{t:kernel} and \cref{t:convrate}, we obtain fixed-energy universality for complex Hermitian Wigner matrices with covariance structure matching $\textsc{GUE}$, as in \cref{d:wigner}, which demonstrates that it is possible to obtain universality of bulk local statistics at fixed energy. 
 \begin{definition}[Correlation Measure] \label{d:cormeas}
     The general class of Wigner matrices encompassed by \cref{d:wigner} includes cases where the $k$-point correlations may not have a density with respect to the Lebesgue measure. Therefore, we define the correlation measure $\widehat{\rho}_{k,n}^{W}$ such that for all $F:\mathbb R^k\to\mathbb R$ bounded and Borel-measurable,
\begin{align*}
\int_{\mathbb{R}^k} F(y_1,\ldots,y_k) 
\widehat{\rho}_{k,n}^{W}(dy_1,\cdots, dy_k)
:=\mathbb{E}^{W}_{n}\left[
\sum_{\substack{i_{1},\ldots,i_{k} \in [[n]]\\ \text{distinct} }}
F(x_{i_1}^{(n)},\ldots,x_{i_k}^{(n)})\right].
\end{align*} 
We also define $\widehat{\rho}_{k,X}^{W,(n)}$ as the push-forward of $\widehat{\rho}_{k,n}^{W}$ under $$(y_1,\ldots,y_k)
\mapsto
(
\rho_{sc}(X)\sqrt n(y_1-X\sqrt n),
\ldots,
\rho_{sc}(X)\sqrt n(y_k-X\sqrt n)
).$$ 
We define $\widehat{\rho}_{k,n}^{\textsc{GUE}}$ and
$\widehat\rho_{k,X}^{\textsc{GUE},(n)}$ analogously, and define $\widehat\rho_{k,X;x^{(m)}}^{\textsc{GUE},(m,T)}$ to be the correlation measure corresponding to the kernel $$\frac{e^{\frac{X}{2\rho_{sc}(X)}(x-y)}}{\rho_{sc}(X)\sqrt{m}} \widetilde{K}^{\textsc{GUE}}_{x^{(m)}}\left(T,X\sqrt{m}+\frac{x}{\rho_{sc}(X)\sqrt{m}}; T, X\sqrt{m}+\frac{y}{\rho_{sc}(X)\sqrt{m}}\right),$$ (with $n:=m+T$) which is appropriately rescaled to converge to $K^{\text{sine}}(x,y)$.
 \end{definition}
 As before, the constants in this theorem may also depend on the constants $\varepsilon$ and $C$ from \cref{d:wigner}. 
\begin{theorem}\label{t:gaps} For each $n\in\mathbb{N}$, let $H^{(n)}$ be an $n\times n$ Wigner matrix satisfying \cref{d:wigner}. Then, for any $\alpha\in (0,2)$,  $k\in\mathbb{N}$, $U\subset\mathbb{R}$ compact, and $F\in C^{1}_{c}(\mathbb{R}^{k})$ with $\mathrm{supp}(F)\subset U^{k}$, there exist $C:=C(U,k,\alpha)>0,c:=c(k,\alpha)>0,n_{0}:=n_{0}(U,k,\alpha)\in\mathbb{N}$ such that for all $X\in (-2+\alpha,2-\alpha)$ and $n>n_{0}$,
    \begin{align*}
        \left|\int_{\mathbb R^k} F(x_{1},\ldots x_{k})\widehat{\rho}_{k,X}^{W,(n)}(dx_{1},\ldots,dx_{k})-\int_{\mathbb R^k} F(x_{1},\ldots,x_{k}) \widehat{\rho}_{k,X}^{\textsc{GUE},(n)}(dx_{1},\ldots, dx_{k})\right|\leq C n^{-c}\|F\|_{C^{1}}.
    \end{align*} 
\end{theorem} 
In the argument, we will set $n:=m+T$, where $m,T$ are the parameters which appear in \cref{t:convrate}, so that $T=o(m)$.
\begin{remark}
    It is also possible to use \cref{t:convrate} and the notion of a Janossy density (see \cite{Bor11}), along with a generalization of the approximate independence argument of Tao \cite[pg. 17]{T13}, and a variant of the fixed-index Lindeberg result \cite[Theorem 15]{TV11} to show universality of fixed-index gaps through similar arguments to those outlined in \cref{s:carg}. 
\end{remark}

\subsection{Organization} In \cref{s:prelim} we set up the definition of a determinantal point process and the Eynard-Mehta theorem. The version of this theorem which we use in this paper is proved in \cref{a:em}. In \cref{s:prelim} we also discuss the local semicircle law and eigenvalue rigidity, which will be essential to the proof of \cref{t:convrate}.  
In \cref{s:kernel} we prove \cref{t:kernel} using the Eynard-Mehta theorem from \cref{s:prelim}. In \cref{s:leolimit} we show that this kernel can also be obtained as the limit of the kernel of lozenge tilings of an appropriate sequence of polygons. In \cref{s:convrate} we prove \cref{t:convrate} using a steepest descent analysis of the kernel from \cref{t:kernel}. \cref{s:steepest} proves a saddle point estimate, which is needed in that argument, and \cref{a:fullgen} states a generalized version of \cref{t:convrate} and outlines the proof.
Finally, in \cref{s:carg} we state a modified Lindeberg swap theorem analogous to \cite[Theorem 6.4]{EYY12} and apply it (as well as \cref{t:kernel} and \cref{t:convrate}) to prove \cref{t:gaps}.

\subsection{Acknowledgments} The author is very grateful to Mehtaab Sawhney for suggesting the topic and supervising this project. The author is also very grateful to Leonid Petrov for discussions concerning his work \cite{Pet14}; to Vadim Gorin and Leonid Petrov for discussions about their joint work \cite{GP19}; and to Amol Aggarwal, Ivan Corwin, Shirshendu Ganguly, and Matthew Nicoletti for additional helpful discussions and comments. This project was supported by NSF MSPRF 2503374. 

\section{Preliminaries}\label{s:prelim}
In this section we set up context which is involved in the construction of the main results, \cref{t:kernel}, \cref{t:convrate}, and \cref{t:gaps}.  In \cref{ss:dpp} we briefly define determinantal point processes and the Eynard-Mehta theorem, which provides a way of computing the determinantal kernel of these point processes, which we will use in \cref{s:kernel}. In \cref{ss:jan} we discuss how to extract the law of gaps from a determinantal kernel and discuss the notions of convergence which are appropriate to obtain convergence of gaps. Finally, in \cref{ss:assumptions}, we establish some notation which will be important in the proof of \cref{t:gaps} and we use the local semicircle law and eigenvalue rigidity to prove \cref{t:local}.
\subsection{Determinantal Kernels and the Eynard-Mehta Theorem}\label{ss:dpp} In this section, we introduce determinantal point processes, a topic that is covered in depth in resources such as~\cite{J06}. We will end the section by discussing the Eynard-Mehta theorem, which can be used to obtain the kernel of a determinantal point process when the measure has certain nice properties.

\begin{definition}[Point Process]\label{d:point} Let $\mathcal{X}$ be a locally compact separable topological space. A point configuration in $\mathcal{X}$ is a locally finite collection of points in $\mathcal{X}$. We use $\mathrm{Conf}(\mathcal{X})$ to denote the set of all point configurations in $\mathcal{X}$. We refer to a relatively compact Borel subset $A\subset \mathcal{X}$ as a window. For a window $A$ and $X\in \mathrm{Conf}(\mathcal{X})$, the number of points of $X$ in the window is $N_{A}(X)=|A\cap X|$. We equip $\mathrm{Conf}(\mathcal{X})$ with the Borel structure generated by functions $N_{A}$ for all windows $A$. 

A random point process on $\mathcal{X}$ is a probability measure on $\mathrm{Conf}(\mathcal{X})$. Given a random point process, we can define a sequence of measures $\{\rho_{n}\}_{n\in\mathbb{N}}$ on $\mathcal{X}^{n}$, where we refer to $\rho_{n}$ as the $n$-th correlation measure. For further discussion of conditions under which these correlation measures exist and are unique, see \cite{Bor11,Len73}. We typically have a natural measure $\mu$ on $\mathcal{X}$, called the reference measure.

If the correlation measure $\rho_n$ has density
$\rho_n(x_1,\ldots,x_n)$ with respect to $\mu^{\otimes n}$, then for
any compactly supported measurable function $F$,
$$
\mathbb{E}\left[
\sum_{\substack{p_{1},\ldots,p_{n}\in \Xi\\ \text{distinct}}}
F(p_{1},\ldots,p_{n})
\right]
=
\int_{\mathcal{X}^n}F(x_{1},\ldots,x_{n})
\rho_{n}(x_{1},\ldots,x_{n})\,
\mu(dx_{1})\cdots\mu(dx_{n}),$$
where $\Xi$ denotes a random element of $\mathrm{Conf}(\mathcal{X})$.
\end{definition}

\begin{definition}[Determinantal Point Process] Assume that we have a point process and a reference measure so that all $k$-point correlation measures have densities with respect to the reference measure. The process is called \textit{determinantal} if there exists a function $K:\mathcal{X}\times \mathcal{X}\to\mathbb{C}$, the determinantal kernel of the process, such that for all $n\in\mathbb{N}$, 
\begin{align*}
    \rho_{n}(x_{1},\cdots,x_{n})=\det\left[K(x_{i},x_{j})\right]_{i,j=1}^{n}.
\end{align*} 
\end{definition}

\begin{remark}[Uniqueness of the Kernel]\label{r:conj} Determinantal kernels are not unique. For instance, any conjugation 
\begin{align*}
    K'(x ,y) & = f(x)K(x,y)f(y)^{-1}
\end{align*} by a function $f$ which is non-vanishing on the state space will result in a kernel for the same point process. 
\end{remark}

The Eynard-Mehta theorem, which first appeared in \cite{EM98} and was extended to the setting which we consider by Borodin and Rains \cite{BR05}, is a general framework for obtaining a determinantal kernel for measures with density expressed as the product of determinants. It was used by \cite{BR05} to obtain the determinantal kernel for the Schur process, and in many later works, including \cite{Pet14,P14b} to obtain variants of that kernel for a variety of applications. There is also useful discussion of this theorem in \cite{Bor11}. 

To set up \cref{t:em}, we define $\mathrm{Conf}(\mathbb{R})=\prod_{n=m+1}^{m+L}\mathrm{Conf}_{n}(\mathbb{R})$ where $\mathrm{Conf}_{n}(\mathbb{R})$ is the space of ordered $n$-tuples $x^{(n)} = (x_{1}^{(n)},\ldots,x_{n}^{(n)})$ such that $x_{i}^{(n)}\in\mathbb{R}$ for all $i\in[[n]]$, $x_{i}^{(n)}\geq x_{j}^{(n)}$ when $i<j$ and with the formal $n+1$-st entry $x_{n+1}^{(n)}=\text{virt}$ which we refer to as a ``virtual particle,'' and which can be thought of as equal to $-\infty$. We also define notation for the associated Borel $\sigma$-algebra $\mathcal{F}=\mathcal{B}(\mathrm{Conf}(\mathbb{R}))$. We use $x^{(m)}$ to denote a fixed element of $\mathrm{Conf}_{m}(\mathbb{R})$. We define a collection of functions 
\begin{align}\label{e:thefunctions}
    \varphi(\cdot,\cdot):\mathbb{R}\times \mathbb{R}\to\mathbb{C}, & & 
    \varphi (\text{virt},\cdot):\mathbb{R}\to\mathbb{C},  
    \nonumber 
    \\
    \psi_{j}(\cdot | m+L):\mathbb{R}\to\mathbb{C}, & & j\in[[1,m+L]], 
    \\ 
    \upsilon_{i}(\cdot |m) : \{x^{(m)}_{j}\}_{j\in[[m]]}\to\mathbb{C}, & & i \in [[1,m]].\nonumber
\end{align} For any functions $f,g:\mathbb{R}^{2}\to\mathbb{R}$ and $h:\mathbb{R}\to\mathbb{R}$, we define the notation 
\begin{align*}
    (f \star g)(x,y):= \int_{-\infty}^{\infty}f(x,z)g(z,y)dz, & & (g \star h) (x) =\int_{-\infty}^{\infty}g(x,y)h(y)dy, & & (h\star g)(y) := \int_{-\infty}^{\infty}h(x)g(x,y)dx.
\end{align*} Similarly, for any $g:\mathbb{R}^{2}\to\mathbb{R}$ and $f:S\to\mathbb{R}$ where $S\subset\mathbb{R}$ satisfies $|S|\in\mathbb{N}$, we define
\begin{align*}
    (f\star g)(y) = \sum_{x\in S} f(x)g(x,y).
\end{align*}For fixed $m\in\mathbb{N}$, we make the following definitions for all $r,n\in [[m,m+L]]$,
\begin{align}\label{e:theconvolutions}\begin{split}
    \varphi^{(n,r)}(x,y) & : = \mathbf{1}_{n<r}(\underset{r-n \text{ times}}{\varphi \star \cdots \star \varphi})(x,y),
    \\ \left(\varphi\star\varphi^{(n,r)}\right)(\text{virt},x) & : = \mathbf{1}_{n\leq r}(\varphi(\text{virt},\cdot)\star \underset{r-n \text{ times}}{\varphi \star \cdots \star \varphi})(x),  \\
    \upsilon_{i}(x|r) & := (\upsilon_{i}(\cdot|m)\star\underset{r-m\text{ times}}{\varphi \star\cdots \star \varphi})(x) ,
    \\ \psi_{j} (x|r) & := (\underset{m+L-r\text{ times}}{\varphi \star \cdots \star \varphi } \star \psi_{j}(\cdot | m+L))(x) ,
    \end{split}
\end{align} where we use the convention that taking the $\star$-convolution with a function zero times leaves the original function unchanged.   
We assign a weight to any configuration $X\in\mathrm{Conf}(\mathbb{R})$ and fixed $x^{(m)}\in\mathrm{Conf}_{m}(\mathbb{R})$,
\begin{align}\label{e:weight}
    W(X,x^{(m)}) & \propto  \mathrm{det}\left[\upsilon_{i}(x_{j}^{(m)}|m)\right]_{i,j=1}^{m}\left(\prod_{k=m+1}^{m+L} \mathrm{det}\left[\varphi(x_{i}^{(k-1)},x_{j}^{(k)})\right]_{i,j=1}^{k}\right)\det\left[\psi_{i}(x_{j}^{(m+L)}|m+L)\right]_{i,j=1}^{m+L}.
\end{align} 
When we sum over all the weights associated with possible ending configurations $x_{j}^{(m+L)}$, we obtain the Gram matrix of the system. The Gram matrix of the weight function $W(X,x^{(m)})$ is
\begin{align*}
    G_{ij} & = \begin{cases} \upsilon_{i}(\cdot | m)\star \varphi^{(m,m+L)} \star \psi_{j}(\cdot |m+L) & i\in [[m]] \\  \left(\varphi\star\varphi^{(i,m+L)}\right)(\mathrm{virt},\cdot) 
    \star \psi_{j}(\cdot |m+L) & i \in [[m+1,m+L]] \end{cases}.
\end{align*}
The Eynard-Mehta theorem uses the Gram matrix to give an explicit determinantal kernel for any measure where the weight function takes a similar form to that of $W(X,x^{(m)})$, above (in particular, any weight function which can be expressed as the product of determinants). 

The version of the Eynard-Mehta theorem which we will apply here is slightly different from the versions which appear as \cite[Theorem 4.2, Theorem 4.4]{Bor11}. For this reason, we include a proof of this version of the theorem in \cref{a:em}, though the method of proof is the same as the $L$-ensemble approach taken by \cite{BR05}. Forrester and Nordenstam~\cite{FN09} also give a nice exposition of this approach, which we modify slightly to prove \cref{t:em}. 
\begin{theorem}[Eynard-Mehta]\label{t:em} The point process on $\mathrm{Conf}(\mathbb{R})$, started from a fixed $x^{(m)}\in \mathrm{Conf}_{m}(\mathbb{R})$, and with measure defined by a weight function $W(\cdot,x^{(m)})$ which takes the form \eqref{e:weight} is determinantal and, if the Gram matrix $G$ is invertible, then for all $x_{1},x_{2}\in\mathbb{R}$ and $n_{1},n_{2}\in[[m+1,m+L]]$, its correlation kernel is given by 
\begin{align*}
    K(n_{1},x_{1};n_{2},x_{2}) & =  -\varphi^{(n_{1},n_{2})}(x_{1},x_{2}) + \sum_{i=1}^{m}\sum_{j=1}^{m+L}[G^{-t}]_{ij}\upsilon_{i}(x_{2}|n_{2})\psi_{j}(x_{1}|n_{1})  \\ & + \sum_{i=m+1}^{n_{2}}\sum_{j=1}^{m+L}[G^{-t}]_{ij}\left(\varphi \star\varphi^{(i,n_{2})}\right)(\text{virt},x_{2}) \psi_{j}(x_{1}|n_{1}).
\end{align*}  
\end{theorem}
\begin{remark} 
In the case we consider in \cref{s:kernel} in the proof of \cref{t:kernel}, we will adjust the functions appearing in the role of \eqref{e:thefunctions} and \eqref{e:theconvolutions} so that the Gram matrix is the identity. Furthermore, in the case considered in \cref{t:kernel}, the expression for the kernel is invariant under any choice of $L$ such that $m+L\geq \max\{n_{1},n_{2}\}.$
\end{remark}
\subsection{Convergence of Kernels}\label{ss:jan}
We will need the following result about the convergence of determinantal point processes.
\begin{proposition}\label{p:unif2jan}
Let $\{\mathbb{P}_{n}\}_{n\in\mathbb{N}}$ be probability measures on $(\mathrm{Conf}_{n}(\mathbb{R}),\mathcal{F}_{n})$ corresponding to determinantal point processes with correlation kernels given by $\{K^{(n)}\}_{n\in\mathbb{N}}$, and let $\mathbb{P}$ be a probability measure on the space of locally finite point configurations on $\mathbb{R}$ corresponding to a determinantal point process with continuous correlation kernel $K$. Further suppose that these determinantal point processes share a locally finite reference measure $\mu$. Consider any function $\eta:\mathbb{N}\to\mathbb{R}$ such that $\eta(n)\geq 0$ and $\eta(n)=o(1)$. If for all compact $U\subset  \mathbb{R}$, there exists $C:=C(U)>0$ such that for all $n\in\mathbb{N}$, 
\begin{align*}
  \sup_{x_{1},x_{2}\in U}\big| K^{(n)}(x_{1},x_{2})- K(x_{1},x_{2})\big|\leq C\eta(n),
\end{align*} then $\mathbb{P}_{n}\to\mathbb{P}$ weakly as point processes. It further follows, denoting by $\{\rho^{(n)}\}_{n\in\mathbb{N}}$ and $\rho$ the correlation measures corresponding to the $\{K^{(n)}\}_{n\in\mathbb{N}}$ and $K$, that for all compact $U\subset \mathbb{R}$ and $k\in\mathbb{N}$, there exists $C:=C(U,k)>0$ such that for all $F\in C_{c}^{1}(\mathbb{R}^{k})$ with $\text{supp}(F)\subset U^{k}$, 
\begin{align*}
    \bigg| \int_{\mathbb{R}^{k}} F(x_{1},\cdots,x_{k})\rho^{(n)}(dx_{1},\cdots,dx_{k}) - \int_{\mathbb{R}^{k}}F(x_{1},\cdots,x_{k})\rho(dx_{1},\cdots,dx_{k})\bigg| \leq C\|F\|_{C^{1}}\eta(n).
\end{align*}
\end{proposition}
\begin{proof} The first claim follows immediately from \cite[Proposition 2.15]{dim26}. The second claim is an application of the multilinearity of the determinant.
\end{proof}
We also record a well-known convergence result for the  rising $\textsc{GUE}$ eigenvalue process kernel.
First we recall the $\textsc{GUE}$ kernel at fixed level $n$ (see \cite[Eqn. 6.15]{J18}, using probabilist's Hermite polynomials rather than physicist's, and setting $n_{1}=n_{2}=n$),
\begin{align*}
    K^{\textsc{GUE}} (n,x_{1};n,x_{2}) & :=   \frac{-1}{(2\pi i )^{2}} \int_{\Gamma}dz  \oint_{\gamma_{r}} dw  \frac{e^{(z-x_{2})^{2}/2 - (w-x_{1})^{2}/2}}{w-z}\left(\frac{z }{w }\right)^{n},
\end{align*} where $\gamma_{r}$ is a positively oriented circle of radius $r$ around the origin, and $\Gamma$ given by $s+it$ for a fixed $s\in\mathbb{R}$ such that $s>r$ and for all $t\in\mathbb{R}$; this contour is oriented in the positive imaginary direction. 
\begin{lemma}[Theorem 1.8, \cite{KSS15}] \label{l:gueunif} Let $\alpha\in (0,2)$. For any compact set $K\subset   \mathbb{R}$ there exist $C:=C(K,\alpha)>0,n_{0}:=n_{0}(K,\alpha)\in\mathbb{N}$ such that for all $ x,y\in K$, $n>n_{0},$ and $X\in (-2+\alpha,2-\alpha)$, 
\begin{align}\label{e:convratebound2}
        \big| \frac{1}{\rho_{sc}(X)\sqrt{n}}K^{\textsc{GUE}}\left(n, X\sqrt{n}+\frac{x }{\rho_{sc}(X)\sqrt{n}};n, X\sqrt{n}+\frac{y}{\rho_{sc}(X)\sqrt{n}}\right) - K^{\text{sine}}(x,y)\big|
        \leq Cn^{-1} 
    \end{align} where $K^{\text{sine}}(x,y)$ is as in \cref{d:extendedsine}. 
\end{lemma}

\subsection{Simple Spectrum, Local Limit Theorem}\label{ss:assumptions} In this section, we prove \cref{t:local} and also note that with high probability, the random matrices of \cref{d:wigner} have simple spectrum.
\begin{remark}\label{r:simple}
     \cref{d:wigner} implies that there exists $\mu=\mu(C,\varepsilon)>0$ such that $\sup_{z\in\mathbb C}\mathbb P(h_{ij}^{(n)}=z)\le 1-\mu.$ 
    Therefore, the conditions of \cite[Theorem 5.1]{TV17} hold, and there exists $n_{0}:=n_{0}(C,\varepsilon)$ such that for all $n>n_{0}$, the eigenvalues $x^{(n)}\in \mathrm{Conf}_{n}(\mathbb{R})$ are simple with probability at least $1-n^{-25}$. 
\end{remark}
We state an eigenvalue rigidity result, which is a consequence of the local semicircle law (see \cite[Theorem 1]{GNT18} for a version which applies to the definition of Wigner matrices with $4+\varepsilon$ moments). 
\begin{theorem}[Theorem 1.4, \cite{GNTT18}]\label{thm:rigid}
Let $M_n$ be an $n\times n$ Wigner matrix with eigenvalues $x_{j}^{(n)}$ for $j\in [[n]]$ and satisfying \cref{d:wigner} for a fixed $\varepsilon>0$. Then there exists a constant $C:=C(\varepsilon)>0$ such that 
\begin{align*}\mathbb{P}^{W}_{n}\left(\sup_{t\in \mb{R}}\bigg|\frac{1}{n}\sum_{j\in [n]}\mathbf{1}_{x_{j}^{(n)}\leq t\sqrt{n}} -\int_{-\infty}^{t} \rho_{\on{sc}}(x)dx\bigg|\leq  \frac{C\log^{2} n}{n}\right)\geq 1-n^{-25}.
\end{align*} 
\end{theorem}
Next, we prove a helpful consequence of this theorem which we will apply repeatedly when taking asymptotics of the kernel in \cref{t:kernel}. The version of eigenvalue rigidity we state in this paper (\cref{thm:rigid}) relies on the bounded $4+\varepsilon$ moment condition in \cref{d:wigner}. This is the origin of that constraint in \cref{d:wigner}.

\begin{proposition}\label{t:local} 
Let $\mathbb{P}_{m}^{W}$ be the law of Wigner eigenvalues on $(\mathrm{Conf}^{0}_{m},\mathcal{F}_{m})$ and let  $T:=T(m)$ such that $\log^{2}{(m)}\ll T(m)\ll m$. For any $X\in (-2,2),$ we define 
\begin{align*}
    \mu_{m}(du)  := \frac{1}{T}\sum_{r=1}^{m}\delta_{\frac{x_{r}^{(m)}/\sqrt{m}-X}{T/m}}(du), & & d_{m}(R)  :=F_{m,R}(x^{(m)})  = \frac{1}{m}\sum_{\frac{|X-x_{r}^{(m)}/\sqrt{m}|}{T/m}\geq R}\frac{1}{X - x_{r}^{(m)}/\sqrt{m}},
\end{align*} where $d_{m}(R)$ has law $\nu_{m,R}=\mathbb{P}_{m}^{W}\circ F_{m,R}^{-1}$. There exists a sequence of events $\Omega_{m}\in\mathcal{F}_{m}$ $\mathbb{P}_{m}^{W}(\Omega_{m})\geq 1-m^{-25}$ such that for all $\alpha\in (0,2)$
\begin{enumerate}
    \item There exists $C:=C(\alpha)>0$ such that for any compactly supported continuously differentiable function $f\in C^{1}_{c}(\mathbb{R})$, and for all $x^{(m)}\in\Omega_{m},$ $X\in (-2+\alpha,2-\alpha),$
    \begin{align*} 
        \bigg|\int_{\mathbb{R}} f(u)\mu_{m}(du) - \int_{\mathbb{R}}f(u)\rho_{sc}(X+uT/m)du \bigg|< C \textsc{TV}(f)\log^{2}{(m)}T^{-1},
    \end{align*} where $\textsc{TV}(f):=\int_{\mathbb{R}}|f'(u)|du.$
    \item There exists $C:=C(\alpha)>0$ such that for all $x^{(m)}\in\Omega_{m}$, $X\in (-2+\alpha,2-\alpha)$, and $R\geq 1$,
    \begin{align*}
        \bigg|d_{m}(R)-\int_{|v-X|\geq RT/m}\frac{\rho_{sc}(v)}{X-v}dv \bigg| < CR^{-1}\log^{2}{(m)}T^{-1}.
    \end{align*} Consequently, for all $R:=R(m)\geq 1$ satisfying $RT/m=o(1)$, the law $\nu_{m,R}$ converges weakly to $\delta_{X/2}$ as $m\to\infty$.
\end{enumerate}
\end{proposition}
The $X/2$ in part $(2)$ of \cref{t:local} appears because $\mathrm{p.v.}\left(\int_{-2}^{2}\rho_{sc}(u)(X-u)^{-1}du\right)=X/2$.
\begin{proof}
To show $(1)$, we will use the notation $\mu_{m}$ for the random empirical measure
\begin{align*}
    \mu_{m}(du) & :=\frac{1}{T}\sum_{r=1}^{m}\delta_{(x_{r}^{(m)}/\sqrt{m}-X)/(T/m)}(du).
\end{align*} Fix $a<b$ and consider 
\begin{align*}
    \int_{a}^{b}\mu_{m}(du) & = \frac{m}{T}\left(\frac{1}{m}\sum_{j\in [m]}\mathbf{1}_{(x_j^{(m)}/\sqrt{m}-X)/(T/m)\le b} - \frac{1}{m}\sum_{j\in [m]}\mathbf{1}_{(x_j^{(m)}/\sqrt{m}-X)/(T/m)\le a}\right).
\end{align*}
    As a result of \cref{thm:rigid}, we know there exists a constant $C>0$ such that 
\begin{align*}\mathbb{P}_{m}^{W}\left(\sup_{t\in \mb{R}}\bigg|\frac{1}{m}\sum_{j\in [m]}\mathbf{1}_{x_j^{(m)}\le t\sqrt{m}} -\int_{-\infty}^{t} \rho_{\on{sc}}(x)dx\bigg|\leq \frac{C(\log m)^2}{m}\right)\geq 1- m^{-25}.
\end{align*}
On this event, therefore, for any $a,b\in\mathbb{R}$ such that $a<b$, 
\begin{align*}
    \bigg| \int_{a}^{b}\mu_{m}(du)  - \int_{a}^{b} \rho_{\on{sc}}(X+uT/m)du\bigg|\leq 2C(\log{(m)})^{2}T^{-1}.
\end{align*} 
Since, for any continuously differentiable compactly supported function, $f\in C^{1}_{c}(\mathbb{R})$, we can uniformly approximate it by a sequence of simple functions, we conclude that for any $f\in C_{c}^{1}(\mathbb{R})$, on the event given by \cref{thm:rigid}, there exists a constant $C:=C(\alpha)>0$ such that for all $X\in (-2+\alpha,2-\alpha),$
$$\bigg| \int_{\mathbb{R}}f(u)\mu_{m}(du)-\int_{\mathbb{R}}f(u)\rho_{sc}(X+uT/m)du \bigg|\leq C\textsc{TV}(f)(\log{(m)})^{2}T^{-1}.$$ 
To show $(2)$, we again apply \cref{thm:rigid}  and the same reasoning to conclude that there exists $C:=C(\alpha)>0$ such that for all $m\in\mathbb{N}$, $x^{(m)}\in\Omega_{m}$, and $X\in (-2+\alpha,2-\alpha)$,
\begin{align*}
    \bigg|\frac{1}{m} \sum_{|X-x_{r}^{(m)}/\sqrt{m}|\geq RT/m}\frac{1}{X-x_{r}^{(m)}/\sqrt{m}}-\int_{|u-X|\geq RT/m}\frac{\rho_{sc}(u)}{X-u}du \bigg| \leq CR^{-1}\log^{2}{(m)}T^{-1}.
\end{align*} 
Furthermore, when $R:=R(m)$ is allowed to scale with $m$ and $RT/m=o(1)$, then 
\begin{align*}
    \lim_{m\to\infty} \int_{|u-X|\ge RT/m}\frac{\rho_{sc}(u)}{X-u} du
 = 
\mathrm{p.v.}\int_{-2}^{2}\frac{\rho_{sc}(u)}{X-u} du,
\end{align*} which evaluates to $X/2$. Since $\log^{2}{(m)}\ll T$, the error term is $o(1)$ and thus, $d_{m}(R)\to X/2$ as $m\to\infty$, and since $\mathbb{P}_{m}^{W}(\Omega_{m})\to 1$ as $m\to\infty$, the $\nu_{m,R}$ converge weakly to $\delta_{X/2}$ as $m\to\infty$. 
\end{proof}

\section{Derivation of the Determinantal Kernel} \label{s:kernel}

In this section we will prove \cref{t:kernel}. In \cref{s:emsetup} we set up the families of functions \eqref{e:thefunctions} which we will need to apply \cref{t:em} to this problem, and that section concludes with the application of \cref{t:em}. In \cref{s:middleterm} and \cref{s:finalterm} we re-sum the terms of that expression to obtain the contour integral expression in \cref{t:kernel}.

\subsection{Notation and Conventions}  
For $x^{(m)}\in\mathrm{Conf}_{m}(\mathbb{R})$, we will use the notation $x^{(m)}\setminus x_{i}^{(m)}$ to denote the element of $\mathrm{Conf}_{m-1}(\mathbb{R})$ which is comprised of all entries of $x^{(m)}$ except $x_{i}^{(m)}.$

We will work with the ``probabilist's'' Hermite polynomials, defined
\begin{align}\label{e:hermite}
    h_{n}(x) & := (-1)^{n} e^{\frac{x^{2}}{2}}\frac{d^{n}}{dx^{n}}e^{-\frac{x^{2}}{2}}.
\end{align} We will frequently use formulas for these polynomials as complex contour or line integrals of varying forms. We note
\begin{align}\label{e:hermite1}
    h_{n}(x) & = \frac{n!}{2\pi i}\int_{C} \frac{e^{tx-\frac{t^{2}}{2}}}{t^{n+1}}dt,
\end{align} where the contour $C$ is a closed, positively oriented contour of arbitrarily small radius which contains $0$. Similarly, we will often use the following representation of the Hermite polynomials as line integrals,
\begin{align}\label{e:hermite2}\begin{split}
    h_{n}(x) & \overset{(*)}{=} \frac{(-i)^{n}e^{x^{2}/2}}{\sqrt{2\pi}}\int_{-\infty}^{\infty}dt \cdot t^{n}\exp{\left(-\frac{t^{2} }{2}+ ixt \right)},
    \\  & \overset{(**)}{=}  \frac{i^{n}e^{x^{2}/2}}{\sqrt{2\pi}}\int_{-\infty}^{\infty}dt \cdot t^{n}\exp{\left(-\frac{t^{2} }{2}- ixt \right)}.
    \end{split}
\end{align}
We will also need the following expression for Hermite polynomials as a sum,
\begin{align}\label{e:hermitesum}
    h_{n}(x) & = \sum_{r=0}^{\lfloor n/2\rfloor }\frac{2^{-r}(-1)^{r}n!}{r!(n-2r)!}x^{n-2r},
\end{align} as well as the following formula expanding the expression $h_{n}(x+y)$,
\begin{align}\label{e:hermitesumexpand}
    h_{n}(x+y) & = \sum_{r=0}^{n}{n\choose r}y^{n-r}h_{r}(x).
\end{align}
We also use the notation $e_{r}$ to denote the $r$th elementary symmetric polynomial, which may take arbitrarily many arguments, but which is identically zero on $k<r$ arguments. 

We will also use the convention that the Fourier transform is $$\mathcal{F}\left\{f(y)\right\}(x)=\frac{1}{\sqrt{2\pi}}\int_{-\infty}^{\infty}f(y)e^{iyx}dy.$$
A standard property of the Fourier transform is that if $f\in C^{n}(\mathbb{R})$ and if all derivatives up to and including the $n$-th are $L^{1}(\mathbb{R})$, then
\begin{align}\label{e:ftprop}
    \mathcal{F}\left\{y^{n}f(y)\right\}(x)
    & = (-i)^{n}\left(\frac{d}{dx}\right)^{n}\mathcal{F}\left\{f(y)\right\}(x).
\end{align}

\subsection{Setting up the Eynard-Mehta Theorem}\label{s:emsetup}  
For an element $X=(x^{(m+1)},\cdots,x^{(m+L)})\in\mathrm{Conf}^{0}(\mathbb{R})$ (meaning that  $x^{(r)}\in\mathrm{Conf}_{r}^{0}(\mathbb{R})$ for each $r\in [[m+1,m+L]]$), we can write the weight function associated with the $\textsc{GUE}$ rising process started from a fixed configuration $x^{(m)}\in\mathrm{Conf}^{0}_{m}(\mathbb{R})$ in the form $W(X,x^{(m)})$ as in \eqref{e:weight}. 

The transition density between levels $n$ and $n+1$ of the rising $\textsc{GUE}$ process is proportional to 
$$p_{n,n+1}(x^{(n)};x^{(n+1)}) \propto \mathbf{1}_{x^{(n)}\prec x^{(n+1)}}\frac{\prod_{k<j}(x_{k}^{(n+1)}-x_{j}^{(n+1)})}{\prod_{k<j}(x_{k}^{(n)}-x_{j}^{(n)})} \exp{\left(-\sum_{k=1}^{n+1}\frac{(x_{k}^{(n+1)})^{2}}{2} + \sum_{k=1}^{n}\frac{(x_{k}^{(n)})^{2}}{2}\right)}.$$
Thus, we can find the distribution of $x^{(m+1)}\prec \cdots \prec x^{(m+L)}$ starting from a fixed configuration $x^{(m)}$ at level $m$, 
\begin{align*}
    \mathbb{P}(x^{(m+1)}\prec \cdots \prec x^{(m+L)}|x^{(m)}) \propto \mathbf{1}_{x^{(m)}\prec\cdots\prec x^{(m+L)}} \frac{\prod_{i<j}(x_{i}^{(m+L)}-x_{j}^{(m+L)})\exp{\left(-\sum_{i=1}^{m+L}\frac{(x_{i}^{(m+L)})^{2}}{2}\right)}}{\prod_{i<j}(x_{i}^{(m)}-x_{j}^{(m)})\exp{\left(-\sum_{i=1}^{m}\frac{(x_{i}^{(m)})^{2}}{2}\right)}}.
\end{align*}
We define the interlacing function 
\begin{align*}
    \varphi(x,y)  := \mathbf{1}_{x\leq y}, 
    & & \varphi(\text{virt},y)  := 1,
\end{align*}
and note that 
\begin{align*}
   \mathbf{1}_{x^{(m)}\prec\cdots\prec x^{(m+L)}} & = \prod_{i=1}^{L} \det\left[\varphi(x_{j}^{(m+i-1)},x_{k}^{(m+i)})\right]_{j,k=1}^{m+i}.
\end{align*}
We likewise see that 
\begin{align*}
     \exp{\left(-\sum_{i=1}^{m+L}\frac{(x_{i}^{(m+L)})^{2}}{2}\right)}\prod_{k<j}(x_{k}^{(m+L)}-x_{j}^{(m+L)})
    &
    = \det\left[\psi_{k}(x_{j}^{(m+L)}|m+L)\right]_{k,j=1}^{m+L},
\end{align*} for a function $\psi(\cdot|m+L)$ defined
\begin{align*}
    \psi_{k}(x|m+L) & := h_{m+L-k}(x)e^{-\frac{x^{2}}{2}},
\end{align*} where $h_{k}(x)$ is the $k$-th probabilist's Hermite polynomial \eqref{e:hermite}.

We must still define the initial level functions $\upsilon_{k}(\cdot|m)$. This involves studying the inverse of the Vandermonde matrix. We note that the reciprocal of the Vandermonde determinant is the determinant of the inverse matrix, 
\begin{align*}
    \prod_{k<j}(x_{k}^{(m)}-x_{j}^{(m)})^{-1}  & = \frac{1}{\det[V(x_{1}^{(m)},\ldots,x_{m}^{(m)})]},
\end{align*}
where 
\begin{align*}
    V(x^{(m)}) : =  \left[\left(x_{j}^{(m)}\right)^{m-k}\right]_{kj}.
\end{align*}
The entries of the inverse of the Vandermonde matrix are 
    \begin{align*}
        [V(x^{(m)})^{-1}]_{k,j} & = \frac{(-1)^{j-1}e_{j-1}(x^{(m)}\setminus x_{k}^{(m)})}{\prod_{r\neq k}(x_{k}^{(m)}-x_{r}^{(m)})} = \frac{(-1)^{j-1}e_{j-1}(x^{(m)}\setminus x_{k}^{(m)})}{2\pi i}\int_{\mathfrak{c}(x_{k}^{(m)})}dz\frac{1}{\prod_{r=1}^{m}(z-x_{r}^{(m)})}.
    \end{align*} 
This formula is well-known, though it is more typical to represent it as a double contour integral \cite[Equation 4.14]{Pet14}. The form stated above is better suited to the computations in this article. Since we are using the probabilists' Hermite polynomials, $V(x^{(m)})$ has the same determinant as $h(x^{(m)}):=[h_{m-k}(x_{j}^{(m)})]_{kj}$ (the change of basis has determinant $1$). We note that $h(x^{(m)}) = R\cdot V(x^{(m)})$ for $R$ defined by
\begin{align*}
    R & := \left[\frac{(m-k)!(-1)^{\frac{j-k}{2}}2^{-\frac{j-k}{2}}}{((j-k)/2)!(m-j)!}\mathbf{1}_{j-k \text{ is even}, j\geq k}\right]_{kj}.
\end{align*}
The inverse is given by
\begin{align*}
    R^{-1} & = \left[ \frac{(m-k)!2^{-\frac{j-k}{2}}}{((j-k)/2)!(m-j)!}\mathbf{1}_{j-k\text{ is even}, j\geq k} \right]_{kj}.
\end{align*}
Therefore, $[h(x^{(m)})^{-1}]_{kj} = [ V(x^{(m)})^{-1}\cdot R^{-1}]_{kj},$
and the full expression for the entries of this matrix is
\begin{align*}
     \frac{1}{2\pi i}\int_{c(x_{k}^{(m)})}dz \frac{1}{\prod_{r=1}^{m}(z-x_{r}^{(m)})} \sum_{l=1}^{m} (-1)^{l-1}e_{l-1}(x^{(m)}\setminus x_{k}^{(m)}) \frac{(m-l)!2^{-\frac{j-l}{2}}}{((j-l)/2)!(m-j)!}\mathbf{1}_{j-l\text{ is even}, j\geq l}.
\end{align*}
We define the last unspecified function from \eqref{e:thefunctions} as
\begin{align*}
    \upsilon_{k}(x_{j}^{(m)}|m)  := [h(x^{(m)})^{-T}]_{kj}e^{(x_{j}^{(m)})^{2}/2}, & & \text{for all } j\in [[m]].
\end{align*}
With the notation we have now established, we can write 
\begin{multline}\label{e:meas}
    \mathbb{P}(x^{(m+1)}\prec \cdots \prec x^{(m+L)}|x^{(m)}) \\ \propto \det\left[\upsilon_{k}(x_{j}^{(m)}|m)\right]_{k,j=1}^{m} \det\left[\psi_{k}(x_{j}^{(m+L)}|m+L)\right]_{k,j=1}^{m+L} 
    \prod_{r=1}^{L}\det\left[\varphi(x_{k}^{(m+r-1)},x_{j}^{(m+r)})\right]_{k,j=1}^{m+r}.
\end{multline}
If we were to apply \cref{t:em} directly to \eqref{e:meas}, the Gram matrix would have entries which diverge due to the fact that taking convolutions with $\varphi$ to generate all elements of \eqref{e:thefunctions} would in some cases result in functions which are not integrable. A work of Metcalfe \cite{Met13} which studies the rising $\textsc{GUE}$ eigenvalue process conditioned on the top row, avoids this issue by truncating the state space (requiring $x_{j}^{(r)}$ to lie in a bounded interval). We will instead follow the approach of \cite{FN09} and deal with this issue by modifying the functions $\varphi$.
\begin{lemma}\label{l:detshift}
For any choice of the function $\kappa(x)$, the following determinants are equal
\begin{align}\label{e:deteq}
     \det\begin{bmatrix}[\varphi(x_{k}^{(r-1)},x_{j}^{(r)})]_{k\in[[r-1]],j\in[[r]]}\\ [1]_{k=1,j\in[[r]]}\end{bmatrix} 
    = 
    \det\begin{bmatrix}[\varphi(x_k^{(r-1)} , x_j^{(r)}) -\kappa(x_{k}^{(r-1)})]_{k\in[[r-1]],j\in [[r]]}\\ [1]_{k=1,j\in[[r]]}\end{bmatrix}.
\end{align} 
\end{lemma}
\begin{proof}
% We will refer to the matrix on the left-hand side as $\Phi^{(r)}$. The determinant on the right-hand side is given by $\det[\Phi^{(r)}+v\mathbf{1}^{T}]$ where $\mathbf{1}^{T}$ is a $1\times r$ row vector with all entries equal to $1$ and where $v$ is a $r\times 1$ column vector with $v_{k}=-\kappa(x_{k}^{(r-1)})$ for $k\in[[r-1]]$ and $v_{r}=0$. By the rank $1$ update formula, this determinant is equal to $\det[\Phi^{(r)}](1+\mathbf{1}^{T}(\Phi^{(r)})^{-1}v)$. We note that $(\Phi^{(r)})^{-1}_{sj} =  M_{sj}/\det[\Phi^{(r)}]$, where $M_{sj}$ is the $(s,j)$-th minor of $\Phi^{(r)}$. Furthermore, we note that when $s\neq r$
% \begin{multline*}\det\begin{bmatrix}[\varphi(x_{\ell}^{(r-1)},x_{k}^{(r)})]_{\ell\neq s, k\neq j} \\ [1]_{k\neq j}\end{bmatrix} \\ = (-1)^{r-1}\left(\sum_{k=1}^{j-1}(-1)^{k-1}\det[\varphi(x_{\ell}^{(r-1)},x_{r}^{(r)})]_{\ell\neq s ,r\neq k,j} + \sum_{k=j+1}^{r}(-1)^{k}\det[\varphi(x_{\ell}^{(r-1)},x_{r}^{(r)})]_{\ell\neq s ,r\neq k,j}\right).
% \end{multline*}

% In particular, when we sum this expression over $j\in[[r]]$, it is zero. Thus, 
% \begin{align*}
%     \mathbf{1}^{T}(\Phi^{(r)})^{-1}v & = -\frac{1}{\det[\Phi^{(r)}]} \sum_{j=1}^{r}\sum_{s=1}^{r-1}\kappa(x_{s}^{(r-1)})M_{sj} = 0,
% \end{align*} and we have verified \eqref{e:deteq}. 
The last row of the matrix on the left-hand side of \eqref{e:deteq} is comprised of entries of the form $\varphi(\text{virt},x_{k}^{(r)})=1$. For each $k\in [[r-1]]$, after subtracting $\kappa(x^{(r-1)}_k)$ times the final row from the $k$-th row, the determinant is unchanged and the $(k,j)$-entry in the first $r-1$ rows takes the form \begin{align*}\phi(x^{(r-1)}_k,x^{(r)}_j)-\kappa(x^{(r-1)}_k),
\end{align*}
while the final row remains the same. 
\end{proof}
We can expand $\varphi(x,y)$ in the basis of Hermite polynomials when $x,y\neq \text{virt}$ to obtain
\begin{align*}
    \varphi(x,y) & = \sum_{k=0}^{\infty} \frac{h_{k}(y)}{\sqrt{2\pi}k!}\int_{x}^{\infty}e^{-z^{2}/2}h_{k}(z)dz.
\end{align*}
Separating the $k=0$ term from the rest, which is equal to 
\begin{align*}
    & = \frac{1}{\sqrt{2\pi}}\int_{x}^{\infty}e^{-z^{2}/2}dz + e^{-x^{2}/2}\sum_{k=0}^{\infty}\frac{1}{\sqrt{2\pi}(k+1)!}h_{k}(x)h_{k+1}(y).
\end{align*}
Thus, setting 
\begin{align*}
    \kappa(x) & :=\frac{1}{\sqrt{2\pi}}\int_{x}^{\infty}e^{-z^{2}/2}dz,
\end{align*} and subtracting off the $\kappa(x)$ term, we define
\begin{align*}
    \widetilde{\varphi}(x,y) 
     := e^{-x^{2}/2}\sum_{k=0}^{\infty}\frac{h_{k}(x)h_{k+1}(y)}{\sqrt{2\pi}(k+1)!},
     & & \widetilde{\varphi}(\text{virt},y)  := \frac{1}{\sqrt{2\pi}}h_{0}(y).
\end{align*} The replacement of the virtual row $1$ by $1/\sqrt{2\pi}$ changes
the total weight by a configuration-independent constant, which is absorbed
into the normalization.
\cref{l:detshift} tells us that we can equivalently use $\widetilde{\varphi}$ in place of $\varphi$ in \eqref{e:meas}. Before we can apply the Eynard-Mehta theorem, we need to compute the Gram matrix and several other helpful functions \eqref{e:theconvolutions}.

\begin{lemma}\label{l:allfunctions} We compute the following helpful expressions:
\begin{enumerate}
    \item For $a,b\in [[m,m+L]]$,
\begin{align*}
        \varphi^{(a,b)}(x,y) & =\mathbf{1}_{a<b}\mathbf{1}_{x\leq y} \frac{(y-x)^{b-a-1}}{(b-a-1)!}.
    \end{align*}
    \item For $a,b\in [[m,m+L]]$,
\begin{align*}
        \widetilde{\varphi}^{(a,b)}(x,y)  = \mathbf{1}_{a<b} \sum_{k=0}^{\infty}\frac{e^{-x^{2}/2}h_{k}(x)h_{k+b-a}(y)}{\sqrt{2\pi}(k+b-a)!},
        & & \widetilde{\varphi}^{(a,b)}(\text{virt},y) 
        = \mathbf{1}_{a\leq b} \frac{1}{\sqrt{2\pi} (b-a)!}h_{b-a}(y).
    \end{align*}
    \item For $n\in [[m+1,m+L]]$ and $k\in [[1,m+L]]$,
    \begin{align*}\psi_{k}(x|n) & = \begin{cases}
        \frac{1}{(k-n-1)!}\int_{x}^{\infty}(y-x)^{k-n-1}e^{-y^{2}/2}dy & n <k
        \\ h_{n-k}(x)e^{-x^{2}/2} & n \geq k
    \end{cases}.
    \end{align*}  
    \item  For $n\in [[m+1,m+L]]$ and $k\in [[1,m+L]]$,
    \begin{align*}\widetilde{\psi}_{k}(x|n) & = \begin{cases}
       h_{n-k}(x)e^{-x^2/2}  & n\geq k
        \\ 0 & n <k
    \end{cases}. 
    \end{align*} In particular, when $n\geq k$, $\widetilde{\psi}_{k}(x|n)=\psi_{k}(x|n)$.
    \item For $n\in[[m+1,m+L]]$ and $k\in [[1,m]]$, 
    \begin{align*}
     \upsilon_{k}(x|n)
     & = \sum_{x_{j}^{(m)}\leq x} \frac{1}{2\pi i}\int_{c(x_{j}^{(m)})}dz \frac{(x-z)^{n -m-1}e^{z^{2}/2}}{(n -m-1)!\prod_{r=1}^{m}(z-x_{r}^{(m)})} 
     \\ & \cdot \sum_{r=1}^{m} (-1)^{r-1}e_{r-1}(x^{(m)}\setminus x_{j}^{(m)}) \frac{(m-r)!2^{-\frac{k-r}{2}}}{((k-r)/2)!(m-k)!}\mathbf{1}_{k-r\text{ is even}, k\geq r}.
    \end{align*}
    Similarly, when $n=m$ we define $\upsilon_{k}(x|m)$ as before, with $x\in \{x_{j}^{(m)}\}_{j\in [[m]]}$.
    \end{enumerate}
\end{lemma}
\begin{proof} 
    To show part $(1)$, we see that
    \begin{align*}
        \varphi(x,z_{a+1})\star\cdots \star \varphi(z_{b-1},y) & = \int_{\mathbb{R}}\cdots\int_{\mathbb{R}} \mathbf{1}_{x\leq z_{a+1}}\mathbf{1}_{z_{a+1}\leq z_{2}}\cdots \mathbf{1}_{z_{b-1}\leq y}  dz_{a+1}\cdots dz_{b-1}
         = \frac{(y-x)^{b-a-1}}{(b-a-1)!}\mathbf{1}_{x\leq y},
    \end{align*} where we have implicitly assumed that $a<b$. 
    To show $(2)$, we note \begin{align*}
    \widetilde{\varphi}^{(a,b)}(x,y) & = \mathbf{1}_{a<b}\underset{b-a\text{ times}}{(\widetilde{\varphi}\star\cdots \star \widetilde{\varphi})} (x,y) = \sum_{k=0}^{\infty}\frac{e^{-x^{2}/2}h_{k}(x)h_{k+b-a}(y)}{\sqrt{2\pi}(k+b-a)!}.
\end{align*}
    To show $(3)$, we apply the expression from $(1)$ to obtain \begin{align}\label{e:psikbig}
     \psi_{k}(x|n) & = \frac{1}{(m+L-n-1)!}\int_{x}^{\infty}(y-x)^{m+L-n-1}h_{m+L-k}(y)e^{-y^{2}/2}dy,
\end{align}
and integrate by parts. In fact, we note that the expression above 
holds for any $L\geq k-m$, a fact which we will apply later on in the construction of the kernel. 
    To show $(4)$, we again simply integrate by parts, using the definition of $\psi_{k}(x|m+L)$. 
    To show $(5)$, we apply $(1)$ and the definition of $\upsilon_{k}(x|m)$. 
\end{proof}
We will apply \cref{t:em} using
\begin{align*}
   \upsilon_{k}(y|m)\star\widetilde{\varphi}^{(m,n)}(y,x), & & k\in [[m]],
    \\ 
    \widetilde{\varphi}\star\widetilde{\varphi}^{(k,n)}(\text{virt},x), & & k\in [[m+1,m+L]],
   \\ 
   \widetilde{\varphi}^{(n,m+L)}(x,y)\star\widetilde{\psi}_{k}(y|m+L), & & k\in [[m+L]],
\end{align*}  
as the functions \eqref{e:theconvolutions}. The Gram matrix associated with these functions is 
\begin{align*}
    G_{kj} & = \begin{cases}
        \upsilon_{k}(\cdot |m)\star\widetilde{\varphi}^{(m,m+L)}\star\widetilde{\psi}_{j}(\cdot |m+L) & k\in [[1,m]]
        \\ (\widetilde{\varphi}\star\widetilde{\varphi}^{(k,m+L)})(\text{virt},\cdot)\star\widetilde{\psi}_{j}(\cdot |m+L) & k\in [[m+1,m+L]]
    \end{cases}.
\end{align*}
\begin{lemma}\label{l:gid}
    $G$ is the $(m+L)\times (m+L)$ identity matrix. 
\end{lemma}
\begin{proof} When $k\in [[m+1,m+L]]$ this reduces to the following calculation \begin{align}\label{e:delta}
     (\widetilde{\varphi} \star\widetilde{\varphi}^{(k,m+L)})(\text{virt},\cdot)\star\widetilde{\psi}_{j}(\cdot |m+L) & = \delta_{k,j}.
\end{align} When $k\in [[1,m]]$, we see that $G_{kj}  = \upsilon_{k}(\cdot |m)\star\widetilde{\psi}_{j}(\cdot |m)$, which is simply
\begin{align*}
    \sum_{r=1}^{m}[h(x^{(m)})^{-T}]_{kr}h_{m-j}(x_{r}^{(m)}) & =  \sum_{r=1}^{m}[h(x^{(m)})^{-1}]_{rk}[h(x^{(m)})]_{jr}  = \delta_{k,j}.
\end{align*}
\end{proof}
With these results in hand, we can now prove an explicit form for the kernel in \cref{t:kernel}, though more work is required to obtain the form of the kernel which appears in that theorem. 
\begin{lemma}\label{l:kernelsum}
    The determinantal kernel of the rising $\textsc{GUE}$ eigenvalue process started at a fixed configuration $x^{(m)}\in\mathrm{Conf}^{0}_{m}(\mathbb{R})$ at level $m$ is given, for all $n_{1},n_{2}>m$, by 
    \begin{align}\label{e:kernel2}
\begin{split}
    K_{x^{(m)}}^{\textsc{GUE}}(n_{1},x_{1};n_{2},x_{2}) & 
    = -\varphi^{(n_{1},n_{2})}(x_{1},x_{2})  
    \\ & + \sum_{\ell=1}^{m}\left(\upsilon_{\ell}(y|m)\star\varphi^{(m,n_{2})}(y,x_{2})\right)h_{n_{1}-\ell}(x_{1})e^{-x_{1}^{2}/2}  
    \\ & + \sum_{k=m+1}^{n_{2}}\frac{h_{n_{2}-k}(x_{2})\psi_{k}(x_{1}|n_{1})}{\sqrt{2\pi}(n_{2}-k)!}
    \\ & - \sum_{\ell=1}^{m}\upsilon_{\ell}(y|m)\star\sum_{k=m+1}^{n_{2}}\frac{h_{n_{2}-k}(x_{2})\psi_{k}(y|m)}{\sqrt{2\pi}(n_{2}-k)!}h_{n_{1}-\ell}(x_{1})e^{-x_{1}^{2}/2}.
    \end{split}
\end{align} 
\end{lemma}
\begin{remark}
    From the expression in \eqref{e:kernel2} it is easy to see that when $m=0$, this reduces to the standard rising $\textsc{GUE}$ eigenvalue process kernel studied by Johansson and Nordenstam \cite{JN06}. All terms vanish except for the first and third, which are exactly the terms which appear in their expression.
\end{remark}
\begin{proof}[Proof of \cref{l:kernelsum}]
\cref{t:em}, along with the fact that the Gram matrix is the identity (\cref{l:gid}) implies that the determinantal kernel of the measure \cref{e:meas} is 
\begin{align}\label{e:kernel}
 -\widetilde{\varphi}^{(n_{1},n_{2})}(x_{1},x_{2}) + \sum_{k=1}^{m}\upsilon_{k}(x_{2}|n_{2})\widetilde{\psi}_{k}(x_{1}|n_{1})   + \sum_{k=m+1}^{n_{2}}(\widetilde{\varphi}\star\widetilde{\varphi}^{(k,n_{2})})(\text{virt},x_{2})\widetilde{\psi}_{k}(x_{1}|n_{1}).
\end{align} 

 Substituting in the expression for  $\widetilde{\psi}_{k}(x_{1}|n_{1})$ in the first sum and noting that $\widetilde{\psi}_{k}(x_{1}|n_{1})=0$ when $k>n_{1}$ and $\widetilde{\psi}_{k}(x_{1}|n_{1})=\psi_{k}(x_{1}|n_{1})$ when $k\leq n_{1}$, we may write
\begin{align*}
& -\widetilde{\varphi}^{(n_{1},n_{2})}(x_{1},x_{2}) + \sum_{k=1}^{m}\upsilon_{k}(x_{2}|n_{2})h_{n_{1}-k}(x_{1})e^{-x_{1}^{2}/2}   + \sum_{k=m+1}^{\min\{n_{1},n_{2}\}}\frac{h_{n_{2}-k}(x_{2})\psi_{k}(x_{1}|n_{1})}{\sqrt{2\pi}(n_{2}-k)!}.
\end{align*}
Following the approach of \cite{FN09}, we will demonstrate that this kernel, in terms of the interlacing function, $\widetilde{\varphi}$ can be expressed in terms of the original interlacing function $\varphi$. By taking the Hermite expansion of $\varphi^{(n_{1},n_{2})}(x_{1},x_{2})$, we find
\begin{align*}
    -\widetilde{\varphi}^{(n_{1},n_{2})}(x_{1},x_{2}) & = -\varphi^{(n_{1},n_{2})}(x_{1},x_{2}) + \mathbf{1}_{n_{1}<n_{2}}\sum_{k=n_{1}+1}^{n_{2}}\frac{h_{n_{2}-k}(x_{2})\psi_{k}(x_{1}|n_{1})}{\sqrt{2\pi}(n_{2}-k)!}.
\end{align*}
We plug this into the definition of $\upsilon_{\ell}(x_{2}|n_{2})$ function to obtain
\begin{align*}
    \upsilon_{\ell}(x_{2}|n_{2}) & = \upsilon_{\ell}(y|m)\star\varphi^{(m,n_{2})}(y,x_{2}) - \upsilon_{\ell}(y|m)\star\sum_{k=m+1}^{n_{2}}\frac{h_{n_{2}-k}(x_{2})\psi_{k}(y|m)}{\sqrt{2\pi}(n_{2}-k)!}.
\end{align*}
The full expression for the kernel \eqref{e:kernel} may therefore be written as
\begin{align*} 
\begin{split}
    K_{x^{(m)}}^{\textsc{GUE}}(n_{1},x_{1};n_{2},x_{2}) & 
    = -\varphi^{(n_{1},n_{2})}(x_{1},x_{2})  
    \\ & + \sum_{\ell=1}^{m}\left(\upsilon_{\ell}(y|m)\star\varphi^{(m,n_{2})}(y,x_{2})\right)h_{n_{1}-\ell}(x_{1})e^{-x_{1}^{2}/2}  
    \\ & + \sum_{k=m+1}^{n_{2}}\frac{h_{n_{2}-k}(x_{2})\psi_{k}(x_{1}|n_{1})}{\sqrt{2\pi}(n_{2}-k)!}
    \\ & - \sum_{\ell=1}^{m}\upsilon_{\ell}(y|m)\star\sum_{k=m+1}^{n_{2}}\frac{h_{n_{2}-k}(x_{2})\psi_{k}(y|m)}{\sqrt{2\pi}(n_{2}-k)!}h_{n_{1}-\ell}(x_{1})e^{-x_{1}^{2}/2}.
    \end{split}
\end{align*}
\end{proof} The next two lemmas show that it is possible to re-sum the terms into a double contour integral expression which is amenable to asymptotic analysis. 
\begin{lemma}\label{l:termsolve1} For any $m\in\mathbb{N}$ and $n_{1},n_{2}>m$, $x_{1},x_{2}\in \mathbb{R}$, and $x^{(m)}\in\mathrm{Conf}^{0}_{m}(\mathbb{R})$,
    \begin{multline*}
        \sum_{\ell=1}^{m}\upsilon_{\ell}(x_{2}|n_{2})h_{n_{1}-\ell}(x_{1})e^{-x_{1}^{2}/2} 
        \\ = \frac{(n_{1}-m)!}{(n_{2}-m-1)!}\frac{(-1)^{n_{2}-n_{1}-1}}{(2\pi i)^{2}}\int_{c(x_{j}^{(m)}\leq x_{2})}dz \int_{c(x_{1})}dw\frac{(z-x_{2})^{n_{2}-m-1}}{(w-x_{1})^{n_{1}-m+1}} \frac{e^{z^{2}/2-w^{2}/2}}{w-z} \prod_{r=1}^{m}\frac{w-x_{r}^{(m)}}{z-x_{r}^{(m)}}.
    \end{multline*}
    In this expression, $c(x_{j}^{(m)}\leq x_{2})$ is the positively oriented contour containing all of the points in $x^{(m)}$ which are less than or equal to $x_{2}$, and $c(x_{1})$ is the positively oriented contour containing $x_{1}$, and only containing $x_{2}$ or $x_{j}^{(m)}$ for some $j\in[[m]]$ if they are equal to $x_{1}$. 
\end{lemma}
\begin{lemma}\label{l:termsolve2} For any $d\in\mathbb{R}$ less than $x_{1},x_{2}$, and all elements of $x^{(m)}$, and for any $m\in\mathbb{N}$ and $n_{1},n_{2}>m$, $x_{1},x_{2}\in\mathbb{R},$ and $x^{(m)}\in\mathrm{Conf}^{0}_{m}(\mathbb{R})$,
    \begin{multline*}
        \sum_{k=m+1}^{n_{2}}\frac{h_{n_{2}-k}(x_{2})\psi_{k}(x_{1}|n_{1})}{\sqrt{2\pi}(n_{2}-k)!}
     - \sum_{\ell=1}^{m}\upsilon_{\ell}(y|m)\star\sum_{k=m+1}^{n_{2}}\frac{h_{n_{2}-k}(x_{2})\psi_{k}(y|m)}{\sqrt{2\pi}(n_{2}-k)!}h_{n_{1}-\ell}(x_{1})e^{-x_{1}^{2}/2}
    \\  =\frac{(-1)^{n_{1}-n_{2}-1}}{(2\pi i)^{2}}\frac{(n_{1}-m)!}{(n_{2}-m-1)!} \int\limits_{d-i\infty}^{d+i\infty} d z\int_{c(x_{1})}dw \frac{(z-x_{2})^{n_{2}-m-1}}{(w-x_{1})^{n_{1}-m+1}}
    \frac{e^{z^{2}/2-w^{2}/2}}{w-z}\prod_{r=1}^{m}\frac{w-x_{r}^{(m)}}{z-x_{r}^{(m)}}, 
    \end{multline*} 
    where $c(x_{1})$ is the positively oriented closed contour containing $x_{1}$, and only containing $x_{2}$ or $x_{j}^{(m)}$ for $j\in [[m]]$ if they are equal to $x_{1}$, and $d<\inf_{w\in c(x_{1})}\Re(w)$.
\end{lemma}
The proofs of these lemmas appear in \cref{s:middleterm} and \cref{s:finalterm}. Before proceeding to those proofs, we will prove \cref{t:kernel}. 
\begin{proof}[Proof of \cref{t:kernel}] We sum the terms involved in the expression \cref{e:kernel2}, using the expressions from \cref{l:termsolve1} and \cref{l:termsolve2},
\begin{align*}
    & -\mathbf{1}_{n_{1}<n_{2}}\mathbf{1}_{x_{1}\leq x_{2}} \frac{(x_{2}-x_{1})^{n_{2}-n_{1}-1}}{(n_{2}-n_{1}-1)!} 
    \\ & + \frac{(n_{1}-m)!}{(n_{2}-m-1)!}\frac{(-1)^{n_{2}-n_{1}-1}}{(2\pi i)^{2}}\int_{c(x_{j}^{(m)}\leq x_{2})}dz \int_{c(x_{1})}dw\frac{(z-x_{2})^{n_{2}-m-1}}{(w-x_{1})^{n_{1}-m+1}} \frac{e^{z^{2}/2-w^{2}/2}}{w-z} \prod_{r=1}^{m}\frac{w-x_{r}^{(m)}}{z-x_{r}^{(m)}}
    \\ & + \frac{(n_{1}-m)!}{(n_{2}-m-1)!} \frac{(-1)^{n_{2}-n_{1}-1}}{(2\pi i)^{2}}\int\limits_{d-i\infty}^{d+i\infty} d z\int_{c(x_{1})}dw \frac{(z-x_{2})^{n_{2}-m-1}}{(w-x_{1})^{n_{1}-m+1}}
    \frac{e^{z^{2}/2-w^{2}/2}}{w-z}\prod_{r=1}^{m}\frac{w-x_{r}^{(m)}}{z-x_{r}^{(m)}},
\end{align*} 
where $d\in\mathbb{R}$ is less than $x_{1},x_{2}$ and all elements of $x^{(m)}$; $c(x_{j}^{(m)}\leq x_{2})$ is the positively oriented contour containing all of the points in $x^{(m)}$ which are less than or equal to $x_{2}$; and $c(x_{1})$ is the positively oriented contour containing $x_{1}$, and only containing $x_{2}$ or $x_{j}^{(m)}$ for $j\in[[m]]$ if they are equal to $x_{1}$.  
We move the contour at $d$ over to the right until it sits at or slightly to the right of $x_{2}$ (without crossing any additional points in $x^{(m)}$ after it crosses $x_{2}$), which produces factors which exactly cancel the first two terms. This is because the deformation crosses all poles for which $x_{j}^{(m)}\leq x_{2}$, which cancel with the second term, and also crosses the $w$ contour if $x_{1}\leq x_{2}$, which cancels with the first term.
What remains is
\begin{align*}
    \frac{(n_{1}-m)!}{(n_{2}-m-1)!}
    \frac{(-1)^{n_{2}-n_{1}-1}}{(2\pi i)^{2}}
    \int\limits_{x_{2}+\delta-i\infty}^{x_{2}+\delta +i\infty} d z\oint\limits_{c(x_{1})}dw
    \frac{(z-x_{2})^{n_{2}-m-1}}{(w-x_{1})^{n_{1}-m+1}}
    \frac{e^{z^{2}/2-w^{2}/2}}{w-z}
    \prod_{r=1}^{m}
    \frac{w-x_{r}^{(m)}}{z-x_{r}^{(m)}},
\end{align*} for $\delta>0$ such that if $x_{1}>x_{2},$ then $x_{1}>x_{2}+\delta$, and if $x_{j}^{(m)}>x_{2}$ for any $j\in[[m]]$, then $x_{j}^{(m)}>x_{2}+\delta$. 
Finally, we make a gauge transformation to remove the factor of $(-1)^{n_{2}-n_{1}}$ (this function is non-vanishing on the state space). 
\end{proof}

\subsection{Integral Expression for the Second Term}\label{s:middleterm} 

\begin{proof}[Proof of \cref{l:termsolve1}]
We begin by plugging in the expression for $\upsilon_{\ell}(x_{2}|n_{2})$ from \cref{l:allfunctions},
\begin{align}\label{e:first}\begin{split}
    \sum_{\ell=1}^{m}\upsilon_{\ell}(x_{2}|n_{2})h_{n_{1}-\ell}(x_{1})e^{-x_{1}^{2}/2}
     & = \sum_{\ell=1}^{m} h_{n_{1}-\ell}(x_{1})e^{-x_{1}^{2}/2}\sum_{x_{j}^{(m)}\leq x_{2}} \frac{1}{2\pi i}\int_{c(x_{j}^{(m)})}dz \frac{(x_{2}-z)^{n_{2}-m-1}e^{z^{2}/2}}{(n_{2}-m-1)!\prod_{r=1}^{m}(z-x_{r}^{(m)})} 
     \\ & \cdot \sum_{k=1}^{m} (-1)^{k-1}e_{k-1}(x^{(m)}\setminus x_{j}^{(m)}) \frac{(m-k)!2^{-\frac{\ell-k}{2}}}{((\ell-k)/2)!(m-\ell)!}\mathbf{1}_{\ell-k\text{ is even}, \ell\geq k} ,  
\end{split}
\end{align} where $c(x_{j}^{(m)})$ is the positively oriented contour containing $x_{j}^{(m)}$ and no other points from $x^{(m)}$. 
We apply an integral expression for the Hermite polynomial, specifically identity $(**)$ from \eqref{e:hermite2}, to the factor of $h_{n_{1}-\ell}(x_{1})$ which appears in \eqref{e:first}, after which that expression becomes
\begin{align*}
    & = \sum_{x_{j}^{(m)}\leq x_{2}} \frac{1}{2\pi i}\int_{c(x_{j}^{(m)})}dz \frac{(x_{2}-z)^{n_{2}-m-1}e^{z^{2}/2}}{(n_{2}-m-1)!\prod_{r=1}^{m}(z-x_{r}^{(m)})}  \frac{1}{\sqrt{2\pi}}\int_{-\infty}^{\infty}dt \cdot \exp{\left(-\frac{t^{2} }{2}- ix_{1}t \right)}
     \\ & \cdot t^{n_{1}-m}\sum_{k=1}^{m} (-1)^{k-1}e_{k-1}(x^{(m)}\setminus x_{j}^{(m)}) i^{n_{1}-k}\sum_{\ell=0}^{\lfloor \frac{m-k}{2}\rfloor}t^{m-k-2\ell}\frac{(m-k)!(-1)^{\ell}2^{-\ell}}{\ell!(m-2\ell-k)!}.
\end{align*}
The final sum over $\ell$ is equal to $h_{m-k}(t)$ by the Hermite sum formula \eqref{e:hermitesum}, and thus our expression \eqref{e:first} is equal to
\begin{align}\label{e:second}\begin{split}
    & = \sum_{x_{j}^{(m)}\leq x_{2}} \frac{1}{2\pi i}\int_{c(x_{j}^{(m)})}dz \frac{(x_{2}-z)^{n_{2}-m-1}e^{z^{2}/2}}{(n_{2}-m-1)!\prod_{r=1}^{m}(z-x_{r}^{(m)})}  \frac{1}{\sqrt{2\pi}}\int_{-\infty}^{\infty}dt \cdot \exp{\left(-\frac{t^{2} }{2}- ix_{1}t \right)}
     \\ & \cdot t^{n_{1}-m}\sum_{k=1}^{m} (-1)^{k-1}e_{k-1}(x^{(m)}\setminus x_{j}^{(m)}) i^{n_{1}-k}h_{m-k}(t).
     \end{split}
\end{align}
We apply another integral expression for the Hermite polynomial, in this case identity $(*)$ from \eqref{e:hermite2}, to the factor of $h_{m-k}(t)$, after which 
the expression \eqref{e:second} becomes
\begin{align*}
    & = \sum_{x_{j}^{(m)}\leq x_{2}} \frac{1}{2\pi i}\int_{c(x_{j}^{(m)})}dz \frac{(x_{2}-z)^{n_{2}-m-1}e^{z^{2}/2}}{(n_{2}-m-1)!\prod_{r=1}^{m}(z-x_{r}^{(m)})}  \frac{1}{2\pi}\int_{-\infty}^{\infty}dt \int_{-\infty}^{\infty}dw \cdot \exp{\left(- ix_{1}t -\frac{w^{2} }{2}+ itw \right)}
     \\ & \cdot t^{n_{1}-m}i^{n_{1}-m}\sum_{k=1}^{m} (-1)^{k-1}e_{k-1}(x^{(m)}\setminus x_{j}^{(m)}) w^{m-k}.
\end{align*}
We can then apply the following identity for elementary symmetric polynomials to the sum which appears in the final line,
\begin{align*}
    \sum_{k=1}^{m} (-1)^{k-1}e_{k-1}(x^{(m)}\setminus x_{j}^{(m)}) w^{m-k} & = \prod_{r\neq j}(w-x_{r}^{(m)}).
\end{align*}
After this transformation, our expression is
\begin{align}\label{e:turtle}\begin{split}
    & = \sum_{x_{j}^{(m)}\leq x_{2}} \frac{1}{2\pi i}\int_{c(x_{j}^{(m)})}dz \frac{(x_{2}-z)^{n_{2}-m-1}e^{z^{2}/2}}{(n_{2}-m-1)!\prod_{r=1}^{m}(z-x_{r}^{(m)})}  i^{n_{1}-m}
     \\ & \cdot \frac{1}{2\pi}\int_{-\infty}^{\infty}dt \int_{-\infty}^{\infty}dw \cdot \exp{\left( -\frac{w^{2} }{2}+ itw - ix_{1}t\right)}t^{n_{1}-m}\prod_{r\neq j}(w-x_{r}^{(m)}).\end{split}
\end{align}
The last line of the equation above is equal to
\begin{align*}
   \mathcal{F}\left\{t^{n_{1}-m}\mathcal{F}\left\{e^{-w^{2}/2} \prod_{r\neq j}(w-x_{r}^{(m)})\right\}(t)\right\}(-x_{1}).
\end{align*} 
Using \eqref{e:ftprop}, applying the Fourier transform twice, and then applying the product rule, we can rewrite this expression as 
\begin{align*}
   & (-i)^{n_{1}-m}\left(\frac{d}{d(-x_{1})}\right)^{n_{1}-m}\mathcal{F}\left\{\mathcal{F}\left\{e^{-w^{2}/2} \prod_{r\neq j}(w-x_{r}^{(m)})\right\}(t)\right\}(-x_{1})
   \\ & = i^{n_{1}-m}\left(\frac{d}{dx_{1}}\right)^{n_{1}-m}e^{-x_{1}^{2}/2} \prod_{r\neq j}(x_{1}-x_{r}^{(m)})
   \\ & = i^{n_{1}-m}\sum_{k=0}^{n_{1}-m} {n_{1}-m \choose k} k! e_{m-1-k}(x_{1}-x^{(m)}\setminus x_{1}-x_{j}^{(m)}) (-1)^{n_{1}-m-k} h_{n_{1}-m-k}(x_{1})e^{-x_{1}^{2}/2}.
\end{align*}
Substituting this back into \eqref{e:turtle} results in
\begin{align*}
    & = \sum_{x_{j}^{(m)}\leq x_{2}} \frac{1}{2\pi i}\int_{c(x_{j}^{(m)})}dz \frac{(x_{2}-z)^{n_{2}-m-1}e^{z^{2}/2}}{(n_{2}-m-1)!\prod_{r=1}^{m}(z-x_{r}^{(m)})} 
     \\ & \cdot \sum_{k=0}^{n_{1}-m} {n_{1}-m \choose k}k! e_{m-1-k}(x_{1}-x^{(m)}\setminus x_{1}-x_{j}^{(m)}) (-1)^{-k} h_{n_{1}-m-k}(x_{1})e^{-x_{1}^{2}/2}.
\end{align*}
Applying the contour integral expression for Hermite polynomials, identity \eqref{e:hermite1}, to the factor of $h_{n_{1}-m-k}(x_{1})$ results in the expression
\begin{align}\label{e:chicken}
    & = \frac{(n_{1}-m)!}{(n_{2}-m-1)!}\sum_{x_{j}^{(m)}\leq x_{2}} \frac{1}{(2\pi i)^{2}}\int_{c(x_{j}^{(m)})}dz \frac{(x_{2}-z)^{n_{2}-m-1}e^{z^{2}/2}}{\prod_{r=1}^{m}(z-x_{r}^{(m)})} \int_{c(0)}dw \frac{e^{-(w-x_{1})^{2}/2}}{w^{n_{1}-m+1}}
     \\ & \cdot (-1)^{m-1}\sum_{k=0}^{n_{1}-m} e_{m-1-k}(x_{1}-x^{(m)}\setminus x_{1}-x_{j}^{(m)}) (-1)^{m-1-k} w^{k}.
\end{align}
We consider the following expression for the product $\prod_{r\neq j}(w-x_{1}+x_{r}^{(m)})$
\begin{align*}
    \prod_{r\neq j}(w-x_{1}+x_{r}^{(m)}) & = \sum_{k=0}^{\infty} e_{m-1-k}(x_{1}-x^{(m)}\setminus x_{1}-x_{j}^{(m)}) (-1)^{m-1-k} w^{k},
\end{align*} from which expression it is clear that the difference of the sum in the last line of \eqref{e:chicken} and the $\prod_{r\neq j}(w-x_{1}+x_{r}^{(m)})$ is divisible by $w^{n_{1}-m+1}$ and therefore, inside the contour integral, the terms contained in that difference vanish. Consequently, making that substitution and pulling the factor of $(-1)^{m-1}$ into the product, the expression above reduces to
\begin{align*}
    & = \frac{(n_{1}-m)!}{(n_{2}-m-1)!}\sum_{x_{j}^{(m)}\leq x_{2}} \frac{1}{(2\pi i)^{2}}\int_{c(x_{j}^{(m)})}dz \frac{(x_{2}-z)^{n_{2}-m-1}e^{z^{2}/2}}{\prod_{r=1}^{m}(z-x_{r}^{(m)})} \int_{c(0)}dw \frac{e^{-(w-x_{1})^{2}/2}}{w^{n_{1}-m+1}}\prod_{r\neq j}(x_{1}-w-x_{r}^{(m)}) .
\end{align*}
We make the change of variables $\widetilde{w}=x_{1}-w$ and relabel $\widetilde{w}$ as $w$ to obtain 
\begin{align*}
    & = \frac{(n_{1}-m)!}{(n_{2}-m-1)!}\sum_{x_{j}^{(m)}\leq x_{2}} \frac{(-1)^{n_{1}+n_{2}-1}}{(2\pi i)^{2}}\int_{c(x_{j}^{(m)})}dz \frac{(z-x_{2})^{n_{2}-m-1}e^{z^{2}/2}}{\prod_{r=1}^{m}(z-x_{r}^{(m)})} \int_{c(x_{1})}dw \frac{e^{-w^{2}/2}\prod_{r\neq j}(w-x_{r}^{(m)}) }{(w-x_{1})^{n_{1}-m+1}}.
\end{align*}
Furthermore, when the contours do not intersect, and when the contour of integration for $w$ does not contain $x_{j}^{(m)}$, we see that 
\begin{multline*}
    \int_{c(x_{j}^{(m)})}dz \frac{(z-x_{2})^{n_{2}-m-1}e^{z^{2}/2}}{\prod_{r=1}^{m}(z-x_{r}^{(m)})} \int_{c(x_{1})}dw \frac{e^{-w^{2}/2}\prod_{r\neq j}(w-x_{r}^{(m)}) }{(w-x_{1})^{n_{1}-m+1}} 
    \\ =  \int_{c(x_{1})}dw \int_{c(x_{j}^{(m)})}dz \frac{(z-x_{2})^{n_{2}-m-1}e^{z^{2}/2}}{\prod_{r=1}^{m}(z-x_{r}^{(m)})}\frac{e^{-w^{2}/2}\prod_{r=1}^{m}(w-x_{r}^{(m)}) }{(w-z)(w-x_{1})^{n_{1}-m+1}}.
\end{multline*}Thus, the entire expression is equal to
\begin{align*}
    & = \frac{(n_{1}-m)!}{(n_{2}-m-1)!}\frac{(-1)^{n_{2}-n_{1}-1}}{(2\pi i)^{2}}\int_{c(x_{j}^{(m)}\leq x_{2})}dz \int_{c(x_{1})}dw\frac{(z-x_{2})^{n_{2}-m-1}}{(w-x_{1})^{n_{1}-m+1}} \frac{e^{z^{2}/2-w^{2}/2}}{w-z} \prod_{r=1}^{m}\frac{w-x_{r}^{(m)}}{z-x_{r}^{(m)}}.
\end{align*} 
\end{proof}
\subsection{Integral Expression for the Final Two Terms}\label{s:finalterm}
Before we begin the proof of \cref{l:termsolve2}, we will demonstrate a nontrivial identity which will appear multiple times in that proof. 
\begin{lemma}\label{l:annoying} For $d,x_{2}\in\mathbb{R}$, $r\in\mathbb{C}$ such that $\Re(r)>d$ and for $m,n_{2}\in\mathbb{N}$ such that $n_{2}>m$, 
    \begin{multline}\label{e:tada}
        \frac{1}{2\pi i}\int\limits_{d-i\infty}^{d+i\infty} d z
    \frac{(z-x_{2})^{n_{2}-m-1}}{z-r}
    e^{z^{2}/2} \\ = (n_{2}-m-1)!\frac{(-1)^{n_{2}-m}}{\sqrt{2\pi}}\int_{-\infty}^{\infty} dy \mathbf{1}_{y>0}e^{-ry-y^{2}/2}\sum_{\ell=m+1}^{n_{2}} \frac{h_{n_{2}-\ell}(x_{2})}{(n_{2}-\ell)!}\cdot \frac{y^{\ell-m-1}}{(\ell-m-1)!}.
    \end{multline}
\end{lemma}
\begin{proof}
    For $\Re(r)>d$, we obtain a general form for the expression 
  \begin{align*}
      \frac{1}{2\pi i}\int\limits_{d-i\infty}^{d+i\infty} d z
    \frac{(z-x_{2})^{n_{2}-m-1}}{z-r}
    e^{z^{2}/2} & = \frac{e^{d^{2}/2}}{2\pi }\int\limits_{-\infty}^{\infty} d u
    \frac{(iu+d-x_{2})^{n_{2}-m-1}}{iu+d-r}
    e^{-u^{2}/2+iud} 
    \\ & = \frac{e^{d^{2}/2}}{\sqrt{2\pi} }\mathcal{F}\left\{
    \frac{(iu+d-x_{2})^{n_{2}-m-1}e^{-u^{2}/2}}{iu+d-r}\right\}(d) \nonumber
    \\ & =\frac{e^{d^{2}/2}}{2\pi }\left(\mathcal{F}\left\{
    \frac{1}{iu+d-r}\right\}(\cdot)*\mathcal{F}\left\{
    (iu+d-x_{2})^{n_{2}-m-1}e^{-u^{2}/2}\right\}(\cdot)\right)(d), 
  \end{align*} where the $*$ in the line directly above is the normal notion of convolution, distinct from $\star$.
  We compute these Fourier transforms, noting that
  \begin{align*}
      \mathcal{F}\left\{
    (iu+d-x_{2})^{n_{2}-m-1}e^{-u^{2}/2}\right\}(a) & = \frac{i^{n_{2}-m-1}e^{(d-x_{2})^{2}/2 -(d-x_{2})a}}{\sqrt{2\pi}}\int_{-\infty}^{\infty}
    u^{n_{2}-m-1}e^{-u^{2}/2-iu(d-x_{2}-a)}du
    \\ & = h_{n_{2}-m-1}(d-x_{2}-a)e^{ -a^{2}/2}.
  \end{align*}
  Since $\Re(r)>d$, we can write $c=r-d$ such that $\Re(c)>0$, and thus, 
\begin{align*}
    \mathcal{F}\left\{
    \frac{1}{iu+d-r}\right\}(a) & = \frac{1}{i\sqrt{2\pi}}\int_{-\infty}^{\infty}\frac{1}{u-id+ir}e^{iua}du
    \\ & = \frac{1}{i\sqrt{2\pi}}\int_{-\infty}^{\infty}\frac{1}{u+ic}e^{iua}du = -\sqrt{2\pi}e^{ca}\mathbf{1}_{a<0} = -\sqrt{2\pi}e^{a(r-d)}\mathbf{1}_{a<0}.
\end{align*}
Taking the convolution of these two functions, therefore, results in
  \begin{align*}
      \int_{-\infty}^{\infty}h_{n_{2}-m-1}(d-x_{2}-a)e^{-a^{2}/2}(-\sqrt{2\pi})e^{(d-a)(r-d)}\mathbf{1}_{0<a-d} da.
  \end{align*}
  Writing $y=a-d,$ this becomes
  \begin{align*}
      e^{-d^{2}/2}\sqrt{2\pi}(-1)^{n_{2}-m}\int_{-\infty}^{\infty}\mathbf{1}_{y>0} e^{-y^{2}/2-yr}h_{n_{2}-m-1}(x_{2}+y)dy.
  \end{align*} Finally, we apply the summation identity for Hermite polynomials, \eqref{e:hermitesumexpand} (under the change of summation variables $\ell\mapsto n_{2}-\ell$) to obtain 
  \begin{align*}
      e^{-d^{2}/2}\sqrt{2\pi}(-1)^{n_{2}-m}(n_{2}-m-1)!\int_{-\infty}^{\infty}\mathbf{1}_{y>0} e^{-y^{2}/2-yr}\sum_{\ell=m+1}^{n_{2}} \frac{h_{n_{2}-\ell}(x_{2})}{(n_{2}-\ell)!}\cdot \frac{y^{\ell-m-1}}{(\ell-m-1)!} dy
  \end{align*}

Thus the entire expression becomes \begin{multline*}
    \frac{e^{d^{2}/2}}{2\pi}\left(\mathcal{F}\left\{
    \frac{1}{iu+d-r}\right\}(\cdot)*\mathcal{F}\left\{
    (iu+d-x_{2})^{n_{2}-m-1}e^{-u^{2}/2}\right\}(\cdot)\right)(d)
    \\ = (n_{2}-m-1)!\frac{(-1)^{n_{2}-m}}{\sqrt{2\pi} }\int_{-\infty}^{\infty} dy \mathbf{1}_{y>0}e^{-ry-y^{2}/2}\sum_{\ell=m+1}^{n_{2}} \frac{h_{n_{2}-\ell}(x_{2})}{(n_{2}-\ell)!}\cdot \frac{y^{\ell-m-1}}{(\ell-m-1)!}. 
\end{multline*}
\end{proof}
\begin{proof}[Proof of \cref{l:termsolve2}]
We will start by obtaining an integral expression for the first term. We can plug in the expression for $\psi_{k}(x_{1}|n_{1})$ \eqref{e:psikbig} with $L=n_{1}+k-2m$ to rewrite this term as
\begin{align*}
     \sum_{k=m+1}^{n_{2}} \frac{h_{n_{2}-k}(x_{2})\psi_{k}(x_{1}|n_{1})}{\sqrt{2\pi}(n_{2}-k)!}
     & = \sum_{k=m+1}^{n_{2}} \frac{h_{n_{2}-k}(x_{2})}{\sqrt{2\pi}(n_{2}-k)!} \int_{-\infty}^{\infty}dy \mathbf{1}_{y>x_{1}}\frac{(y-x_{1})^{k-m-1}}{(k-m-1)!}h_{n_{1}-m}(y)e^{-y^{2}/2}
     \\ & = \sum_{k=m+1}^{n_{2}} \frac{h_{n_{2}-k}(x_{2})}{\sqrt{2\pi}(n_{2}-k)!} \int_{-\infty}^{\infty}dy \mathbf{1}_{y>0}\frac{(y)^{k-m-1}}{(k-m-1)!}h_{n_{1}-m}(y+x_{1})e^{-(y+x_{1})^{2}/2},
\end{align*}after which we apply \eqref{e:hermite1} to obtain the expression
\begin{align*}
    & = (-1)^{n_{1}-m}\frac{(n_{1}-m)!}{2\pi i\sqrt{2\pi}} \int_{c(x_{1})} \frac{dw}{(w-x_{1})^{n_{1}-m+1}}e^{-w^{2}/2}\int_{-\infty}^{\infty}dy \mathbf{1}_{y>0}e^{-y^{2}/2 -yw}\sum_{k=m+1}^{n_{2}} \frac{h_{n_{2}-k}(x_{2})}{(n_{2}-k)!} \frac{y^{k-m-1}}{(k-m-1)!}.
\end{align*} 
This expression contains a copy of \eqref{e:tada} with $r=w$. Thus, applying \cref{l:annoying}, we see that for all $d<\inf_{w\in c(x_{1})}\Re(w)$,
\begin{align*}
    & = \frac{(-1)^{n_{1}-n_{2}}}{(2\pi i)^{2}}\frac{(n_{1}-m)!}{(n_{2}-m-1)!} \int\limits_{d-i\infty}^{d+i\infty} d z\int_{c(x_{1})}dw \frac{(z-x_{2})^{n_{2}-m-1}}{(w-x_{1})^{n_{1}-m+1}}
    \frac{e^{z^{2}/2-w^{2}/2}}{z-w}.
\end{align*} 

To obtain an integral expression for the second term, we plug in the expression for $\upsilon_{\ell}(y|m)$ to obtain
\begin{align}\label{e:sheep}\begin{split} 
    &  -\sum_{\ell=1}^{m}h_{n_{1}-\ell}(x_{1})e^{-x_{1}^{2}/2} 
    \upsilon_{\ell}(y|m)\star
     \sum_{k=m+1}^{n_{2}} \frac{h_{n_{2}-k}(x_{2})\psi_{k}(y|m)}{\sqrt{2\pi}(n_{2}-k)!}  
   \\  & = -\sum_{j=1}^{m}\frac{e^{(x_{j}^{(m)})^{2}/2}}{\prod_{r\neq j}(x_{j}^{(m)}-x_{r}^{(m)})}
     \sum_{k=m+1}^{n_{2}} \frac{h_{n_{2}-k}(x_{2})\psi_{k}(x_{j}^{(m)}|m)}{\sqrt{2\pi}(n_{2}-k)!} 
     \\ & \cdot \sum_{\ell=1}^{m}h_{n_{1}-\ell}(x_{1})e^{-x_{1}^{2}/2}\sum_{r=1}^{m}(-1)^{r-1}e_{r-1}(x^{(m)}\setminus x_{j}^{(m)})\frac{(m-r)!2^{-\frac{\ell-r}{2}}}{((\ell-r)/2)!(m-\ell)!} \mathbf{1}_{\ell-r\text{ even};\ell\geq r}.
     \end{split}
\end{align}
By repeating a series of steps similar to those which appear in the proof of \cref{l:termsolve1}, the last line of \eqref{e:sheep} is equal to 
\begin{align*}
    & \sum_{\ell=1}^{m}h_{n_{1}-\ell}(x_{1})e^{-x_{1}^{2}/2}\sum_{r=1}^{m}(-1)^{r-1}e_{r-1}(x^{(m)}\setminus x_{j}^{(m)})\frac{(m-r)!2^{-\frac{\ell-r}{2}}}{((\ell-r)/2)!(m-\ell)!} \mathbf{1}_{\ell-r\text{ even};\ell\geq r}
    \\ & = \frac{1}{\sqrt{2\pi}}\int_{-\infty}^{\infty}dt e^{-t^{2}/2-ix_{1}t} t^{n_{1}-m}\sum_{r=1}^{m}(-1)^{r-1}e_{r-1}(x^{(m)}\setminus x_{j}^{(m)})i^{n_{1}-r}h_{m-r}(t)
    \\ & = i^{n_{1}-m}\frac{1}{2\pi}\int_{-\infty}^{\infty}dt\int_{-\infty}^{\infty}dw e^{-w^{2}/2+iwt-ix_{1}t}t^{n_{1}-m}\prod_{r\neq j}(w-x_{r}^{(m)})
    \\ & = i^{n_{1}-m}\mathcal{F}\left\{t^{n_{1}-m}\mathcal{F}\left\{ e^{-w^{2}/2}\prod_{r\neq j}(w-x_{r}^{(m)})\right\}(t)\right\}(-x_{1})
    \\ & =(-1)^{n_{1}-m}\left(\frac{d}{dx_{1}}\right)^{n_{1}-m}e^{-x_{1}^{2}/2}\prod_{r\neq j}(x_{1}-x_{r}^{(m)}).
\end{align*}
    Plugging this back into \eqref{e:sheep} and differentiating, we find that the left-hand side of that equation becomes  
    \begin{align*}
         & = (-1)^{n_{1}-m-1}\sum_{j=1}^{m}\frac{e^{(x_{j}^{(m)})^{2}/2}}{\prod_{r\neq j}(x_{j}^{(m)}-x_{r}^{(m)})}
     \sum_{k=m+1}^{n_{2}} \frac{h_{n_{2}-k}(x_{2})\psi_{k}(x_{j}^{(m)}|m)}{\sqrt{2\pi}(n_{2}-k)!} 
     \\ & \cdot \sum_{l=0}^{n_{1}-m}{n_{1}-m \choose l} (-1)^{n_{1}-m-l}l!e_{m-1-l}\left(x_{1}-x^{(m)}\setminus x_{1}-x_{j}^{(m)}\right)h_{n_{1}-m-l}(x_{1})e^{-x_{1}^{2}/2}  
    \end{align*}
    Applying integral identity \cref{e:hermite1} to $h_{n_{1}-m-l}(x_{1})$, this becomes
    \begin{align*}
         & = \sum_{j=1}^{m}\frac{e^{(x_{j}^{(m)})^{2}/2}}{\prod_{r\neq j}(x_{j}^{(m)}-x_{r}^{(m)})}
     \sum_{k=m+1}^{n_{2}} \frac{h_{n_{2}-k}(x_{2})\psi_{k}(x_{j}^{(m)}|m)}{\sqrt{2\pi}(n_{2}-k)!} 
     \\ & \cdot (-1)^{m}\frac{(n_{1}-m)!}{2\pi i}\int_{c(0)}dw\frac{e^{-(x_{1}-w)^{2}/2}}{w^{n_{1}-m+1}}\sum_{l=0}^{n_{1}-m}(-1)^{m-1-l}e_{m-1-l}\left(x_{1}-x^{(m)}\setminus x_{1}-x_{j}^{(m)}\right)w^{l}.
     \end{align*} As before, we note that 
     \begin{align*}
    \prod_{r\neq j}(w-x_{1}+x_{r}^{(m)}) & = \sum_{k=0}^{\infty} e_{m-1-k}(x_{1}-x^{(m)}\setminus x_{1}-x_{j}^{(m)}) (-1)^{m-1-k} w^{k},
\end{align*} and since all terms of order higher than $w^{n_{1}-m}$ vanish under our closed contour integral, the expression becomes 
\begin{align*}
     & = -\sum_{j=1}^{m}\frac{e^{(x_{j}^{(m)})^{2}/2}}{\prod_{r\neq j}(x_{j}^{(m)}-x_{r}^{(m)})}
     \sum_{k=m+1}^{n_{2}} \frac{h_{n_{2}-k}(x_{2})\psi_{k}(x_{j}^{(m)}|m)}{\sqrt{2\pi}(n_{2}-k)!} 
      \frac{(n_{1}-m)!}{2\pi i}\int_{c(0)}dw\frac{e^{-(x_{1}-w)^{2}/2}}{w^{n_{1}-m+1}}\prod_{r\neq j}(x_{1}-w-x_{r}^{(m)}).
\end{align*} Changing variables in the integral and plugging in the definition of $\psi_{k}(x_{j}^{(m)}|m)$ (and then making a further change of variables inside the integral in that expression) results in 
\begin{align*}
     & = (-1)^{n_{1}-m-1} \frac{(n_{1}-m)!}{2\pi i\sqrt{2\pi}}\sum_{j=1}^{m}\frac{1}{\prod_{r\neq j}(x_{j}^{(m)}-x_{r}^{(m)})}
     \int_{-\infty}^{\infty} dy \mathbf{1}_{y>0}e^{-y^{2}/2-yx_{j}^{(m)}}\sum_{k=m+1}^{n_{2}} \frac{h_{n_{2}-k}(x_{2})}{(n_{2}-k)!} \frac{y^{k-m-1}}{(k-m-1)!}
   \\ & \cdot \int_{c(x_{1})}dw\frac{e^{-w^{2}/2}}{(w-x_{1})^{n_{1}-m+1}}\prod_{r\neq j}(w-x_{r}^{(m)}).
\end{align*}  
We notice that this expression contains a copy of the right-hand side of \eqref{e:tada} with $r=x_{j}^{(m)}$. We can therefore exchange that factor for an expression of the form of the left-hand side of \eqref{e:tada}. Therefore, for all $d<x_{j}^{(m)}$,
\begin{align*}
     & = \frac{(-1)^{n_{1}-n_{2}-1}}{(2\pi i)^2} \frac{(n_{1}-m)!}{(n_{2}-m-1)!}\sum_{j=1}^{m}\frac{1}{\prod_{r\neq j}(x_{j}^{(m)}-x_{r}^{(m)})}
    \int\limits_{d-i\infty}^{d+i\infty} d z
    \frac{(z-x_{2})^{n_{2}-m-1} e^{z^{2}/2}}{z-x_{j}^{(m)}}
    \int_{c(x_{1})}dw\frac{e^{-w^{2}/2}\prod_{r\neq j}(w-x_{r}^{(m)})}{(w-x_{1})^{n_{1}-m+1}}.
\end{align*}  
Finally, to combine the terms, we note that 
  \begin{align*}
      & \frac{\prod_{r=1}^{m}(w-x_{r}^{(m)})}{(z-w)\prod_{r=1}^{m}(z-x_{r}^{(m)})}  = \frac{1}{z-w} + \sum_{j=1}^{m}\frac{C_{j}}{z-x_{j}^{(m)}}.
  \end{align*}
  For $C_{j}$ defined by 
  \begin{align*}
      C_{j} := -\frac{\prod_{r\neq j}(w-x_{r}^{(m)})}{\prod_{r\neq j}(x_{j}^{(m)}-x_{r}^{(m)})}.
  \end{align*}
We have shown that for $d<x_{1},x_{2}$ and less than all elements of $x^{(m)}$, the sum of the two terms in \cref{l:termsolve2} is equal to 
\begin{align*}
 & =\frac{(-1)^{n_{1}-n_{2}}}{(2\pi i)^{2}}\frac{(n_{1}-m)!}{(n_{2}-m-1)!} \int\limits_{d-i\infty}^{d+i\infty} d z\int_{c(x_{1})}dw \frac{(z-x_{2})^{n_{2}-m-1}}{(w-x_{1})^{n_{1}-m+1}}
    \frac{e^{z^{2}/2-w^{2}/2}}{z-w}
    \\ & +\sum_{j=1}^{m} \frac{(-1)^{n_{1}-n_{2}-1}}{(2\pi i)^2} \frac{(n_{1}-m)!}{(n_{2}-m-1)!}
   \int\limits_{d-i\infty}^{d+i\infty} d z
    \frac{(z-x_{2})^{n_{2}-m-1} e^{z^{2}/2}}{z-x_{j}^{(m)}}
    \int_{c(x_{1})}dw\frac{e^{-w^{2}/2}}{(w-x_{1})^{n_{1}-m+1}}\frac{\prod_{r\neq j}(w-x_{r}^{(m)})}{\prod_{r\neq j}(x_{j}^{(m)}-x_{r}^{(m)})}.
\end{align*}
As a consequence, we can combine the terms to obtain
\begin{align*}
 & =\frac{(-1)^{n_{1}-n_{2}}}{(2\pi i)^{2}}\frac{(n_{1}-m)!}{(n_{2}-m-1)!} \int\limits_{d-i\infty}^{d+i\infty} d z\int_{c(x_{1})}dw \frac{(z-x_{2})^{n_{2}-m-1}}{(w-x_{1})^{n_{1}-m+1}}
    \frac{e^{z^{2}/2-w^{2}/2}}{z-w}\prod_{r=1}^{m}\frac{w-x_{r}^{(m)}}{z-x_{r}^{(m)}},
    \\ & = \frac{(-1)^{n_{1}-n_{2}-1}}{(2\pi i)^{2}}\frac{(n_{1}-m)!}{(n_{2}-m-1)!} \int\limits_{d-i\infty}^{d+i\infty} d z\int_{c(x_{1})}dw \frac{(z-x_{2})^{n_{2}-m-1}}{(w-x_{1})^{n_{1}-m+1}}
    \frac{e^{z^{2}/2-w^{2}/2}}{w-z}\prod_{r=1}^{m}\frac{w-x_{r}^{(m)}}{z-x_{r}^{(m)}},
\end{align*} 
where $d$ is any value which lies to the left of $x_{1},x_{2}$ and all of the $x^{(m)}$. 
\end{proof}
\section{Convergence to the Extended Semi-Discrete Sine Kernel}\label{s:convrate}
In this section, we prove \cref{t:convrate}. Most of the section will be devoted to using the method of steepest descent to obtain pointwise convergence of the kernels.  
We use $\{\Omega_{m}\}_{m\in\mathbb{N}}$ to denote the intersection of the rigidity event of \cref{t:local} with the simple spectrum event of \cref{r:simple}. 
To set up the scaling in \cref{t:convrate}, we note that the eigenvalues in question live on scale $\sqrt{m}$, and we want to look at a short time $T$ above the level $m$. Up to gauge transformations and under the change of variables $z\mapsto \frac{T}{\sqrt{m}}z+X\sqrt{m},w\mapsto \frac{T}{\sqrt{m}}w +X\sqrt{m}$, the appropriately rescaled kernel is equivalent to 
\begin{multline}\label{e:presteepkernel}
   \frac{1}{\sqrt{m}}\widetilde{K}_{x^{(m)}}\left(n_{1}+T, X\sqrt{m}+\frac{x_{1}}{\sqrt{m}}; n_{2}+T, X\sqrt{m}+\frac{x_{2}}{\sqrt{m}}\right) \\ =\frac{n_{1}+T}{T} \frac{-1}{(2\pi i)^{2}}\int_{\frac{x_{2}}{T}+\varepsilon-i\infty}^{\frac{x_{2}}{T}+\varepsilon+i\infty} dz \oint_{c\left(\frac{x_{1}}{T}\right)}dw R_{m}(z,w) \frac{e^{T\left(S_{m}(z) -S_{m}(w)\right) }}{w-z},
\end{multline} 
where 
\begin{align*}
R_{m}(z,w) := \frac{\left(z-\frac{x_{2}}{T}\right)^{n_{2}+T-1}w^{T}}{z^{T}\left(w-\frac{x_{1}}{T}\right)^{n_{1}+T+1}}, 
& & S_{m}(z)  :=zX  + \log_{\mathbb{H}}{(z)} +\frac{Tz^{2}}{2m}-\frac{1}{T}\sum_{r=1}^{m}\log_{\mathbb{H}}{\left(z-u_{r}^{(m)}\right)}, 
\end{align*} with $u_{r}^{(m)}:=(x_{r}^{(m)}/\sqrt{m}-X)/(T/m)$, and, following \cite{GP19}, we use the notation $\log_{\mathbb{H}}(z) = \log{(ze^{-i\pi/2})} +i\pi/2$ to denote the complex logarithm defined with the cut along the negative imaginary axis so that $S_{m}(z)$ is holomorphic in the upper half-plane $\mathbb{H}:=\{z\in\mathbb{C}, \Im(z)>0\}$. We needed to make a choice of branch cut in order for the expression above to be well defined. In the integrand of the kernel, however, $S_{m}(z)$ only appears in an exponential, and, consequently, the choice of branch cut does not affect the integral. 

In this section we will establish bounds explicitly in the upper half-plane. The lower half-plane estimates follow in the same way. When we work in the upper half-plane, we formally define $S_{m}(x):=\lim_{r\to 0}S_{m}(x+ir)$ for all $x\in\mathbb{R}$. 

We also note the expression for $S_{m}'(z),$
\begin{align}\label{e:sderiv}
S_{m}'(z) & = X +\frac{1}{z} +  \frac{T}{m}z-\frac{1}{T} \sum_{r=1}^{m}\frac{1}{z-u_{r}^{(m)}}.
\end{align}

\begin{remark}\label{r:paramdep}
    We observe that $m^{-1/2}\widetilde{K}_{x^{(m)}}(n_{1}+T,Xm^{1/2}+x_{1}m^{-1/2};n_{2}+T,Xm^{1/2}+ x_{2}m^{-1/2})$ depends on $x_{1},x_{2},n_{1},n_{2}$ only through $R_{m}(z,w)$, the pre-factor $(n_{1}+T)T^{-1}$, and the placement and definition of the contours. 
\end{remark}

The proof of \cref{t:convrate} has three main parts. We begin by locating the unique critical point of $S_{m}(z)$ in the upper half-plane (\cref{s:zeros}). We construct local steepest descent
and ascent contours through this critical point (\cref{s:contours}) within a region defined by
\begin{align}\label{e:region}
    \mathcal{R}(c,\kappa ):=\{x+iy\in\mathbb{C}| |x|<(c\kappa )^{-1}, \kappa <y<\kappa ^{-1}\},
\end{align} for $\kappa \in(0,1)$ and $c\in (0,1)$. Then we construct the global steepest descent contours (\cref{s:contours2}). We deform the original contours to these steepest descent contours; and apply \cref{p:steepest}, see below, on $\mathcal{R}(c,\kappa)$, together with an estimate on the decay of the integral along the global steepest descent contours outside of $\mathcal{R}(c,\kappa)$ (\cref{l:errortermscross}, \cref{s:crossterms}) to show that the steepest descent contribution vanishes. Then we show that
the remaining residue contour converges to the extended semi-discrete sine kernel (\cref{s:proofofmain}).

To show that the contribution along the steepest descent contours vanishes, we will need the next proposition, which we prove in \cref{s:steepest}. 
\begin{proposition}\label{p:steepest}
       Let $D\subset \mathbb{C}$ be open and let $\{f_{n}\}_{n\in\mathbb{N}}$, $\{g_{n}\}_{n\in\mathbb{N}}$ be sequences of functions defined so that for all $n\in\mathbb{N}$, $f_{n}:\mathbb{C}^{2}\to\mathbb{C}$, $g_{n}:\mathbb{C}\to\mathbb{C}$ with $f_{n}$ holomorphic on $D\times D,$ and $g_{n}$ holomorphic on $D.$ 
    
    Suppose that for all $n\in\mathbb{N}$, $g_{n}$ has exactly one critical point $z_{0}^{(n)}\in D$, that $z_{0}^{(n)}$ is nondegenerate, and that the $z_{0}^{(n)}$ converge to a point $z_{0}\in D$ as $n\to\infty$. There exist neighborhoods $U_{n}$ of $z_{0}^{(n)}$ where the complex Morse lemma guarantees a local chart $\zeta_{n}$ such that $g_{n}(z)=g_{n}(z_{0}^{(n)})+\frac{1}{2}\zeta_{n}(z)^{2}$ and $C_{1}^{(n)},C_{2}^{(n)}\subset D$ to be contours such that for all $n\in\mathbb{N}$ there exists $\delta_{n}>0$ such that $C_{1}^{(n)}\cap U_{n}:=\{\zeta^{-1}_{n}(it):|t|<\delta_{n}\},C_{2}^{(n)}\cap U_{n}:=\{\zeta^{-1}_{n}(t):|t|<\delta_{n}\}$ are respectively the steepest descent and ascent branches of the set $\{\Im(g_{n}(z))=\Im(g_{n}(z_{0}^{(n)})\}$, and $C_{1}^{(n)}\setminus U_{n}$ and $C_{2}^{(n)}\setminus U_{n}$ are the contours obtained by continuing from those branches along the integral curves of $-\nabla \Re(g_{n}(z))$ and $\nabla \Re(g_{n}(z))$.
    \begin{enumerate}
        \item Further assume that there exists $c_{1}>0$ such that $\inf_{n}|g_{n}''(z_{0}^{(n)})|>c_{1}$ and that there exist $c_{2},r>0$ such that $B_{r}(z_{0}^{(n)})\subset D$ for all $n\in\mathbb{N}$ and  $\sup_{n}\sup_{|z-z_{0}^{(n)}|<r}|g_{n}'''(z)|<c_{2}$. Then it is possible to choose Morse charts $\zeta_{n}$ and neighborhoods $U_{n}$ so that there exists an open neighborhood $z_{0}\in U\subset D$  and $n_{0}\in\mathbb{N}$ such that $U\subset U_{n}$ for all $n>n_{0}$.
        \item Assuming $(1)$, and further assuming that there exists $M>0$ such that $\sup_{n}\|f_{n}\|_{L^{\infty}(C_{1}^{(n)}\times C_{2}^{(n)})}<M$ and that there exists $c_{3}>0$ such that for all $z\in C_{1}^{(n)}\setminus U$, $w\in C_{2}^{(n)}\setminus U$, and $n>n_{0},$
    \begin{align*}
        \Re\left(g_{n}(z_{0}^{(n)})-g_{n}(z)\right), \Re\left(g_{n}(w) - g_{n}(z_{0}^{(n)})\right)\geq c_{3},
    \end{align*} that $C_{1}^{(n)},C_{2}^{(n)}$ have uniformly bounded length, bounded by $L>0$, and that there exists $c_{4}>0$ such that 
    \begin{align*}
    \inf_{n>n_{0}}\min\left\{\inf_{\substack{z\in C_{1}^{(n)}\setminus U \\ w\in C_{2}^{(n)}\setminus U}}|w-z|, \inf_{\substack{z\in C_{1}^{(n)}\setminus U \\ w\in C_{2}^{(n)}\cap U}}|w-z|, \inf_{\substack{z\in C_{1}^{(n)}\cap U \\ w\in C_{2}^{(n)}\setminus U}}|w-z|\right\}>c_{4}.
    \end{align*}
    Then there exists $c>0$ such that 
    \begin{align*}
        \bigg|\int_{C_{1}^{(n)}}dz\int_{C_{2}^{(n)}}dw f_{n}(z,w)\frac{e^{n(g_{n}(z)-g_{n}(w))}}{w-z}\bigg|\leq cML^{2}c_{4}^{-1}  n^{-1/2},\end{align*} where we interpret the double contour integral as the limit 
        \begin{align*}
            \lim_{\varepsilon\to 0}\int_{C_{1}^{(n)}\setminus B_{\varepsilon,n}}dz\int_{C_{2}^{(n)}\setminus B_{\varepsilon,n}}dw f_{n}(z,w)\frac{e^{n(g_{n}(z)-g_{n}(w))}}{w-z}
        \end{align*} where $B_{\varepsilon,n}$ is a small ball of radius $\varepsilon$ containing $z_{0}^{(n)}$.

        Furthermore, we can obtain, from the same argument, the slightly stronger statement that there exists $c>0$ such that 
        \begin{align}\label{e:latter}
        \bigg|\int_{C_{1}^{(n)}}dz\int_{C_{2}^{(n)}}dw f_{n}(z,w)\frac{e^{n(g_{n}(z)-g_{n}(w))}}{w-z}\bigg| 
        & \leq cML^{2}c_{4}^{-1}e^{-c_{3}n} + c\sup_{n}\{\sup_{z,w\in U}|f_{n}(z,w)|\}\cdot n^{-1/2}.\end{align}
    \end{enumerate}
\end{proposition}
\begin{remark}\label{r:steepest}
    Similarly, if the same statements and bounds hold when we reflect the region $D$ over the real axis, and if the functions involved are symmetric above the real axis, and we call the resulting reflected contours $C_{1}^{(n),-},C_{2}^{(n),-}$, then  
    \begin{align*}\bigg|\int_{C_{1}^{(n)}}dz\int_{C_{2}^{(n),-}}dw f_{n}(z,w)\frac{e^{n(g_{n}(z)-g_{n}(w))}}{w-z}\bigg|, 
        & & \bigg|\int_{C_{1}^{(n),-}}dz\int_{C_{2}^{(n)}}dw f_{n}(z,w)\frac{e^{n(g_{n}(z)-g_{n}(w))}}{w-z}\bigg|,  
        \end{align*}  decay at the same rate as in \eqref{e:latter}.
\end{remark}
In all subsections of this section, it will be helpful to establish convergence results for $S_{m},S'_{m}$, therefore, before continuing into the proof of \cref{t:convrate}, we prove the following lemma.  
\begin{lemma}\label{l:sconv} Fix $\alpha\in (0,2)$. There exists $C:=C(\alpha)>0$ such that for all $z\in\mathbb{H}$, $x^{(m)}\in\Omega_{m}$, $X\in (-2+\alpha,2-\alpha)$,  $$ |S_{m}'(z)-S_{*}'(z)|\leq C\Im(z)^{-1}\log^{2}{(m)}T^{-1}+Tm^{-1}|z|,$$ where $S_{*}'(z)  :=\frac{1}{2}X+\frac{1}{z} +i\pi\rho_{sc}(X)$.
    
    Furthermore, there exists a measurable function $F:\mathrm{Conf}(\mathbb{R})\times (-2,2)\to\mathbb{C}$ and a sequence of random variables $C_{m}:=F(x^{(m)},X)$ and $C:=C(\alpha)>0$ such that for all $z\in\mathbb{H}$ for all $x^{(m)}\in\Omega_{m}$ and $X\in(-2+\alpha,2-\alpha)$, $$|S_{m}(z)-C_{m}-S_{*}(z)|\leq  C|z-i| \Im(z)^{-1/2}\log^{2}{(m)}T^{-1}+T|z|^{2}/2m,$$ where $S_{*}(z) := \frac{1}{2}X(z+i) +\log_{\mathbb{H}}{(z)} + i\pi\rho_{sc}(X)(z-i).$  

    Finally, there exists $C:=C(c,\alpha)>0$ such that for all $x^{(m)}\in\Omega_{m}$, $X\in (-2+\alpha,2-\alpha)$ and $z\in \mathcal{R}(c,\kappa),$
    $$ |S_{m}'(z)-S_{*}'(z)|\leq  C\kappa^{-1}e_{m},$$ where $e_{m}:=\log^2(m)T^{-1}+Tm^{-1}$. Similarly, there exists $C(c,\alpha)>0$ such that for all  $x^{(m)}\in\Omega_{m}$, $X\in (-2+\alpha,2-\alpha),$ and  $z\in \mathcal{R}(c,\kappa)$, $$|S_{m}(z)-C_{m}-S_{*}(z)|\leq C(\kappa^{-3/2}e_{m}+\kappa^{-2}Tm^{-1}).$$
\end{lemma}
\begin{proof}[Proof of \cref{l:sconv}] 
To prove the first statement, we introduce the parameter $R>0$ and write the sum as 
\begin{align*}
    \frac{1}{T} \sum_{r=1}^{m}\frac{1}{z-u_{r}^{(m)}} & = G_{R,m}(z) + d_{m}(R),
\end{align*} where 
\begin{align*}
    G_{R,m}(z) := \frac{1}{T} \sum_{r=1}^{m}\left(\frac{1}{z-u_{r}^{(m)}} +\frac{\mathbf{1}_{|u_{r}^{(m)}|>R}}{u_{r}^{(m)}}\right), & & d_{m}(R)  = -\frac{1}{T}\sum_{|u_{r}^{(m)}|>R }\frac{1}{u_{r}^{(m)}},
\end{align*} noting that this $d_{m}(R)$ is the same as that which appeared in \cref{t:local}. Thus, for any $R>0$ we can write 
\begin{align*}
    S'_{m}(z) & = X+\frac{1}{z}+\frac{T}{m}z - G_{R,m}(z)-d_{m}(R). 
\end{align*} The reason to introduce $G_{R,m},d_{m}(R)$ is that for all $z\in\mathbb{H}$, $(z-u)^{-1}+\mathbf{1}_{|u|>R}u^{-1}$ is integrable with respect to the Lebesgue measure, and is continuous on any compact interval $[-M,M]\subset\mathbb{R}$. We use the notation $t_{M,s}(u)$ for some $s>0$ to denote a function which is $1$ on $[-M,M]$, $0$ outside $[-M-s,M+s]$ and smooth and bounded between $0$ and $1$ on $[-M-s,M+s]\setminus [-M,M]$. For all $z\in\mathbb{H}$, we define a truncation  
\begin{align*}
    G_{R,m}^{(M,s)}(z) := \int_{\mathbb{R}}t_{M,s}(u)\left(\frac{1}{z-u} +\frac{\mathbf{1}_{|u|>R}}{u}\right)\mu_{m}(du),
\end{align*} so that $\lim_{M\to\infty}G_{R,m}^{(M,s)}(z)=G_{R,m}(z)$.  We apply \cref{t:local} $(1)$ to conclude that for any $z\in\mathbb{C}\setminus\mathbb{R}$, there exists $C:=C(\alpha)>0$ such that for all $x^{(m)}\in\Omega_{m}$, $X\in (-2+\alpha,2-\alpha),$
\begin{align*}
    \big| G_{R,m}^{(M,s)}(z) - \int_{\mathbb{R}}t_{M,s}(u)  \left(\frac{1}{z-u}+\frac{\mathbf{1}_{|u|>R}}{u}\right) \rho_{sc}(X+uT/m)du \big| \leq C (\Im(z)^{-1}+R^{-1})\log^{2}{(m)}T^{-1}, 
\end{align*} where the factor of $(\Im(z)^{-1}+R^{-1})$ comes from the total variation in \cref{t:local}. Thus, for all $x^{(m)}\in\Omega_{m},$ $X\in (-2+\alpha,2-\alpha),$ and $M>0$, 
\begin{align*}
    \big| G_{R,m}(z) - \int_{\mathbb{R}}\left(\frac{1}{z-u}+\frac{\mathbf{1}_{|u|>R}}{u}\right) \rho_{sc}(X+uT/m)du \big|\leq C (\Im(z)^{-1}+R^{-1})\log^{2}{(m)}T^{-1} + O(M^{-1}), 
\end{align*} and thus for all $x^{(m)}\in\Omega_{m},$ $X\in (-2+\alpha,2-\alpha),$
\begin{align*}
    \big| G_{R,m}(z) - \int_{\mathbb{R}}\left(\frac{1}{z-u}+\frac{\mathbf{1}_{|u|>R}}{u}\right) \rho_{sc}(X+uT/m)du\big|\leq  C (\Im(z)^{-1}+R^{-1})\log^{2}{(m)}T^{-1}, 
\end{align*}
We combine this with \cref{t:local} $(2)$ to conclude that for all $x^{(m)}\in\Omega_{m}$, $X\in (-2+\alpha,2-\alpha),$ and for $C>0$ possibly distinct from before,  
\begin{align*}
    \big| S_{m}'(z) - \frac{1}{2}X-\frac{1}{z}+\int_{\mathbb{R}}\left(\frac{1}{z-u}+\frac{\mathbf{1}_{|u|>R}}{u}\right) \rho_{sc}(X)du \big|\leq  C (\Im(z)^{-1}+R^{-1})\log^{2}{(m)}T^{-1}+|z|Tm^{-1},
\end{align*} where the constant $C>0$ possibly takes a different value than in previous lines. Finally, we compute the integral \begin{align*}
\int_{\mathbb{R}}\left(\frac{1}{z-u}+\frac{\mathbf{1}_{|u|>R}}{u}\right) \mu_{\text{loc}}(du) & = \rho_{sc}(X)\int_{\mathbb{R}}\left(\frac{1}{z-u}+\frac{\mathbf{1}_{|u|>R}}{u}\right) du
\\ & = \rho_{sc}(X)\int_{\mathbb{R}\setminus(-R,R)}\left(\frac{1}{z-u}+\frac{1}{u}\right) du + \rho_{sc}(X)\int_{-R}^{R}\left(\frac{1}{z-u}\right) du
\\ & = \lim_{u\to\infty} \rho_{sc}(X)\log_{\mathbb{H}}{\left(\frac{u + z}{-u + z}\right)} = -i\pi \rho_{sc}(X).
\end{align*}
Therefore, we are justified in taking $R>0$ to be arbitrarily large, so that the multiplicative factor $\Im(z)^{-1}+R^{-1}$ is simply proportional to $\Im(z)^{-1}.$

To prove the second statement, we examine the sum term which appears in $S_{m}(z)$,
    \begin{align}\label{e:threeint}\begin{split}
        \frac{1}{T}\sum_{r=1}^{m}\log_{\mathbb{H}}{\left(z-u_{r}^{(m)}\right)} & = \int_{-\infty}^{\infty}\log_{\mathbb{H}}{(z-u)}\mu_{m}(du)
        \\ & =\int_{-\infty}^{\infty}\left(\log_{\mathbb{H}}{(z-u)}-\log_{\mathbb{H}}{(i-u)}+(z-i)\frac{\mathbf{1}_{|u|>R}}{u}\right)\mu_{m}(du) 
        \\ & + \int_{-\infty}^{\infty}\log_{\mathbb{H}}{(i-u)}\mu_{m}(du) - \int_{-\infty}^{\infty}(z-i)\frac{\mathbf{1}_{|u|>R}}{u}\mu_{m}(du) . 
        \end{split}
    \end{align} We define $C_{m}  := -\int_{-\infty}^{\infty}\log_{\mathbb{H}}{(i-u)}\mu_{m}(du).$ The integrand of the first term decays like $u^{-2}$ at infinity, and thus, applying \cref{t:local} $(1),$ and the same truncation argument as before, we find that there exists $C:=C(\alpha)>0$ such that for all $x^{(m)}\in\Omega_{m}$ and $X\in (-2+\alpha,2-\alpha),$
    \begin{multline*}
        \big|\int_{-\infty}^{\infty}\left(\log_{\mathbb{H}}{(z-u)}-\log_{\mathbb{H}}{(i-u)}+(z-i)\frac{\mathbf{1}_{|u|>R}}{u}\right)\mu_{m}(du)
        \\ -\int_{-\infty}^{\infty}\left(\log_{\mathbb{H}}{(z-u)}-\log_{\mathbb{H}}{(i-u)}+(z-i)\frac{\mathbf{1}_{|u|>R}}{u}\right)\rho_{sc}(X+uT/m)du\big| 
        \\ \leq C|z-i| (\Im(z)^{-1/2}+R^{-1})\log^{2}{(m)}T^{-1}. 
    \end{multline*} Combining this with \cref{t:local} $(2)$, and possibly using a distinct $C>0,$ we conclude that for all $x^{(m)}\in\Omega_{m}$ 
    \begin{multline*}
        \big|S_{m}(z)-C_{m} - \frac{1}{2}X(z+i) - \log_{\mathbb{H}}{(z)}
        \\ + \int_{-\infty}^{\infty}\left(\log_{\mathbb{H}}{(z-u)}-\log_{\mathbb{H}}{(i-u)}+(z-i)\frac{\mathbf{1}_{|u|>R}}{u}\right)\rho_{sc}(X)du
        \big|
        \\ \leq C|z-i| (\Im(z)^{-1/2}+R^{-1})\log^{2}{(m)}T^{-1}+|z|^{2}Tm^{-1}.
    \end{multline*} Finally, we evaluate the integral
    \begin{align*}
        \int_{-\infty}^{\infty}\left(\log_{\mathbb{H}}{(z-u)}-\log_{\mathbb{H}}{(i-u)}+(z-i)\frac{\mathbf{1}_{|u|>R}}{u}\right)\rho_{sc}(X)du =-i\pi(z-i)\rho_{sc}(X).
    \end{align*} Combining the estimates for $G_{R,m}(z)$ and $d_{m}(R)$ removes the cutoff from the deterministic main term, and the remaining dependence is $O(R^{-1})$. Thus, we can choose $R>0$ arbitrarily large so that the factor of $\Im(z)^{-1/2}+R^{-1}$ is simply proportional to $\Im(z)^{-1/2}.$
\end{proof}

\subsection{Critical Points}\label{s:zeros} 
In this section, we show that for all $x^{(m)}\in\Omega_{m}$, there is a unique critical point in the strict upper half-plane. As in \cref{l:sconv}, for the rest of \cref{s:convrate} we will use the notation $e_{m}:=\log^{2}{(m)}T^{-1} +Tm^{-1}$.
\begin{lemma}\label{l:unique} Fix $\alpha\in(0,2)$. There exists $m_{0}:=m_{0}(\alpha)\in\mathbb{N}$ such that for all $m>m_{0}$, $x^{(m)}\in\Omega_{m}$, and $X\in (-2+\alpha,2-\alpha),$ there exists a unique zero $z_{0}^{(m)}$ of $S_{m}'(z)$ such that $\Im(z_{0}^{(m)})>0$. Furthermore, for all $\delta>0$ there exists $m_{\delta}:=m_{\delta}(\delta,\alpha)\in\mathbb{N}$ such that for all $m>m_{\delta}$, $x^{(m)}\in\Omega_{m}$, and $X\in (-2+\alpha,2-\alpha)$, $|z_{0}^{(m)}-z_{0}|<\delta$, where $z_{0}:=-X/2 +i\sqrt{1-(X/2)^{2}}$. 
\end{lemma}
We begin by showing that there can be at most one critical point in the strict upper half-plane. Then, we will prove that such a critical point exists.
\begin{lemma}\label{l:maximaginary} For all $x^{(m)}\in\Omega_{m}$ and $X\in\mathbb{R}$, the equation $S_{m}'(z)=0$ has at most two zeros, counted with multiplicity, which do not lie on the real line. 
\end{lemma}
\begin{proof} By our definition of $\Omega_{m}$, for all $x^{(m)}\in\Omega_{m}$, all of the eigenvalues are distinct. 

     Finding the zeros of $S'_{m}(z)$ is, away from the poles, equivalent to the problem of finding the zeros of a polynomial with leading term $z^{m+2}$, which therefore has $m+2$ roots, counted with multiplicity, over the complex plane. In this argument, we return to the macroscopic coordinate $X+\frac{T}{m}z\mapsto z$.
    Under this change in coordinates, the problem is equivalent to finding $z$ such that 
     \begin{align}\label{e:solns}
         z & = \frac{1}{m}\sum_{r=1}^{m}\frac{1}{z-x_{r}^{(m)}/\sqrt{m}} - \frac{T}{m}\cdot \frac{1}{z-X} =:f(z) =f_{X}(z).
     \end{align} 
     Along the real axis in $\mathbb{C}$, $f(z)$ is a real-valued function (except at the locations of the poles). The function $f(z)$ is the sum of $m+1$ functions which have a single simple pole and which are otherwise are holomorphic on the rest of $\mathbb{C}$.

     This argument is conducted in the case the $x_{j}^{(m)}/\sqrt{m}\neq X$ for all $j\in [[m]]$. If $x_{j}^{(m)}/\sqrt{m}=X$ for some $j\in[[m]]$, we consider a sequence $X^{(\varepsilon)}\to X$ as $\varepsilon\to 0$. On compact subsets of $\mathbb{C}\setminus \mathbb{R}$,  functions $z-f_{X^{(\varepsilon)}}(z)\to z- f_{X}(z)$ converge uniformly as $\varepsilon\to 0$ (and therefore, by Hurwitz' theorem the number of zeros in that set cannot increase in the limit). Therefore, the same bound that we prove below will also hold in this degenerate case. 
      
     Along the real axis, we will call a pole ``positive'' if the additive term corresponding to that pole, of the form $c(z-a)^{-1}$ for $c,a\in\mathbb{R}$, is negative for all $x<a$ and positive for all $x>a$. When the signs are reversed, we call a pole of that form ``negative.''

     We consider the right-hand side of this equation and note that it has $m$ positive poles at locations $z=x_{r}^{(m)}/\sqrt{m}$, 
     and one negative pole at $z=X$. 
     On the line segment with $y=0$ between any two positive poles, $f(z)$ takes all values between $-\infty$ and $\infty$. Therefore, on any interval of this type, there must be at least one solution to \eqref{e:solns}. We call this a solution of type $*$. Recall that the $x_{i}^{(m)}$ are ordered so that $x_{i}^{(m)}\geq x_{j}^{(m)}$ when $i<j$. We consider three cases:
     \begin{enumerate} 
        \item Case that $X\in \left(x_{m}^{(m)}/\sqrt{m},x_{1}^{(m)}/\sqrt{m}\right)$: There are at least $m-2$ solutions of type $*$. To the right of the pole at $x_{1}^{(m)}/\sqrt{m}$,  is a region in $x$ where $f(z)$ takes all values $(0,\infty)$. To the left of the pole at $x_{m}^{(m)}/\sqrt{m}$ is a region in $x$ where $f(z)$ takes all values $(-\infty,0)$. This guarantees at least two additional solutions to \eqref{e:solns}, bringing the total number of solutions up to at least $m$. 
            \item ($X<x_{m}^{(m)}/\sqrt{m}$): There are at least $m-1$ solutions of type $*$. To the right of the pole at $x_{1}^{(m)}/\sqrt{m}$ there is a region in $x$ where $f(z)$ takes all values in $(0,\infty)$, which guarantees at least an additional solution to \eqref{e:solns}, bringing the total number of solutions up to at least $m$.
            \item ($x_{1}^{(m)}/\sqrt{m}<X$): There are at least $m-1$ solutions of type $*$. To the left of the pole $x_{m}^{(m)}/\sqrt{m}$ there is a region in $x$ where $f(z)$ takes all values in $(-\infty,0)$. This latter region will necessarily contribute at least one solution of \eqref{e:solns}, which brings the total number of solutions up to at least $m$. 
        \end{enumerate} The coefficients of the equation \eqref{e:solns} are all real, meaning that any zero with a nonzero complex part must come with a complex conjugate zero. In all of the cases above, there are at most two total zeros remaining, making at most one pair of complex conjugate zeros. 
\end{proof}
Using this result, we can prove \cref{l:unique}.
\begin{proof}[Proof of \cref{l:unique}]
For $z_{0}=-X/2 +i\sqrt{1-(X/2)^{2}}$ as defined in \cref{l:unique}, we note that $S_{*}'(z_{0})=0$. By Hurwitz's theorem and \cref{l:sconv}, for all $\delta>0$ there exists $m_{\delta}\in\mathbb{N}$ such that for all $m>m_{\delta}$ and $x^{(m)}\in\Omega_{m}$, there is exactly one zero $z_{0}^{(m)}$ of multiplicity $1$ in the disc defined by $|z-z_{0}|<\delta$. By \cref{l:maximaginary}, this $z_{0}^{(m)}$ is the unique zero of $S_{m}'(z)$ in the upper half-plane.
\end{proof}

\subsection{Assumptions of \cref{p:steepest}} \label{s:contours} 

In this section, we study the steepest descent contours of $S_{m}(z)$ within the region $\mathcal{R}(c,\kappa)$. We will verify the assumptions of \cref{p:steepest}, roughly in the order they are stated.

We begin by commenting on the assumptions which follow immediately from what we have done so far, without further proof:
\begin{itemize}
    \item  Under our choice of branch cut, $S_{m}(z)$ and $R_{m}(w,z)$ are holomorphic on the strict upper half-plane, and therefore also for all $z,w\in \mathcal{R}(c,\kappa)$. 
    \item For $\kappa,c>0$ sufficiently small, this region must contain $z_{0}$, and therefore, for $m$ sufficiently large, must contain $z_{0}^{(m)}$. Furthermore, $z_{0}^{(m)}$ is the unique critical point in $\mathcal{R}(c,\kappa)$ by \cref{l:unique}. 
    \item We have also shown that the $z_{0}^{(m)}$ converge to $z_{0}$ as $m\to\infty$. 
\end{itemize}

Next, we will verify the assumptions of \cref{p:steepest} $(1)$ and show that for $m$ sufficiently large, the critical point is nondegenerate. 

\begin{proposition}\label{l:conditions1} Fix $\alpha\in (0,2)$. There exist $c_{1},c_{2},c_{3},r>0$ and $m_{0}\in\mathbb{N}$, depending only on $\alpha$, such that for all $m\geq m_{0}$, $z\in B_{r}(z_{0})$, $x^{(m)}\in \Omega_{m}$, and $X\in (-2+\alpha,2-\alpha),$ $c_{1}<|S_{m}''(z)|<c_{2}$ and $|S_{m}'''(z)|\leq c_{3}$, and likewise for $S_{*}''(z),S_{*}'''(z)$. 
\end{proposition}
\begin{proof}
     Since $S'_{m},S'_{*}$ are holomorphic in the upper half-plane, we can apply \cref{l:sconv} on a compact region such as $\overline{\mathcal{R}(c,\kappa)}$ to  conclude that $S''_{m}(z)\to S_{*}''(z)$ and $S'''_{m}(z)\to S_{*}'''(z)$ uniformly along any sequence of $x^{(m)}$ with  $x^{(m)}\in\Omega_{m}$ at each step. 

     We first establish that these estimates hold at $z_{0}^{(m)}$. Suppose, by way of contradiction, that there exist sequences of constants $\{c_{1,m}\}_{m\in\mathbb{N}}$ and $\{c_{2,m}\}_{m\in\mathbb{N}}$ such that $c_{1,m}\to 0$ and $c_{2,m}\to\infty$ as $m\to\infty$ and such that (a) $|S_{m}''(z_{0}^{(m)})|<c_{1,m}$ or (b) $|S_{m}''(z_{0}^{(m)})|>c_{2,m}$ for some $x^{(m)}\in\Omega_{m}$ at infinitely many $m\in\mathbb{N}$. 
    
    In case $(a)$, it would follow that there exists a subsequence of $m$ and $x^{(m)}$ along which $S''_{*}(z_{0}^{(m)})\to 0$, in which case (relying also on the convergence of $z_{0}^{(m)}$ to $z_{0}$ from \cref{l:unique}) the order of the critical point $z_{0}$ would need to be at least two. By Hurwitz's theorem, this would imply that for all $m$ sufficiently large along this subsequence, there would exist two critical points (counted with multiplicity) of $S_{m}(z)$ in the upper half-plane, which contradicts \cref{l:unique}. In case $(b)$, the contradiction is immediate from the fact that $S_{*}''$ is bounded around $z_{0}$, and the uniform convergence in a compact neighborhood in the upper half-plane from \cref{l:sconv}. 

    Finally, since $S_{*}''$ is continuous and nonzero at $z_{0}$, and due to the uniform convergence of $S''_{m}$ to $S''_{*}$ from \cref{l:sconv}, and the uniformity of the constants over $X\in (-2+\alpha,2-\alpha),$ 
    we can conclude that there exists $c_{1},c_{2},r_{1}>0$ and $m_{1}\in\mathbb{N}$ such that for all $m>m_{0}$, $x^{(m)}\in\Omega_{m}$ and $z$ such that $|z-z_{0}|\leq r_{1}$, $c_{1}< |S_{m}''(z)|< c_{2}$. Since $S_{*}'''$ is bounded around $z_{0}$, and we likewise have uniform convergence on a compact region containing $z_{0}$, we can also conclude that there exists $c_{3},r_{2}>0$ and $m_{2}\in\mathbb{N}$ such that for all $m>m_{2}$ and $x^{(m)}\in\Omega_{m}$, for all $|z-z_{0}|<r_{2},$ $|S_{m}'''(z)|\leq c_{3}$. To conclude, we take $m_{0}=\max\{m_{1},m_{2}\}$ and $r=\min\{r_{1},r_{2}\}$.
\end{proof}
    
    \begin{remark}\label{r:c1c2}
        As a consequence of \cref{l:conditions1}, we can take a Morse chart on an open set $U_{m}$ containing $z_{0}^{(m)}$ and define contours $C_{1}^{(m)},C_{2}^{(m)}\subset \mathcal{R}(c,\kappa)$ (which implicitly depend on $c,\kappa$, though we suppress this in the notation) by setting them equal to, respectively, the steepest descent and ascent branches of $S_{m}(z)$ at $z_{0}^{(m)}$ in $U_{m}$, and then extending to the region $\mathcal{R}(c,\kappa)\setminus U_{m}$ by continuing respectively along the integral curves of $-\nabla \Re(S_{m}(z))$ and $\nabla \Re(S_{m}(z))$ and taking the intersection with $\mathcal{R}(c,\kappa)$ (or the connected component of that intersection which includes $z_{0}^{(m)}$ if there is more than one). We note that by 
        \cref{l:unique} and \cref{l:conditions1}, these charts can be chosen uniformly for all $X\in (-2+\alpha,2-\alpha)$. By the quantitative Morse lemma (for instance, see \cite[Theorem B.1]{DGBLR25}), we can choose the neighborhoods $U_{m}$ around $z_{0}^{(m)}$ so that the Morse chart around $z_{0}^{(m)}$ covers $U_{m}$, and such that there exists $m_{0}\in\mathbb{N},\kappa_{0},c\in (0,1)$ and an open set $U$ such that for all $m>m_{0}$ and $\kappa\in (0,\kappa_{0})$, $U\subset U_{m}\subset\mathcal{R}(c,\kappa)$ and $z_{0}^{(m)}\in U$. 
    \end{remark}

    Next, we verify the assumptions of \cref{p:steepest} $(2)$. In the next remark, we discuss an upper bound for $R_{m}(z,w)$ over $z,w\in\mathcal{R}(c,\kappa)$. Except for this remark, throughout this section all constants are independent of $n_{1},n_{2},x_{1},x_{2}$, since the kernel in \cref{t:convrate} depends on those values only through the configuration of the original contours and through $R_{m}(z,w)$. 
    \begin{remark}\label{r:Rbd} Let $D_{c}:=\sqrt{1+c^{-2}}$. For all $x_{1},x_{2}\in\mathbb{R}$ and $n_{1},n_{2}\in\mathbb{N}$, whenever $T\kappa\geq 2\max\{|x_{1}|,|x_{2}|,1\}$ and $T\geq \max\{n_{1},n_{2}\}$,
\begin{align*}\sup_{z,w\in\mathcal R(c,\kappa)}
|R_m(z,w)|
\le
D_c^{\,n_2-1}
\kappa^{-n_1-n_2}
e^{4(|x_1|+|x_2|)\kappa^{-1}}.
\end{align*}
  Furthermore, if $A_{m}(w):= w^{T}\left(w-\frac{x_{1}}{T}\right)^{-n_{1}-T-1}$ and $B_{m}(z)= z^{-T}\left(z-\frac{x_{2}}{T}\right)^{n_{2}+T-1}$, then, under the same assumption on $T\kappa$, 
  \begin{align*} 
  \sup_{w\in \mathcal{R}(c,\kappa)}|A_{m}(w)|\leq e^{4\kappa^{-1}|x_{1}|}\kappa^{-n_{1}-1} & & \sup_{z\in \mathcal{R}(c,\kappa)}|B_{m}(z)|\leq D_{c}^{n_{2}-1}e^{4\kappa^{-1}|x_{2}|}\kappa^{-n_{2}+1} \end{align*} 

    By similar arguments, when we allow $z,w\in U$ only, and note that $U$ is contained within a neighborhood of $z_{0}$ of fixed constant radius, then within that region $R_{m}(z,w)$ there exists a constant $C>0$ which is independent of $c,\kappa,$ (and can be chosen uniformly over $X\in (-2+\alpha,2-\alpha)$).  
\end{remark}
 The next remark verifies another assumption of \cref{p:steepest} $(2)$.
 \begin{remark}\label{r:extra} It is an immediate consequence of \cref{l:conditions1} and the definition of the contours that there exists $d:=d(\alpha)>0$ such that for all $m>m_{0}$, $x^{(m)}\in\Omega_{m},$ $X\in (-2+\alpha,2-\alpha)$, and $z\in C_{1}^{(m)}\setminus U,w\in C_{2}^{(m)}\setminus U$,
         \begin{align*}
        \Re\left(S_{m}(z_{0}^{(m)})-S_{m}(z)\right), \Re\left(S_{m}(w) - S_{m}(z_{0}^{(m)})\right)\geq d.
        \end{align*}
 \end{remark} 
 The next lemma also verifies an assumption of \cref{p:steepest} $(2)$. 
\begin{lemma}\label{l:lengthbound} Fix $\alpha\in (0,2),$ $c\in (0,1)$, and $\kappa\in (0,1)$. 
    There exists $C:=C(c,\alpha)>0$  and $m_{0}:=m_{0}(c,\kappa,\alpha)\in\mathbb{N}$ such that for all $m>m_{0}$, $x^{(m)}\in\Omega_{m}$, and $X\in (-2+\alpha,2-\alpha),$ the lengths of the contours $C_{1}^{(m)},C_{2}^{(m)}$ are uniformly bounded by $L(c,\kappa):=C\kappa^{-2}$.
\end{lemma}
\begin{proof} For $m_{0}\in\mathbb{N}$ as in the definition of the contours, so that $U\subset U_{m}$ for all $m>m_{0}$, we separately consider the parts of $C_{1}^{(m)},C_{2}^{(m)}$ which lie inside $U$ and the parts which lie in $\mathcal{R}(c,\kappa)\setminus U$. In the former, the uniform boundedness of the contours for all $m>m_{0}$ is guaranteed by the quantitative Morse lemma.

%By the fact that $\overline{\mathcal{R}(c,\kappa)}\setminus U$ is compact and contains no critical points and by the fact that $|S''_{*}(z)|$ is bounded away from zero for $z\in U$, there exists 
Since $S'_{*}(z)$ has a unique zero in the upper half-plane at $z_{0}(X)$, there is a $\delta:=\delta(c,\alpha)>0$ such that $\inf_{X\in[-2+\alpha,2-\alpha]}\inf_{z\in \overline{\mathcal{R}(c,\kappa)}\setminus U}|S_{*}'(z)|>\delta$. Furthermore, by the uniform convergence of the $S'_{m}(z)$ to $S_{*}'(z)$ on compact sets, there exists $\delta':=\delta'(c,\alpha)>0$ such that for all $m>m_{0}$, 
\begin{align*}
    \inf_{X\in [-2+\alpha,2-\alpha]}\inf_{z\in\overline{\mathcal{R}(c,\kappa)}\setminus U} |\nabla \Re(S_{m}(z))|>\delta.
\end{align*}
We write $C_{1}^{(m)}\setminus U,C_{2}^{(m)}\setminus U$ as the union of two curves each, defining $\gamma_{i,j}^{(m)}$ for $i,j\in\{1,2\}$ so that $\gamma_{1,1}^{(m)}\cup\gamma_{1,2}^{(m)}=C_{1}^{(m)}\setminus U,$ and $\gamma_{2,1}^{(m)}\cup\gamma_{2,2}^{(m)}=C_{2}^{(m)}\setminus U.$ We assume that each of these curves is parameterized by its arc length, so that $(\gamma_{i,j}^{(m)})'(t)=\pm \nabla \Re(S_{m}(\gamma_{i,j}^{(m)}(t)) )|\nabla \Re(S_{m}(\gamma_{i,j}^{(m)}(t))|^{-1}.$

For arbitrary $i,j\in\{1,2\}$ we let $\ell$ denote the length of $\gamma_{ij}^{(m)}$. We compute 
\begin{align*}
    \ell \delta\leq \int_{0}^{\ell }\big|\nabla\Re(S_{m}(\gamma_{i,j}^{(m)}(t))\big|dt & = \int_{0}^{\ell }\big|\frac{d}{dt}\Re(S_{m}(\gamma_{i,j}^{(m)}(t))\big|dt 
    \\ & \overset{*}{=} |\Re(S_{m}(\gamma_{i,j}^{(m)}(\ell)))-\Re(S_{m}(\gamma_{i,j}^{(m)}(0)))| 
    \\ & \leq \max_{z_{1},z_{2}\in\mathcal{R}(c,\kappa)}|z_{1}-z_{2}|\sup_{z\in\mathcal{R}(c,\kappa)}|S'_{m}(z)|
     \leq C\kappa^{-1}\sqrt{4(c\kappa)^{-2}+\kappa^{-2}}.
\end{align*} The equality $*$ uses the monotonicity of $\Re(S_{m})$ along a steepest descent or ascent curve. Due to the uniform convergence over $\mathcal{R}(c,\kappa)$ from \cref{l:sconv}, the right-hand side is uniformly bounded for all fixed $\kappa\in(0,1)$. This calculation holds equally well for each of the $i,j\in\{1,2\}$. 
\end{proof}
The next remark discusses the last remaining assumption required in \cref{p:steepest} $(2)$.
\begin{remark}\label{r:extra2} We established as part of \cref{r:extra} that there exists $\delta>0$ such that for all $m>m_{0}$, 
\begin{align*}
    \min_{z\in\overline{\mathcal{R}(c,\kappa)}\setminus U} |\nabla \Re(S_{m}(z))|>\delta.
\end{align*} Furthermore, by \cref{l:conditions1} and the quantitative Morse lemma, we can conclude that when we evaluate $\Re(S_{m}(z))$ at the points in $C_{1}^{(m)}\cap\partial U,C_{2}^{(m)}\cap \partial U$, there is a uniform constant order discrepancy of the value of $\Re(S_{m}(z))$ at each $z\in (C_{1}^{(m)}\cap\partial U)\cup(C_{2}^{(m)}\cap \partial U)$ and its closest neighbors in that set. 

Therefore, by the monotonicity of $\Re(S_{m}(z))$ along the steepest descent and ascent contours, there exists $c:=c(\alpha,c)>0$ such that 
\begin{align*}
    \inf_{m>m_{0}}\inf_{\substack{z\in C_{1}^{(m)}\setminus U \\ w\in C_{2}^{(m)}\setminus U}}|\Re(S_{m}(z))-\Re(S_{m}(w))| >c.
\end{align*}
Consequently, appealing to the fact that $|\Re(S_{m}(z))-\Re(S_{m}(w))|\leq |z-w|\sup_{z\in\mathcal{R}(c,\kappa)}|S_{m}'(z)|$, and the uniform convergence of $S_{m}'(z)$ to $S_{*}'(z)$ for $z\in\mathcal{R}(c,\kappa)$ following from \cref{l:sconv}, we conclude that there exists $c':=c'(\alpha,c)>0$ such that 
\begin{align*} \inf_{m>m_{0}}\min\left\{\inf_{\substack{z\in C_{1}^{(m)}\setminus U \\ w\in C_{2}^{(m)}\setminus U} }|w-z|,\inf_{\substack{z\in C_{1}^{(m)}\cap U \\ w\in C_{2}^{(m)}\setminus U} }|w-z|,\inf_{\substack{z\in C_{1}^{(m)}\setminus U \\ w\in C_{2}^{(m)}\cap U} }|w-z|\right\}>c'\kappa.
\end{align*} In particular, to obtain the latter two bounds (on the mixed terms), we use the quantitative Morse lemma \cite[Theorem B.1]{DGBLR25} and analyze those expressions using the local chart.
This verifies another assumption which is required for \cref{p:steepest} $(2)$. 
\end{remark}

\subsection{Contour Configuration}\label{s:contours2} In this section we study the locations of the points where the steepest descent contours of $S_{m}(z)$ intersect with $\partial\mathcal{R}(c,\kappa)$. Through this section, we will use the notation $e_{m}:=\log^{2}{(m)}T^{-1} +Tm^{-1}$. 

\begin{proposition}\label{l:conditions3} Fix $\alpha\in(0,2)$. There exists $\kappa_{0}>0$ such that for all $\kappa\in (0,\kappa_{0}),$ there exists $m_{0}:=m_{0}(\kappa,\alpha)\in\mathbb{N}$ and $c_{0}:=c_{0}(\alpha)\in (0,1)$ and constants $c,C,R>0$, which depend on $\alpha$, such that for all $m>m_{0}$, $x^{(m)}\in\Omega_{m}$, and $X\in (-2+\alpha,2-\alpha),$ the following hold:
    \begin{enumerate}
        \item The set $\partial \mathcal{R}(c_{0},\kappa)\cap (C_{1}^{(m)}\cup C_{2}^{(m)})$ consists of four points. Three of these points lie on the line $y=\kappa$ and one lies along the line $y=\kappa^{-1}$.  These points can be chosen and labeled in a clockwise order $z_{1}^{(m)},z_{2}^{(m)},z_{3}^{(m)},z_{4}^{(m)}$, so that $z_{4}^{(m)}$ lies along $y=\kappa^{-1} $ and such that $$\Re(S_{m}(z_{1}^{(m)})),\Re(S_{m}(z_{3}^{(m)}))> \Re(S_{m}(z_{0}^{(m)}))>\Re(S_{m}(z_{2}^{(m)})),\Re(S_{m}(z_{4}^{(m)})).$$  
        \item The three points which intersect $y=\kappa$ satisfy 
        \begin{align*}
            |\Re(z_{2}^{(m)})|\le R\kappa,
                & &
            -C\leq \Re(z_{1}^{(m)}) \leq -c\kappa^{1/2}, & & c\kappa^{1/2}\leq \Re(z_{3}^{(m)})\le C.
        \end{align*} Consequently, for any $x_{1},x_{2}\in\mathbb{R}$, and $i\in \{1,2\}$, 
        \begin{align*}
            |x_{i}T^{-1}-\Re(z_2^{(m)})|
            \le |x_{i}|T^{-1}+R\kappa.
    \end{align*}
    \end{enumerate}
    \end{proposition}
    The argument follows the framework established by Gorin and Petrov \cite[Section 6]{GP19}, with significant simplifications. In order to prove \cref{l:conditions3}, we will need several estimates on the real and imaginary parts of $S_{m}(z)$. The next lemma plays a similar role to \cite[Lemma 6.1, Lemma 6.2, Lemma 6.4]{GP19}.
\begin{proposition}\label{l:6162} Fix $\alpha\in(0,2)$ and $c\in (0,1),$
    \begin{enumerate}
        \item There exists $C:=C(\alpha)>0$ such that for all $\kappa\in(0,1)$ there exists  $m_{0}:=m_{0}(c,\kappa)\in\mathbb{N}$ such that for all $x+iy\in\mathcal{R}(c,\kappa),$ $m>m_{0}$, $x^{(m)}\in\Omega_{m}$, and $X\in (-2+\alpha,2-\alpha)$,
        \begin{align*}
            \big|\Im(S_{m}(x+iy))-\Im(S_{m}(x))\big|  \leq  C(y+\arctan(y|x|^{-1})+T|xy|m^{-1}+e_{m}).
        \end{align*}
        We take the convention that at $x=0$ and for all $y>0$,  $\arctan(y/|x|)=\pi/2$. 
        \item For every $R>0$ and $x\in\mathbb{R}$ there  exists $m_{0}\in\mathbb{N}$ such that for all $m>m_{0}$, $x^{(m)}\in\Omega_{m},$ and $X\in (-2+\alpha,2-\alpha)$,
        \begin{align*}
            \Im(S_{m}(x +R))-\Im(S_{m}(x ))>\pi\rho_{sc}(X)R -\pi-1.
        \end{align*} 
        \item There exists $\kappa_{0}>0$ such that for all $\kappa\in (0,\kappa_{0})$ and $y\geq \kappa^{-1}$ there exists $d:=d(\alpha)>0$ and  $m_{0}:=m_{0}(c,\kappa,d)$ such that for all $m>m_{0}$ $|x|\leq (c\kappa)^{-1}$, $x^{(m)}\in\Omega_{m}$, and $X\in (-2+\alpha,2-\alpha)$,
        \begin{align*}
            -\frac{\partial}{\partial y}\Re(S_{m}(x+iy))=\frac{\partial}{\partial x}\Im(S_{m}(x+iy))>d.
        \end{align*}
    \end{enumerate}
\end{proposition}

\begin{proof} 
    To show part $(1)$, we see that 
    \begin{multline*}
        \Im(S_{m}(x+iy)) -\Im(S_{m}(x)) 
        \\ = yX + \arctan{\left(\frac{y}{x}\right)}  + \frac{Txy}{m} -\frac{1}{T}\sum_{r=1}^{m}\left(\Im\left(\log_{\mathbb{H}}(x+iy-u_{r}^{(m)})\right)-\Im\left(\log_{\mathbb{H}}(x-u_{r}^{(m)})\right)\right). 
    \end{multline*} 
    To deal with the sum, we note 
    \begin{align*}
        \frac{1}{T}\sum_{r=1}^{m}\left(\Im\left(\log_{\mathbb{H}}(x+iy-u_{r}^{(m)})\right)-\Im\left(\log_{\mathbb{H}}(x-u_{r}^{(m)})\right)\right) 
        & = \frac{1}{T}\sum_{r=1}^{m}\arctan{\left(\frac{y}{x-u_{r}^{(m)}}\right)}.
    \end{align*}
    Thus, taking $d_{m}(R)$ defined as in \cref{t:local}, we can write
    \begin{align*}
        \Im(S_{m}(x+iy)) -\Im(S_{m}(x)) & = yX + \arctan{\left(\frac{y}{x}\right)}  + \frac{Txy}{m} 
        \\ & -\frac{1}{T}\sum_{r=1}^{m}\left(\arctan{\left(\frac{y}{x-u_{r}^{(m)}}\right)}+\frac{y\mathbf{1}_{|u_{r}^{(m)}|>R}}{u_{r}^{(m)}}\right) - yd_{m}(R).
    \end{align*} 
    Therefore, by \cref{t:local}, for all $x^{(m)}\in\Omega_{m}$, there exists a constant $C:=C(\alpha)>0$ such that
    \begin{align*}
        \Im(S_{m}(x+iy)) -\Im(S_{m}(x)) & = y\frac{X}{2} + \arctan{\left(\frac{y}{x}\right)}  + \frac{Txy}{m} + (1+yR^{-1})O(\log^{2}{(m)}T^{-1}+Tm^{-1}).
    \end{align*} 
    Therefore, taking $R>0$ to be arbitrarily large, there exists $C:=C(\alpha)>0$ and $m_{0}:=m_{0}(c,\kappa)\in\mathbb{N}$ such that for all $m>m_{0}$, $x+iy\in\mathcal{R}(c,\kappa)$ and $x^{(m)}\in\Omega_{m}$, 
    \begin{align*}
        \big|\Im(S_{m}(x+iy)) -\Im(S_{m}(x))\big| & \leq C(y+\arctan(y|x|^{-1})+e_{m} +T|xy|m^{-1}),
    \end{align*} where $e_{m}:=\log^{2}{(m)}T^{-1} +Tm^{-1}$ as in previous results.
    
    To show part $(2)$, we see that 
    \begin{align*}
        \Im(S_{m}(x_{1}+R))-\Im(S_{m}(x_{1})) & = \pi\left(\mathbf{1}_{x_{1}+R<0} -\mathbf{1}_{x_{1}<0}\right)  + \frac{\pi}{T}\sum_{r=1}^{m}\mathbf{1}_{x_{1}<u_{r}^{(m)}<x_{1}+R}.
    \end{align*}
    By \cref{thm:rigid}, for all $x^{(m)}\in\Omega_{m}$, there exists $C:=C(\alpha)>0$  such that 
    \begin{align*}
        \bigg|\frac{\pi}{T}\sum_{r=1}^{m}\mathbf{1}_{x_{1}<u_{r}^{(m)}<x_{1}+R} - \pi\rho_{sc}(X)R \bigg| \leq  Ce_{m}. 
    \end{align*}
    Therefore, 
    \begin{align*}
        \Im(S_{m}(x_{1}+R))-\Im(S_{m}(x_{1})) & =  \pi\rho_{sc}(X)R - \pi \mathbf{1}_{x_{1}<0<x_{1}+R}+Ce_{m}. 
    \end{align*}
    For all $R>0$, there exists $m_{0}:=m_{0}(\alpha)\in\mathbb{N}$ such that for all $m>m_{0},$ $Ce_{m}$ is of magnitude less than $1$. Thus, for all $x_{1}\in\mathbb{R}$, and $m>m_{0}$,
    \begin{align*}
        \Im(S_{m}(x_{1}+R))-\Im(S_{m}(x_{1})) & \geq   \pi\rho_{sc}(X)R -\pi-1.
    \end{align*}
    
    To show part $(3)$, it follows directly from \cref{l:sconv} that for all $m$ sufficiently large and $x^{(m)}\in \Omega_{m},$ 
     \begin{align*}
        \frac{\partial}{\partial x}\Im(S_{m}(x+iy)) = \pi\rho_{sc}(X)-\frac{y}{x^{2}+y^{2}} + \frac{Ty}{m} + O(y^{-1}e_{m}).
    \end{align*} The estimates in the statement of part $(3)$ follow immediately since $\pi\rho_{sc}(X)>0$ uniformly for all $X\in (-2+\alpha,2-\alpha)$.
\end{proof}
Finally, we can prove \cref{l:conditions3}. 
    \begin{proof}[Proof of \cref{l:conditions3}]
        To show $(1)$, we note that because the critical point $z_{0}^{(m)}$ is unique in the upper half-plane, it is impossible for there to be another point in $\overline{\mathcal{R}(c,\kappa)}$ with $S_{m}'(z)=0$ and consequently there can be no additional points where the contours $C_{1}^{(m)},C_{2}^{(m)}$ intersect. Furthermore, by \cref{l:lengthbound}, the contours have uniformly bounded length. Therefore, both ends of $C_{1}^{(m)},C_{2}^{(m)}$ must intersect $\partial\mathcal{R}(c,\kappa)$ or form an accumulation point, which would also be a disallowed critical point, and therefore a contradiction. 

        Next we study the location of these intersection points. 
        
        By \cref{l:6162} $(3)$, there exists $\kappa_{0}>0,$ such that for all $\kappa\in (0,\kappa_{0})$ there exists $m_{0}:=m_{0}(\kappa)\in\mathbb{N}$ such that for all $m>m_{0}$ and $x^{(m)}\in\Omega_{m}$, $\Im(S_{m}(z))$ is strictly increasing in the positive real direction along $y=\kappa^{-1}$, and therefore any contour along which $\Im(S_{m}(z))=\Im(S_{m}(z_{0}^{(m)})$ may intersect this part of $\partial\mathcal{R}(c,\kappa)$ in at most one point. 

        We will adjust $c_{0}$ show that no intersection points can lie on the vertical sides of $\mathcal{R}(c_{0},\kappa)$. As a consequence of \cref{l:6162} $(1)$ and $(2)$, there exist $\kappa_{0}>0$ and $c\in (0,1)$ such that for all $\kappa\in (0,\kappa_{0})$ there exist $m_{0}:=m_{0}(\kappa)\in\mathbb{N}$, such that for all $m>m_{0}$ $y\in[\kappa,\kappa^{-1}]$, and $x^{(m)}\in\Omega_{m}$, 
        \begin{align*}\Im(S_{m}(-(c\kappa)^{-1}+iy))<-2|\Im(S_{m}(z_{0}^{(m)})|, & & 2|\Im(S_{m}(z_{0}^{(m)})|<\Im(S_{m}((c\kappa)^{-1}+iy)).
        \end{align*} Therefore, none of the intersection points can lie on the vertical sides of $\mathcal{R}(c,\kappa)$. Thus either three or four of these points lie on the line $y=\kappa$. By \cref{l:sconv}, for all $x\in [-(c_{0}\kappa)^{-1},(c_{0}\kappa)^{-1}]$, 
        \begin{align}\label{e:xderiv}
            \frac{\partial}{\partial x}\Im(S_{m}(x+i\kappa)) & =\pi\rho_{sc}(X) -\frac{\kappa}{x^{2}+\kappa^{2}}+O(\kappa^{-1}e_{m}).
        \end{align} For fixed $\kappa$,  we take $m$ sufficiently large so that $\kappa^{-1}e_{m}$ is sufficiently small. This function changes sign at most twice, so there can only be a maximum of three intersection points of the steepest descent and ascent contours with $y=\kappa$. 
        
        With this now established, we label the point on the line $y=\kappa^{-1}$ as $z_{4}^{(m)}$ and then proceed clockwise, labeling $z_{1}^{(m)},z_{2}^{(m)},z_{3}^{(m)}$ along $y=\kappa$. By \cref{l:6162} $(3)$, $\Re(S_{m}(x+iy))$ must be decreasing along any  contour that includes $y$ sufficiently large, which means that $z_{2}^{(m)},z_{4}^{(m)}$ lie along the descent contour of $S_{m}(z)$ and $z_{1}^{(m)},z_{3}^{(m)}$ lie along the ascent contour. This finishes the proof of part $(1)$.

        The zeros of \eqref{e:xderiv} are located at 
        \begin{align*}
             a  := -\sqrt{\frac{\kappa}{\pi\rho_{sc}(X)}}\sqrt{1-\pi\rho_{sc}(X)\kappa} - C\kappa^{-1/2}e_{m}, 
              & & b  := \sqrt{\frac{\kappa}{\pi\rho_{sc}(X)}}\sqrt{1-\pi\rho_{sc}(X)\kappa} + C\kappa^{-1/2}e_{m}. 
        \end{align*} Consequently, $\Re(z_{1}^{(m)})\leq a$, $\Re(z_{2}^{(m)})\in [a,b]$ and $\Re(z_{3}^{(m)})\geq b$. We also conclude that for $m$ sufficiently large, there exists $c>0$ such that $a\leq -c\kappa^{1/2}, b\geq c\kappa^{1/2}$.  
        
    We will show that there exists $R>0$ such that $|\Re(z_{2}^{(m)})|\leq R\kappa$. We consider $\Im(S_{*}(x+i\kappa)-\Im(S_{*}(z_{0}))$. If we change coordinates so that $x=\widetilde{x}\kappa$, 
    \begin{align*}
        \Im(S_{*}(\kappa\widetilde{x}+i\kappa) - \Im(S_{*}(z_{0})) = \mathrm{Arg}(\widetilde{x}+i)-\mathrm{Arg}(z_{0})  + \kappa(X/2 +\widetilde{x}\pi\rho_{sc}(X)).
    \end{align*}
    There is a unique value $\widetilde{x}_{0}:=\Re(z_{0})/\Im(z_{0})$ (and a compact set of $\widetilde{x}_{0}$ for $X\in (-2+\alpha,2-\alpha)$) where $ \mathrm{Arg}(\widetilde{x}+i)-\mathrm{Arg}(z_{0})=0$. For some sufficiently small $\delta>0$, we can choose $\widetilde{x}_{-}<\widetilde{x}_{0}<\widetilde{x}_{+}$ so that 
    \begin{align*}
        \mathrm{Arg}(\widetilde{x}_{-}+i)-\mathrm{Arg}(z_{0}) >2\delta & & \mathrm{Arg}(\widetilde{x}_{+}+i)-\mathrm{Arg}(z_{0})<-2\delta.
    \end{align*} There exists $\kappa_{0}>0$ such that for all $0<\kappa<\kappa_{0}$, 
    \begin{align*}
        \Im(S_{*}(\kappa\widetilde{x}_{-}+i\kappa))-\Im(S_{*}(z_{0}))\geq \delta & & \Im(S_{*}(\kappa\widetilde{x}_{+}+i\kappa))-\Im(S_{*}(z_{0}))\leq -\delta 
    \end{align*}
    Therefore for $m$ sufficiently large, 
    \begin{align*}
        \Im(S_{m}(\kappa\widetilde{x}_{-}+i\kappa))-\Im(S_{m}(z_{0}^{(m)}))>0 & & \Im(S_{m}(\kappa\widetilde{x}_{+}+i\kappa))-\Im(S_{m}(z_{0}^{(m)}))<0
    \end{align*} 
    Therefore, there exists $x\in [\kappa\widetilde{x}_{-},\kappa\widetilde{x}_{+}]$, such that $\Im(S_{m}(x+i\kappa))-\Im(S_{m}(z_{0}^{(m)}))=0$, and since any such region lies inside $[a,b]$ for $\kappa$ sufficiently small, we conclude that $|\Re(z_{2}^{(m)})|\leq R\kappa$ for some $R>0$. The choices of $\widetilde{x}_{-},\widetilde{x}_{+},R$ can also be made uniformly over $X\in (-2+\alpha,2-\alpha)$. 

        Finally, to show that there exists $C:=C(\alpha)>0$ such that $|\Re(z_{1}^{(m)})|,|\Re(z_{3}^{(m)})|<C$, we see that up to an error of order $O(\kappa^{-1}e_{m})$, $\Im(S_{m}(x+i\kappa))-\Im(S_{m}(z_{0}^{(m)}))$ is given by $\pi\rho_{sc}(X)x + \frac{X}{2}(1+\kappa)-\Im(S_{*}(z_{0})) +\mathrm{Arg}(x+i\kappa)$.  Since the last term is bounded between $\pi$ and $0$ we conclude that there exists $R>0$ such that this expression is bounded below by $\pi\rho_{sc}(X)x-R$ and above by $\pi\rho_{sc}(X)x+R$. Since $\rho_{sc}(X)$ is bounded uniformly away from $0$ for $X\in (-2+\alpha,2-\alpha)$, this implies that there exists $C:=C(\alpha)>0$ such that 
        \begin{align*}
            \Im(S_{*}(-C+i\kappa))-\Im(S_{*}(z_{0} ))<-2 & & \Im(S_{*}(C+i\kappa))-\Im(S_{*}(z_{0} ))>2
        \end{align*}
        Therefore, for $m$ sufficiently large
         \begin{align*}
            \Im(S_{m}(-C+i\kappa))-\Im(S_{m}(z_{0}^{(m)}))<-1 & & \Im(S_{m}(C+i\kappa))-\Im(S_{m}(z_{0}^{(m)}))>1
        \end{align*}
        We can now use the expression for $\frac{\partial}{\partial x}\Im(S_{m}(x+i\kappa))$ and the bounds above to conclude that for $C,m$ sufficiently large, $\frac{\partial}{\partial x}\Im(S_{m}(x+i\kappa))>0$ for all $|x|>C$. From this, we conclude that $\Re(z_{1}^{(m)}),\Re(z_{2}^{(m)})$, and $\Re(z_{3}^{(m)})$ must all lie within $[-C,C]$. 
    \end{proof}

   We conclude the section by defining the global steepest descent contours. Recalling the definitions of $C_{1}^{(m)},C_{2}^{(m)}$ from \cref{r:c1c2}, we define a contour $C_{1}^{(m),-}$ by reflecting $C_{1}^{(m)}$ across the real axis, and then define 
\begin{align*}
    C_{z}^{(m)} & := C_{1}^{(m)} \cup C_{1}^{(m),-} \cup \{\Re(z_{4}^{(m)})+iy||y|\geq \kappa^{-1}\} \cup v \cup h_{1}\cup h_{2}, 
\end{align*} where $v,h_{1},h_{2}$ are defined
\begin{align}\label{e:vh1h2}\begin{split}
    v & :=\{x_{2}T^{-1}+ it\kappa| t\in [-1,1] \}, 
    \\ h_{1} & :=\{x_{2}T^{-1}(1-t) +t\Re(z_{2}^{(m)}) +i\kappa|t\in [0,1]\},
    \\ h_{2} & :=\{x_{2}T^{-1}t +(1-t)\Re(z_{2}^{(m)}) -i\kappa|t\in [0,1]\},
    \end{split}
\end{align} and both connected components of $\{\Re(z_{4}^{(m)})+iy||y|\geq \kappa^{-1}\}$ are oriented in the positive imaginary direction, and the segments defined above are oriented in the positive $t$ direction. 

The contour $C_{z}^{(m)}$ follows the steepest descent contour inside
$\mathcal R(c_{0},\kappa)$ and its reflected copy in the lower
half-plane. Outside this region it continues vertically for all $z$ such that 
$|\Im(z)|\geq\kappa^{-1}$. For all $z$ such that $|\Im z|\leq \kappa$, it takes a
small detour so that the contour crosses the real axis at $x_{2}T^{-1}$. By
\cref{l:conditions3}, the lengths of $h_{1}$ and $h_{2}$ are bounded by $|x_{2}|T^{-1}+R\kappa.$ 

Similarly, defining $C_{2}^{(m),-}$ as the reflection of $C_{2}^{(m)}$ across the real axis, we define 
\begin{align*}
    C_{w}^{(m)} & := C_{2}^{(m)}\cup C_{2}^{(m),-}\cup v_{1}\cup v_{2}\cup p_{1}\cup p_{2},
\end{align*} where 
\begin{align*}
    v_{1} & := \begin{cases}
        \{\Re(z_{1}^{(m)})- it\kappa| t\in [-1,1]\} & x_{1}T^{-1}-\kappa>\Re(z_{1}^{(m)})
        \\ \{x_{1}T^{-1}-\kappa-it\kappa| t\in [-1,1]\} & \text{else}
    \end{cases},
    \\ v_{2} & := \begin{cases}
        \{\Re(z_{3}^{(m)})+ it\kappa| t\in [-1,1]\} & x_{1}T^{-1}+\kappa<\Re(z_{3}^{(m)})
        \\ \{x_{1}T^{-1}+\kappa+ it\kappa| t\in[-1,1]\} & \text{else}
    \end{cases},
    \\ p_{1}  & := \begin{cases} 
        \{(x_{1}T^{-1}-\kappa)(1-t)+\Re(z_{3}^{(m)})t+ i\kappa| t\in [0,1] \} & x_{1}T^{-1}+\kappa<\Re(z_{1}^{(m)})
        \\ \{(x_{1}T^{-1}-\kappa)t+\Re(z_{1}^{(m)})(1-t)+ i\kappa| t\in [0,1]\} & x_{1}T^{-1}-\kappa>\Re(z_{3}^{(m)})
        \\ \emptyset & \text{else}
    \end{cases} ,
    \\ p_{2} & := \begin{cases} 
        \{(x_{1}T^{-1}+\kappa)t+\Re(z_{3}^{(m)})(1-t)- i\kappa| t\in [0,1] \} & x_{1}T^{-1}+\kappa<\Re(z_{1}^{(m)})
        \\ \{(x_{1}T^{-1}-\kappa)(1-t)+\Re(z_{1}^{(m)})t- i\kappa| t\in [0,1]\} & x_{1}T^{-1}-\kappa>\Re(z_{3}^{(m)})
        \\ \emptyset & \text{else}
    \end{cases} .
\end{align*} Each of these contours is oriented in the positive $t$ direction, and when we take the union to form $C_{w}^{(m)}$, the closed contour is oriented counterclockwise. 

The contours
$v_{1},v_{2}$ are line segments parallel to the imaginary axis spanning $|\Im(z)|\leq \kappa$ such that $\Re(z_{1})<x_{1}T^{-1}$ for all $z_{1}\in v_{1}$ and $\Re(z_{2})>x_{1}T^{-1}$ for all $z_{2}\in v_{2}.$ The location of the real part of $v_{1},v_{2}$ will depend on the location of $x_{1}$. 
One of the two segments $v_{1}$ or $v_{2}$ will have real part equal to either $\Re(z_{1}^{(m)})$ or $\Re(z_{3}^{(m)})$, and the other will have real part separated from the other of these points by a maximum displacement of $ \kappa+|x_{1}|T^{-1}$. If $x_{1}T^{-1}$ is already between $\Re(z_{1}^{(m)})+\kappa$ and $\Re(z_{3}^{(m)})-\kappa$, then the real parts of $v_{1}$ and $v_{2}$ are respectively equal to $\Re(z_{1}^{(m)})$ and $\Re(z_{3}^{(m)})$. We can also choose $v_{1}$ and $v_{2}$ so that they are guaranteed to be separated from $x_{1}T^{-1}$ by at least $ \kappa$, and, if either one does not have real part equal to the corresponding $\Re(z_{1}^{(m)}),\Re(z_{3}^{(m)})$, we can require that that contour is separated from $x_{1}T^{-1}$ by exactly $\kappa$. The lines parallel to the real axis at $y=\pm \kappa$ which connect this segment to $C_{2}^{(m)}$ and $C_{2}^{(m),-}$ are the sets labeled by $p_{1},p_{2}$, which therefore have length bounded by $|x_{1}|T^{-1}+\kappa$.  

For $m$ sufficiently large, and under the assumption that $Tm^{-1}\ll \log^{2}{(m)}T^{-1/2}\ll \kappa$ and $\kappa\to 0$ as $m\to\infty$, in this range of parameters, $p_{1},p_{2}$ have length $0$. The definitions also imply that $C_{z}^{(m)}$ is separated from $v_{1}\cup v_{2}\cup p_{1}\cup p_{2}$ by at least $c\kappa$ for some $c>0$.

\subsection{Proof of \cref{t:convrate}}\label{s:proofofmain}
Using the global steepest descent contours constructed in the previous section, we will prove \cref{t:convrate}. We will state one additional result before that proof. The next lemma is proved in \cref{s:crossterms}.
\begin{lemma}\label{l:errortermscross} Assume that $\kappa\to 0$ as $m\to\infty$ and that $Tm^{-1}\ll \log^{2}{(m)}T^{-1/2}\ll \kappa$. Also assume that $\log^{4}{(m)}\ll T\ll m^{2/3}$. For all $\alpha\in (0,2)$, and $K\subset \mathbb{N}\times \mathbb{R}$ compact, there exist $C,c>0,m_{0}\in\mathbb{N}$ such that for all $(n_{1},x_{1}),(n_{2},x_{2})\in K$, $m>m_{0},$ $x^{(m)}\in\Omega_{m}$, and $X\in (-2+\alpha,2-\alpha),$
\begin{align*}
    \bigg| \frac{n_{1}+T}{T} \frac{-1}{(2\pi i)^{2}}\int_{C_{z}^{(m)}\setminus (C_{1}^{(m)}\cup C_{1}^{(m),-})} dz \int_{C_{w}^{(m)}}dw R_{m}(z,w)  \frac{e^{T\left(S_{m}(z) -S_{m}(w)\right) }}{w-z}\bigg|  & \leq Ce^{-cT},
    \\ 
        \bigg|\frac{n_{1}+T}{T} \frac{-1}{(2\pi i)^{2}}\int_{C_{z}^{(m)}} dz \int_{C_{w}^{(m)}\setminus (C_{2}^{(m)}\cup C_{2}^{(m),-})}dw R_{m}(z,w)  \frac{e^{T\left(S_{m}(z) -S_{m}(w)\right) }}{w-z}\bigg| & \leq  Ce^{-cT}.
    \end{align*}

\end{lemma}
Now we can prove the main result of the section.
\begin{proof}[Proof of \cref{t:convrate}]
We separate the kernel into two terms using our new contours $C_{z}^{(m)}$ and $C_{w}^{(m)}$, taking the positive (counter-clockwise) orientation of $C_{w}^{(m)}$, and orienting $C_{z}^{(m)}$ in the positive imaginary direction,
\begin{multline}\label{e:bigequation}
    \frac{n_{1}+T}{T} \frac{-1}{(2\pi i)^{2}}\int_{\frac{x_{2}}{T}+\varepsilon-i\infty}^{\frac{x_{2}}{T}+\varepsilon+i\infty} dz \oint_{c\left(\frac{x_{1}}{T}\right)}dw \frac{\left(z-\frac{x_{2}}{T}\right)^{n_{2}+T-1}w^{T}}{z^{T}\left(w-\frac{x_{1}}{T}\right)^{n_{1}+T+1}}  \frac{e^{T\left(S_{m}(z) -S_{m}(w)\right) }}{w-z} 
    \\ = \frac{n_{1}+T}{T} \frac{-1}{(2\pi i)^{2}}\int_{C_{z}^{(m)}} dz \oint_{C_{w}^{(m)}}dw \frac{\left(z-\frac{x_{2}}{T}\right)^{n_{2}+T-1}w^{T}}{z^{T}\left(w-\frac{x_{1}}{T}\right)^{n_{1}+T+1}}  \frac{e^{T\left(S_{m}(z) -S_{m}(w)\right) }}{w-z} \\ + \frac{n_{1}+T}{T} \frac{1}{2\pi i}\int_{\gamma_{m}} dz \frac{\left(z-\frac{x_{2}}{T}\right)^{n_{2}+T-1}}{\left(z-\frac{x_{1}}{T}\right)^{n_{1}+T+1}}. 
\end{multline}
Where $\gamma_{m}$ is a contour which starts at $\overline{z_{0}^{(m)}}$, ends at $z_{0}^{(m)}$, and crosses the real line at $x_{2}T^{-1}+\varepsilon$, where $\varepsilon$ can be infinitesimally small (so that if $x_{2}T^{-1}<x_{1}T^{-1}$, then for all finite $m$, $x_{2}T^{-1}+\varepsilon<x_{1}T^{-1}$).   

We consider the first term on the right-hand side of \eqref{e:bigequation}, splitting it into several components, 
\begin{align*}
    & \bigg|\frac{n_{1}+T}{T} \frac{-1}{(2\pi i)^{2}}\int_{C_{z}^{(m)}} dz \oint_{C_{w}^{(m)}}dw \frac{\left(z-\frac{x_{2}}{T}\right)^{n_{2}+T-1}w^{T}}{z^{T}\left(w-\frac{x_{1}}{T}\right)^{n_{1}+T+1}}  \frac{e^{T\left(S_{m}(z) -S_{m}(w)\right) }}{w-z} \bigg|
    \\ & \leq \bigg|\frac{n_{1}+T}{T} \frac{-1}{(2\pi i)^{2}}\int_{C_{1}^{(m)}} dz \oint_{C_{2}^{(m)}}dw \frac{\left(z-\frac{x_{2}}{T}\right)^{n_{2}+T-1}w^{T}}{z^{T}\left(w-\frac{x_{1}}{T}\right)^{n_{1}+T+1}}  \frac{e^{T\left(S_{m}(z) -S_{m}(w)\right) }}{w-z} \bigg|
    \\ & + \bigg|\frac{n_{1}+T}{T} \frac{-1}{(2\pi i)^{2}}\int_{C_{1}^{(m),-}} dz \oint_{C_{2}^{(m),-}}dw \frac{\left(z-\frac{x_{2}}{T}\right)^{n_{2}+T-1}w^{T}}{z^{T}\left(w-\frac{x_{1}}{T}\right)^{n_{1}+T+1}}  \frac{e^{T\left(S_{m}(z) -S_{m}(w)\right) }}{w-z} \bigg|
    \\ & + \bigg|\frac{n_{1}+T}{T} \frac{-1}{(2\pi i)^{2}}\int_{C_{1}^{(m)}} dz \oint_{C_{2}^{(m),-}}dw \frac{\left(z-\frac{x_{2}}{T}\right)^{n_{2}+T-1}w^{T}}{z^{T}\left(w-\frac{x_{1}}{T}\right)^{n_{1}+T+1}}  \frac{e^{T\left(S_{m}(z) -S_{m}(w)\right) }}{w-z} \bigg|
    \\ & + \bigg|\frac{n_{1}+T}{T} \frac{-1}{(2\pi i)^{2}}\int_{C_{1}^{(m),-}} dz \oint_{C_{2}^{(m)}}dw \frac{\left(z-\frac{x_{2}}{T}\right)^{n_{2}+T-1}w^{T}}{z^{T}\left(w-\frac{x_{1}}{T}\right)^{n_{1}+T+1}}  \frac{e^{T\left(S_{m}(z) -S_{m}(w)\right) }}{w-z}\bigg|
     \\ & + \bigg|\frac{n_{1}+T}{T} \frac{-1}{(2\pi i)^{2}}\int_{C_{z}^{(m)}\setminus (C_{1}^{(m)}\cup C_{1}^{(m),-})} dz \int_{C_{w}^{(m)}}dw \frac{\left(z-\frac{x_{2}}{T}\right)^{n_{2}+T-1}w^{T}}{z^{T}\left(w-\frac{x_{1}}{T}\right)^{n_{1}+T+1}}  \frac{e^{T\left(S_{m}(z) -S_{m}(w)\right) }}{w-z} \bigg|
     \\ & + \bigg|\frac{n_{1}+T}{T} \frac{-1}{(2\pi i)^{2}}\int_{C_{z}^{(m)}} dz \int_{C_{w}^{(m)}\setminus (C_{2}^{(m)}\cup C_{2}^{(m),-})}dw \frac{\left(z-\frac{x_{2}}{T}\right)^{n_{2}+T-1}w^{T}}{z^{T}\left(w-\frac{x_{1}}{T}\right)^{n_{1}+T+1}}  \frac{e^{T\left(S_{m}(z) -S_{m}(w)\right) }}{w-z}\bigg|.
\end{align*} The expression above has an inequality instead of an equality because the final two terms double count one part of the contour. 

Using \cref{r:Rbd}, \cref{l:unique}, \cref{l:conditions1}, \cref{r:extra}, \cref{l:lengthbound}, \cref{r:extra2}, and \cref{l:conditions3}, (and their implicit analogues in the lower half-plane) we can apply \cref{p:steepest} and \cref{r:steepest} to conclude that the first four terms on the right-hand side above converge to $0$ as $m\to\infty$ at a rate bounded above by 
\begin{align}\label{e:finalbds}
    (1+c^{-2})^{(n_2-1)/2}
\kappa^{-n_1-n_2-5}
e^{4(|x_1|+|x_2|)\kappa^{-1}}e^{-dT} + C T^{-1/2}
\end{align} where $d$ is as in \cref{r:extra}, and where $C>0$ does not depend on $x_{1},x_{2},n_{1},n_{2}.$

Thus, for any $\kappa$ asymptotically dominating  $(dT-\frac{1}{2}\log{(T)})^{-1}4(|x_{1}|+|x_{2}|)$, the entire expression tends to $0$ at rate $T^{-1/2}$, up to constants depending on $x_{1},x_{2},n_{1},n_{2}$. Similarly, we require that $\kappa^{-1}e_{m}\ll T^{-1/2},$ which is satisfied whenever $\kappa$ asymptotically dominates $\log^{2}{(m)}/T^{1/2} + T^{3/2}/m $. Therefore, we choose  $\kappa=\kappa(m)\to 0$ so that
$\log^{2}{(m)} T^{-1/2}+T^{3/2}m^{-1}\ll \kappa$. 
This is possible because we have only required $\log^{4}{(m)}\ll T\ll m^{2/3}$. 
We combine this with the result of \cref{l:errortermscross} to see that for all compact $K\subset \mathbb{N}\times \mathbb{R}$, there exists $C''>0$ such that the first term of \eqref{e:bigequation} is bounded by $C''T^{-1/2}$.

Now we focus on the second term of \eqref{e:bigequation}. We will deal with this term in two cases. \newline\hfill\newline 
\textbf{Case 1:} In this case we assume $\Re(z_{0})\neq 0$. 
In this case, for all $m$ sufficiently large, $|x_{1}|T^{-1},|x_{2}|T^{-1}<|\Re(z_{0}^{(m)})|$.

When $\Re(z_{0}^{(m)})>0$ and  $x_{1}\leq x_{2}$, for $m$ sufficiently large, it is therefore equivalent to take $\gamma_{m}$ to be the straight segment between $\overline{z_{0}^{(m)}}$ and $z_{0}^{(m)}$. When $x_{2}<x_{1}$, the integral along $\gamma_{m}$ is equal to the integral along the straight segment between $\overline{z_{0}^{(m)}}$ and $z_{0}^{(m)}$ minus the residue at $x_{1}T^{-1}.$  
Thus, these terms become 
\begin{align*}
       -\mathbf{1}_{n_{2}\geq n_{1}+1}\mathbf{1}_{x_{1}>x_{2}}\frac{n_{1}+T}{T}{ n_{2}+T-1 \choose n_{1}+T} \left(\frac{x_{1}-x_{2}}{T}\right)^{n_{2}-n_{1}-1} + \frac{n_{1}+T}{T} \frac{1}{2\pi i}\int_{\overline{z_{0}^{(m)}}}^{z_{0}^{(m)}} dz \frac{\left(z-\frac{x_{2}}{T}\right)^{n_{2}+T-1} }{\left(z-\frac{x_{1}}{T}\right)^{n_{1}+T+1} }.
 \end{align*}
 When $\Re(z_{0}^{(m)})<0$ and $x_{2}<x_{1}$, we can replace the contour $\gamma_{m}$ with the segment between $\overline{z_{0}^{(m)}}$ and $z_{0}^{(m)}$, and when $x_{1}\leq x_{2}$, it can be replaced with that contour, plus the residue at $x_{1}T^{-1}$. 
\begin{align*}
       \mathbf{1}_{n_{2}\geq n_{1}+1}\mathbf{1}_{x_{1}\leq x_{2}}\frac{n_{1}+T}{T}{ n_{2}+T-1 \choose n_{1}+T} \left(\frac{x_{1}-x_{2}}{T}\right)^{n_{2}-n_{1}-1} + \frac{n_{1}+T}{T} \frac{1}{2\pi i}\int_{\overline{z_{0}^{(m)}}}^{z_{0}^{(m)}} dz \frac{\left(z-\frac{x_{2}}{T}\right)^{n_{2}+T-1} }{\left(z-\frac{x_{1}}{T}\right)^{n_{1}+T+1} }.
 \end{align*}

By Stirling's formula and dominated convergence, and since $\Re(z_{0})\neq 0$ by assumption, as $m\to\infty,$ this goes to
\begin{align*}
   -\mathbf{1}_{n_{2}\geq n_{1}+1}\mathbf{1}_{\Re(z_{0})>0;x_{1}>x_{2}}\frac{(x_{1}-x_{2})^{n_{2}-n_{1}-1}}{(n_{2}-n_{1}-1)!}+\mathbf{1}_{n_{2}\geq n_{1}+1}\mathbf{1}_{\Re(z_{0})<0;x_{1}\leq x_{2}}\frac{(x_{1}-x_{2})^{n_{2}-n_{1}-1}}{(n_{2}-n_{1}-1)!}  \\ + \frac{1}{2\pi i}\int_{\overline{z_{0}}}^{z_{0}} dz z^{n_{2}-n_{1}-2}e^{(x_{1}-x_{2})/z}. 
\end{align*} The error rate of this convergence is  $O(T^{-1/2})$, up to constants which may depend on $x_{1},x_{2},n_{1},n_{2},$ and which may be chosen uniformly over $(n_{1},x_{1}),(n_{2},x_{2})\in K$ for any compact $K\subset\mathbb{N}\times\mathbb{R}$.  
We perform the change of variables $u=-1/z$ to obtain, up to conjugation by non-vanishing functions, and deformation of the contours
\begin{align*}
    -\mathbf{1}_{n_{2}\geq n_{1}+1}\mathbf{1}_{\Re(z_{0})>0;x_{1}>x_{2}}\frac{(x_{1}-x_{2})^{n_{2}-n_{1}-1}}{(n_{2}-n_{1}-1)!}+\mathbf{1}_{n_{2}\geq n_{1}+1}\mathbf{1}_{\Re(z_{0})<0;x_{1}\leq x_{2}}\frac{(x_{1}-x_{2})^{n_{2}-n_{1}-1}}{(n_{2}-n_{1}-1)!} \\ -\frac{1}{2\pi i}\int_{-1/u \in [\overline{z_{0}},z_{0}]}du u^{n_{1}-n_{2}}e^{u(x_{2}-x_{1})}, 
\end{align*} We then deform the contour of the integral to $\Re(-1/z_{0})+i\mathbb{R}\setminus [\Im(-\overline{1/z_{0}}),\Im(-1/z_{0})]$, which produces two factors when $u$ crosses $0$, which cancel exactly with the first two terms, sending the entire expression to
\begin{align*}
    -\frac{1}{2\pi i}\int_{\Re(-1/z_{0})+i\mathbb{R}\setminus [\Im(-\overline{1/z_{0}}),\Im(-1/z_{0})]}du u^{n_{1}-n_{2}}e^{u(x_{2}-x_{1})}.
\end{align*}
When $n_{2}\leq n_{1}$, this is equal to 
\begin{align*}
    \frac{1}{2\pi i}\int_{-\overline{1/z_{0}}}^{-1/z_{0}}du u^{n_{1}-n_{2}}e^{u(x_{2}-x_{1})},
\end{align*} and these two cases together exactly match 
$K^{\text{sine}}_{-1/z_{0}}(n_{1},x_{1};n_{2},x_{2})$ as defined in \cref{d:extendedsine}.
\newline\hfill\newline
\textbf{Case 2:} In this case, we allow $\Re(z_{0}^{(m)})\to 0$ as $m\to\infty$. Without loss of generality, by deforming the contours, we can assume that $\Re(z_{0}^{(m)})=0$ for all $m$ sufficiently large (in this case, if $x_{1}=0$, we allow the contour to deviate towards the positive real axis around $z=0$ to avoid the singularity). With this simplification, the term in question takes the form
\begin{align*}
    \frac{n_{1}+T}{T} \frac{1}{2\pi i}\int_{-i\Im(z_{0}^{(m)})}^{i\Im(z_{0}^{(m)})} dz \frac{\left(z-\frac{x_{2}}{T}\right)^{n_{2}+T-1}}{\left(z-\frac{x_{1}}{T}\right)^{n_{1}+T+1} }.
\end{align*} If $x_{1}> x_{2}$, then for some $c>0$, this is equivalent to 
\begin{multline*}
    \frac{n_{1}+T}{T} \frac{1}{2\pi i}\int_{-c-i\Im(z_{0}^{(m)})}^{-c+i\Im(z_{0}^{(m)})} dz \frac{\left(z-\frac{x_{2}}{T}\right)^{n_{2}+T-1}}{\left(z-\frac{x_{1}}{T}\right)^{n_{1}+T+1} } +\frac{n_{1}+T}{T} \frac{1}{2\pi i}\int_{-c+i\Im(z_{0}^{(m)})}^{i\Im(z_{0}^{(m)})} dz \frac{\left(z-\frac{x_{2}}{T}\right)^{n_{2}+T-1}}{\left(z-\frac{x_{1}}{T}\right)^{n_{1}+T+1} } \\ +\frac{n_{1}+T}{T} \frac{1}{2\pi i}\int_{-i\Im(z_{0}^{(m)})}^{-c-i\Im(z_{0}^{(m)})} dz \frac{\left(z-\frac{x_{2}}{T}\right)^{n_{2}+T-1}}{\left(z-\frac{x_{1}}{T}\right)^{n_{1}+T+1} }.
\end{multline*} As $m\to\infty$, by dominated convergence, this expression goes to 
\begin{multline*}
    \frac{1}{2\pi i}\int_{-c-i\Im(z_{0} )}^{-c+i\Im(z_{0} )} dz z^{n_{2}-n_{1}-2}e^{(x_{1}-x_{2})/z}  +  \frac{1}{2\pi i}\int_{-c+i\Im(z_{0} )}^{i\Im(z_{0} )}   dz z^{n_{2}-n_{1}-2}e^{(x_{1}-x_{2})/z} \\ +  \frac{1}{2\pi i}\int_{-i\Im(z_{0} )}^{-c-i\Im(z_{0} )}  dz z^{n_{2}-n_{1}-2}e^{(x_{1}-x_{2})/z}.
\end{multline*} The error rate of this convergence is  $O(T^{-1}+e_{m})$. We can do the same when $x_{1}\leq x_{2}$ by shifting the contour by a fixed amount away from $0$. Then under the change of variables $z=-1/u$, up to conjugation by non-vanishing functions and sending $c\to 0$ after changing variables, this becomes 
\begin{align*}
    -  \frac{1}{2\pi i}\int_{i\mathbb{R}\setminus [-i/\Im(z_{0}),i/\Im(z_{0})]}   du u^{n_{1}-n_{2}}e^{u(x_{2}-x_{1})}.
\end{align*}
As before, when $n_{2}\leq n_{1}$, this is equal to 
\begin{align*}
     \frac{1}{2\pi i}\int_{-\overline{1/z_{0}}}^{-1/z_{0}}   du u^{n_{1}-n_{2}}e^{u(x_{2}-x_{1})}.
\end{align*}
\end{proof}

\subsection{Proof of \cref{l:errortermscross}}\label{s:crossterms}
To prove \cref{l:errortermscross}, we will need estimates on how the product, which appears in the integrand, \cref{t:kernel} behaves close to the real line. 

\begin{lemma} For all $\alpha\in (0,2),$  there exist $C:=C(\alpha),C':=C'(\alpha)>0$ such that for all $m\in \mathbb{N}$, $x^{(m)}\in\Omega_{m}$, 
   $|y|\in (0,\kappa)$, and $x\in\mathbb{R}$ such that $X,X+xT/m\in (-2+\alpha,2-\alpha)$, 
    \begin{align}\label{e:bd1}
        \bigg| \prod_{r=1}^{m}\frac{x+i\kappa- u_{r}^{(m)}}{x+iy-u_{r}^{(m)}}\bigg| 
        & \leq \left(\frac{\kappa}{|y|}\right)^{C\log^{2}{(m)}}  
    \exp{\left(\frac{C\pi }{2}T(\kappa-|y|) \right)},
    \end{align} and 
    \begin{align}\label{e:bd2}
    \bigg| \prod_{r=1}^{m}\frac{x+iy- u_{r}^{(m)}}{x+i\kappa-u_{r}^{(m)}}\bigg| &  \leq \exp{\left(-  \frac{\pi\rho_{sc}(X)T(\kappa^{2}-y^{2})}{2\kappa } + \frac{C'(\kappa^{2}-y^{2})\log^{2}{(m)}}{\kappa^{2}} \right)}.
\end{align} Furthermore, there exists $C'':=C''(\alpha)>0$ such that for all $m\in\mathbb{N}$, $x^{(m)}\in\Omega_{m}$, $x\in\mathbb{R}$, $\delta>0$, with $X,X+xT/m,X+(x+\delta)T/m\in (-2+\alpha,2-\alpha)$,  
\begin{align}\label{e:bd3}
     \bigg| \prod_{r=1}^{m}\frac{x+\delta+i\kappa-u_{r}^{(m)}}{x+i\kappa-u_{r}^{(m)}}\bigg|  & \leq   \exp{\left(\frac{T|X|\delta}{2}+\frac{T^{2}}{2m}(x\delta+\delta^{2}/2) + C''\delta T e_{m}\kappa^{-1} \right)}.
\end{align}
\end{lemma}
\begin{proof}
    \cref{thm:rigid} shows that the number of eigenvalues within any interval of length $\log^{2}{(m)}m^{-1}$ is $O(\log^{2}{(m)})$. Consequently, there are $O(\log^{2}{(m)})$ of the $u_{r}^{(m)}$ within $\log^{2}{(m)}T^{-1}$ of $\Re(z)$, and $O(\log^{2}{(m)})$ which are separated by at least $\log^{2}{(m)}T^{-1}$ and at most $2\log^{2}{(m)}T^{-1}$ and similarly for separations of at least $(k-1)\log^{2}{(m)}T^{-1}$ and at most $k\log^{2}{(m)}T^{-1}$ for $k\in [[3,\lfloor m\log^{-2}{(m)}\rfloor ]]$. Therefore, for all $x+iy$ such that $X+xT/m\in (-2+\alpha,2-\alpha)$ and $|y|\in (0,\kappa)$,  there exists a collection $\{c_{k}(x)\}_{k=1}^{m\log^{-2}{(m)}-1},$ defined by 
    \begin{align*}
        c_{k}:= c_{k}(x) & = \frac{1}{\log^{2}{(m)}}|\{ u_{r}^{(m)} | |x- u_{r}^{(m)}| \in [(k-1)\log^{2}{(m)}T^{-1},k\log^{2}{(m)}T^{-1}] \}|.
    \end{align*}  Furthermore, there exists $C_{\text{rig}}>0$ such that for all admissible $x,k,m$, $c_{k}(x)\leq C_{\text{rig}}$. 
    
We conclude that for all $x^{(m)}\in  \Omega_{m}$, $x\in\mathbb{R}$, and $y$ such that $|y|\in (0, \kappa)$,
\begin{align*}
    & \bigg| \prod_{r=1}^{m}\frac{x+i\kappa-u_{r}^{(m)}}{x+iy- u_{r}^{(m)}}\bigg| 
    \\ & \leq \left(\prod_{r=1}^{c_{1}\log^{2}{(m)}}\frac{(x-\hat{u}_{r}^{(m)})^{2}+\kappa^{2}}{(x-\hat{u}_{r}^{(m)})^{2}+y^{2}}\right)^{1/2} 
    \exp{\left(\log^{2}{(m)}\int_{|y|}^{\kappa}\left( \sum_{k=1}^{\lfloor m\log^{-2}{(m)}-1\rfloor }\frac{sc_{k+1}}{s^{2}+k^{2}\log^{4}{(m)}T^{-2}}\right)ds\right)},
\end{align*} where the $\{\hat{u}_{r}^{(m)}\}_{r=1}^{c_{1}\log^{2}{(m)}}$ are the relabeled $c_{1}\log^{2}{(m)}$ eigenvalues within $\log^{2}T^{-1}$ of $x$. This expression is bounded by 
\begin{align*}
    & \leq \left(\frac{\kappa}{|y|}\right)^{C_{\text{rig}}\log^{2}{(m)}} \exp{\left(C_{\text{rig}}\log^{2}{(m)}\int_{|y|}^{\kappa}\left( \sum_{k=1}^{\lfloor m\log^{-2}{(m)}-1\rfloor }\frac{s}{s^{2}+k^{2}\log^{4}{(m)}T^{-2}}\right)ds\right)}.
\end{align*} We use the estimate 
\begin{align*}
    \sum_{k=1}^{m\log^{-2}{(m)}-1}\frac{1}{s^{2}+k^{2}\log^{4}{(m)}T^{-2}}  & \leq  \frac{\pi T}{2\log^{2}{(m)}|s|},
\end{align*} to conclude that 
\begin{align*}
   \bigg| \prod_{r=1}^{m}\frac{x+i\kappa-u_{r}^{(m)}}{x+iy- u_{r}^{(m)}}\bigg| & \leq \left(\frac{\kappa}{|y|}\right)^{C_{\text{rig}}\log^{2}{(m)}}  \exp{\left(\frac{C_{\text{rig}}\pi }{2}T(\kappa-|y|) \right)}.
\end{align*}
To show the second statement, we write 
\begin{align*}
    \bigg| \prod_{r=1}^{m}\frac{x+iy- u_{r}^{(m)}}{x+i\kappa-u_{r}^{(m)}}\bigg| &  = \exp{\left(\frac{1}{2}\sum_{r=1}^{m}\log{\left(1 -\frac{\kappa^{2}-y^{2}}{\kappa^{2}+(x-u_{r}^{(m)})^{2}}\right)}\right)}
     \leq  \exp{\left(-\frac{\kappa^{2}-y^{2}}{2}\sum_{r=1}^{m}\frac{1}{\kappa^{2}+(x-u_{r}^{(m)})^{2}} \right)}.
\end{align*} 
For $\mu_{m}$ as defined in \cref{t:local},
\begin{align*}
    \frac{1}{T}\sum_{r=1}^{m}\frac{1}{\kappa^{2}+(x-u_{r}^{(m)})^{2}} = \int_{\mathbb{R}}\frac{1}{\kappa^{2}+(x-u)^{2}}\mu_{m}(du),
\end{align*} and therefore, by eigenvalue rigidity, for all $x\in\mathbb{R}$ such that $X+xT/m\in (-2+\alpha,2-\alpha)$, we find that
\begin{align*}
    \sum_{r=1}^{m}\frac{1}{\kappa^{2}+(x-u_{r}^{(m)})^{2}} & = \frac{\pi\rho_{sc}(X)T}{\kappa} \pm O\left(\log^{2}{(m)}\mathrm{TV}\left(u\mapsto \frac{1}{\kappa^{2}+(x-u)^{2}}\right)\right)
    \\ & = \frac{\pi\rho_{sc}(X)T}{\kappa } \pm O\left(\frac{\log^{2}{(m)}}{\kappa^{2}} \right).
\end{align*} 
Therefore, we conclude
\begin{align*}
    \bigg| \prod_{r=1}^{m}\frac{x+iy- u_{r}^{(m)}}{x+i\kappa-u_{r}^{(m)}}\bigg|  
    & \leq \exp{\left(-  \frac{\pi\rho_{sc}(X)T(\kappa^{2}-y^{2})}{2\kappa } + O\left(\frac{(\kappa^{2}-y^{2})\log^{2}{(m)}}{\kappa^{2}} \right)\right)}.
\end{align*} 
To obtain the third bound, we perform a similar computation, 
\begin{align*}
    \bigg| \prod_{r=1}^{m}\frac{x+\delta+i\kappa-u_{r}^{(m)}}{x+i\kappa-u_{r}^{(m)}}\bigg| & =  \prod_{r=1}^{m}\left(\frac{(x+\delta-u_{r}^{(m)})^{2}+\kappa^{2}}{(x-u_{r}^{(m)})^{2}+\kappa^{2}}\right)^{1/2}
    \\  & =  \exp{\left(\int_{0}^{\delta}\Re\left(\sum_{r=1}^{m}\frac{1}{x+s+i\kappa-u_{r}^{(m)}} \right)ds\right)}.
\end{align*} By \cref{l:sconv}, this is bounded by 
\begin{align*}
    & \leq   \exp{\left(\int_{0}^{\delta} T\left(\frac{X}{2}+\frac{T}{2m}(x+s) + O(e_{m}\kappa^{-1})\right)ds\right)}  \leq   \exp{\left(\frac{T|X|\delta}{2}+\frac{T^{2}}{2m}(x\delta+\delta^{2}/2) + O(\delta T e_{m}\kappa^{-1}) \right)}.
\end{align*}
\end{proof}
We will use these bounds to demonstrate that the one dimensional integrals along the contours involved in \cref{l:errortermscross} are acceptably bounded. We begin by studying the $z$ contour.
\begin{lemma}\label{l:zrealbd} Assume that $\kappa:=\kappa(m)\to 0$ as $m\to\infty$ and that $Tm^{-1}\ll\log^{2}{(m)}T^{-1/2}\ll\kappa$. Also assume that $\log^{4}{(m)}\ll T\ll m^{2/3}$. For all $\alpha\in (0,2)$ and $K\subset\mathbb{N}\times \mathbb{R}$ compact, there exist $C:=C(K,\alpha)>0,d:=d(\alpha)>0,m_{0}:=m_{0}(K,\alpha)\in\mathbb{N}$ such that for all $m>m_{0}$, $x^{(m)}\in  \Omega_{m}$, $X\in (-2+\alpha,2-\alpha)$ and $(n_{2},x_{2})\in K$, each of the expressions 
    \begin{align*}
       & \bigg| \int_{v}dz \frac{\left(z-\frac{x_{2}}{T}\right)^{n_{2}+T-1}}{z^{T}}e^{TS_{m}(z)-TS_{m}(z_{0}^{(m)})}\bigg|,  \bigg| \int_{h_{1}\cup h_{2}} dz  \frac{\left(z-\frac{x_{2}}{T}\right)^{n_{2}+T-1}}{z^{T}}e^{TS_{m}(z)-TS_{m}(z_{0}^{(m)})}\bigg|,
       \\ & \bigg| \int_{\{\Re(z_{4}^{(m)})+iy| |y|\geq \kappa^{-1}\}} dz \frac{\left(z-\frac{x_{2}}{T}\right)^{n_{2}+T-1}}{z^{T}}e^{TS_{m}(z)-TS_{m}(z_{0}^{(m)})}\bigg|,
    \end{align*} is bounded by $Ce^{-dT }$. 
\end{lemma}
\begin{proof} 
To deal with the first term, we simplify the expression using the assumptions on the relative rates of $\kappa,T,m$. There exist $C:=C(K,\alpha),C':=C'(K,\alpha)>0$ such that 
\begin{align*}
    & \bigg|\int_{v}dz \frac{\left(z-\frac{x_{2}}{T}\right)^{n_{2}+T-1}}{z^{T}}e^{TS_{m}(z)-TS_{m}(x_{2}T^{-1}+i\kappa)}\bigg| 
    \\ & \leq C'\frac{1}{|x_{2}T^{-1}+i\kappa|^{T}} \kappa^{C\log^{2}{(m)}} e^{\frac{C\pi }{2}T\kappa }\int_{0}^{\kappa} dy y^{n_{2}+T-1-C\log^{2}{(m)}}   \leq \frac{C'e^{\frac{C\pi }{2}T\kappa }\kappa^{n_{2}+T}}{|x_{2}T^{-1}+i\kappa|^{T}(n_{2}+T-C\log^{2}{(m)})} . 
\end{align*}
 Applying again the assumptions on the relative rates of $\kappa,T,m$, we find that for  any fixed compact $K$, for any  $(n_{2},x_{2})\in K$ there exist $C':=C'(K,\alpha)>0$ and $c>0$ such that this is bounded by 
\begin{align*}
     \bigg| \int_{v}dz \frac{\left(z-\frac{x_{2}}{T}\right)^{n_{2}+T-1}}{z^{T}}e^{TS_{m}(z)-TS_{m}(x_{2}T^{-1}+i\kappa)}\bigg|\leq \frac{C'  e^{cT\kappa} \kappa^{n_{2}}}{n_{2}+T-C\log^{2}{(m)}}\leq C'e^{cT\kappa}\kappa^{n_{2}}.
\end{align*}
Consequently, (and again applying the assumptions on the relative scale of $T,\kappa,m$), there exists $C:=C(K,\alpha)>0$ such that  
\begin{multline*}
     \bigg| \int_{v}dz \frac{\left(z-\frac{x_{2}}{T}\right)^{n_{2}+T-1}}{z^{T}}e^{TS_{m}(z)-TS_{m}(z_{0}^{(m)})}\bigg| 
     \\ \leq Ce^{T\Re(S_{m}(z_{2}^{(m)})-S_{m}(z_{0}^{(m)}))}\bigg| \int_{v}dz \frac{\left(z-\frac{x_{2}}{T}\right)^{n_{2}+T-1}}{z^{T}}e^{TS_{m}(z)-TS_{m}(x_{2}T^{-1}+i\kappa)}\bigg|
     \\ \cdot e^{x_{2}X-T\Re(z_{2}^{(m)})X}\left(1+\frac{x_{2}^{2}T^{-2}-\Re(z_{2}^{(m)})^{2}}{\Re(z_{2}^{(m)})^{2}+\kappa^{2}}\right)^{T/2}\prod_{r=1}^{m}\bigg|\frac{\Re(z_{2}^{(m)})+i\kappa-u_{r}^{(m)}}{x_{2}T^{-1}+i\kappa-u_{r}^{(m)}} \bigg|.
\end{multline*} Again applying the assumptions about the relative scale of $\kappa, T,m$,  which imply that  $|x_{2}T^{-1}-\Re(z_{2}^{(m)})|\leq \kappa$, and applying \eqref{e:bd3}, we conclude that there exist $C:=C(K,\alpha),C':=C'(K,\alpha), d:=d(\alpha)>0$ such that for all $(n_{2},x_{2})\in K$, 
\begin{multline*}
     \bigg| \int_{v}dz \frac{\left(z-\frac{x_{2}}{T}\right)^{n_{2}+T-1}}{z^{T}}e^{TS_{m}(z)-TS_{m}(z_{0}^{(m)})}\bigg|
     \\ \leq Ce^{T\Re(S_{m}(z_{2}^{(m)})-S_{m}(z_{0}^{(m)}))}\bigg| \int_{v}dz \frac{\left(z-\frac{x_{2}}{T}\right)^{n_{2}+T-1}}{z^{T}}e^{TS_{m}(z)-TS_{m}(x_{2}T^{-1}+i\kappa)}\bigg|\cdot \frac{|x_{2}T^{-1}+i\kappa|^{T}}{|z_{2}^{(m)}|^{T}} e^{\frac{3T|X|\kappa}{2} + C''T e_{m}}, 
\end{multline*} is bounded by $C'e^{-dT}$. 

Similarly, by the results of \cref{l:lengthbound}, the length of $ h_{1} $ is bounded by $R\kappa+|x_{2}|T^{-1}$ for some $R>0$. By applying the relative scale assumptions for all $z\in h_{1},$ we conclude that there exist $C:=C(K,\alpha),C':=C'(K,\alpha),C'':=C''(K,\alpha),C''':=C'''(K,\alpha),d:=d(\alpha)>0$ such that for all $(n_{2},x_{2})\in K$, 
\begin{multline*}
     \bigg|\int_{h_{1}}dz z^{-T}\left(z-\frac{x_{2}}{T}\right)^{n_{2}+T-1}e^{TS_{m}(z)-TS_{m}(z_{2}^{(m)})} \bigg| \\  \leq  Ce^{C'T\kappa + C''(R\kappa +|x_{2}|T^{-1}) T e_{m}\kappa^{-1}  } \bigg|\int_{h_{1}} dz(z_{2}^{(m)})^{-T}\left(z-\frac{x_{2}}{T}\right)^{n_{2}+T-1} \bigg|,
\end{multline*} is bounded $C''' e^{-dT}$. 

Finally, by \cref{l:6162}, for all $z$ with imaginary part sufficiently large, $\frac{\partial}{\partial y}\Re(S_{m}(z))$ is bounded above by a negative constant (possibly depending on $\alpha$). Applying the fact that $|\Re(z_{4}^{(m)})|\leq c\kappa^{-1}$, from which we conclude that for all $K\subset\mathbb{N}\times\mathbb{R}$ compact, there exist $C:=C(K,\alpha),C':=C'(K,\alpha),C'':=C''(K,\alpha),d:=d(\alpha),d':=d'(\alpha)>0$ such that
\begin{multline*}
    \bigg| \int_{\{\Re(z_{4}^{(m)})+iy| |y|\geq \kappa^{-1}\}} dz z^{-T}\left(z-\frac{x_{2}}{T}\right)^{n_{2}+T-1}e^{TS_{m}(z)-TS_{m}(z_{0}^{(m)})}\bigg| 
    \\ \leq e^{-dT}\bigg| \int_{\{\Re(z_{4}^{(m)})+iy| |y|\geq \kappa^{-1}\}} dz z^{-T}\left(z-\frac{x_{2}}{T}\right)^{n_{2}+T-1}e^{TS_{m}(z)-TS_{m}(z_{4}^{(m)})}\bigg|
    \leq \frac{C}{T}\kappa^{1-n_{2}}e^{-dT+C'\kappa}, 
\end{multline*} and such that this expression is bounded by $C''e^{-d'T}$.  
\end{proof}
We also prove a similar result for the $w$ contour. 
\begin{lemma}\label{l:wrealbd} Assume that $\kappa:=\kappa(m)\to 0$ as $m\to\infty$ and that $Tm^{-1}\ll\log^{2}{(m)}T^{-1/2}\ll \kappa$. Also assume that $\log^{4}{(m)}\ll T\ll m^{2/3}$. For all $\alpha\in (0,2)$ and $K\subset \mathbb{N}\times \mathbb{R}$ compact, there exist $m_{0}:=m_{0}(\alpha,K)\in\mathbb{N}$ and $C:=C(\alpha,K),d:=d(\alpha,K) >0$ such that for all $m>m_{0}$, $(n_{1},x_{1}) \in K$,  $x^{(m)}\in\Omega_{m}$, and $X\in (-2+\alpha,2-\alpha),$
    \begin{align*}
         \bigg| \int_{v_{1}\cup v_{2}} \frac{w^{T}}{\left(w-\frac{x_{1}}{T}\right)^{n_{1}+T+1}}e^{TS_{m}(z_{0}^{(m)})-TS_{m}(w)}\bigg| , \bigg| \int_{p_{1}\cup p_{2}} \frac{w^{T}}{\left(w-\frac{x_{1}}{T}\right)^{n_{1}+T+1}}e^{TS_{m}(z_{0}^{(m)})-TS_{m}(w)}\bigg| & \leq   C e^{-dT} .
    \end{align*}
\end{lemma}

\begin{proof} Starting with the first expression and applying \eqref{e:bd2}, we see that
\begin{multline*}
    \bigg|\int_{v_{1}}dw\frac{w^{T}}{\left(w-\frac{x_{1}}{T}\right)^{n_{1}+T+1}} e^{TS_{m}(\Re(v_{1})+i\kappa)-TS_{m}(w)}\bigg| 
    \\  \leq |\Re(v_{1})+i\kappa|^{T} \bigg|\int_{v_{1}}dw\frac{1}{\left(w-\frac{x_{1}}{T}\right)^{n_{1}+T+1}} \prod_{r=1}^{m}\frac{\Re(v_{1})+iy-u_{r}^{(m)}}{\Re(v_{1})+i\kappa-u_{r}^{(m)}}\bigg|  
    \\  \leq C\kappa|\Re(v_{1})+i\kappa|^{T} \left|\Re(v_{1})-\frac{x_{1}}{T}\right|^{-n_{1}-T-1}\frac{1 }{ \sqrt{\pi\rho_{sc}(X)T \kappa/2 - C'\log^{2}{(m)}  }}.
\end{multline*} 
Since $\rho_{sc}(X)$ is uniformly bounded away from $0$ for all $X\in (-2+\alpha,2-\alpha)$ , the denominator is bounded below uniformly for all $m$ sufficiently large. Thus, for $m$ sufficiently large so that $|x_{1}/T|<\kappa^{1/2}-\kappa$, and by the definition of the contour $v_{1}$, then for fixed compact $K\subset \mathbb{N}\times \mathbb{R}$, there exist $C,C'>0$ such that for all $(n_{1},x_{1}) \in K$, $\kappa|\Re(v_{1})+i\kappa|^{T} \left(\Re(v_{1})-\frac{x_{1}}{T}\right)^{-n_{1}-T-1}$ is bounded by $Ce^{C'T\kappa}\kappa^{-(n_{1}-1)/2},$  and thus there exist $C'',C'''>0$ such that under the same assumptions,
\begin{align*}
   \bigg|\int_{v_{1}}dw\frac{w^{T}}{\left(w-\frac{x_{1}}{T}\right)^{n_{1}+T+1}} e^{TS_{m}(\Re(v_{1})+i\kappa)-TS_{m}(w)}\bigg|  \leq C''e^{C'''T\kappa}\kappa^{-(n_{1}-1)/2}. 
\end{align*} 
Therefore, under the relative scale assumptions assumptions of $\kappa,T,m$, for fixed compact $K\subset \mathbb{N}\times \mathbb{R}$ there exists $C>0$ such that for all $(n_{1},x_{1}) \in K$, 
\begin{multline*}
     \bigg|\int_{v_{1}}dw\frac{w^{T}}{\left(w-\frac{x_{1}}{T}\right)^{n_{1}+T+1}} e^{TS_{m}(z_{0}^{(m)})-TS_{m}(w)}\bigg| 
     \leq Ce^{T\Re(S_{m}(z_{0}^{(m)}))-T\Re(S_{m}(z_{1}^{(m)}))}e^{T\Re(z_{1}^{(m)})X - T\Re(v_{1})X }
     \\ \cdot \bigg|\prod_{r=1}^{m}\frac{\Re(v_{1})+i\kappa-u_{r}^{(m)}}{\Re(z_{1}^{(m)})+i\kappa-u_{r}^{(m)}} \bigg| \bigg|\int_{v_{1}}dw\frac{w^{T}}{\left(w-\frac{x_{1}}{T}\right)^{n_{1}+T+1}}e^{TS_{m}(\Re(v_{1})+i\kappa)-TS_{m}(w)}\bigg|. 
\end{multline*}
Applying \eqref{e:bd3}, \cref{r:extra}, and since $|\Re(v_{1})-\Re(z_{1}^{(m)})|$ is bounded by $|x_{1}|T^{-1}+\kappa$, and $|\Re(z_{1}^{(m)})|\in (\kappa^{1/2},C)$ for some $C>0$, for fixed compact $K\subset \mathbb{N}\times\mathbb{R}$ there exist $C',d>0$ such that for all $(n_{1},x_{1}) \in K$, 
\begin{multline*}
    \leq e^{T\Re(S_{m}(z_{0}^{(m)}))-T\Re(S_{m}(z_{1}^{(m)}))}e^{T(|x_{1}|T^{-1}+\kappa)|X|}\bigg|\int_{v_{1}}dw\frac{w^{T}}{\left(w-\frac{x_{1}}{T}\right)^{n_{1}+T+1}}e^{TS_{m}(\Re(v_{1})+i\kappa)-TS_{m}(w)}\bigg| 
     \\ \cdot \exp{\left(\frac{T|X|(|x_{1}|T^{-1}+\kappa)}{2} + C''(|x_{1}|T^{-1}+\kappa) T e_{m}\kappa^{-1} \right)},
\end{multline*} is bounded by $C'e^{-dT},$ (all factors involving $X$ may be chosen uniformly over $X\in (-2+\alpha,2-\alpha)$.

   The same bound applies to the integral over $v_{2}$. We now address the integral over $p_{1}\cup p_{2}.$ Applying the fact that the length of $p_{1},p_{2}$ is bounded by $|x_{1}|T^{-1}+\kappa$, and the scale assumptions of $\kappa,T,m,$ for fixed compact $K\subset \mathbb{N}\times \mathbb{R}$ there exists $C,c>0$ such that for all $(n_{1},x_{1})\in K$,
   \begin{multline*}
        \bigg| \int_{p_{1}} \frac{w^{T}}{\left(w-\frac{x_{1}}{T}\right)^{n_{1}+T+1}}e^{TS_{m}(z_{1}^{(m)})-TS_{m}(w)}\bigg| 
        \\  \leq C\kappa e^{cT\kappa } \cdot \bigg|\prod_{r=1}^{m}\frac{\Re(z_{1}^{(m)})+|x_{1}|T^{-1}+\kappa+i\kappa-u_{r}^{(m)}}{\Re(z_{1}^{(m)})+i\kappa -u_{r}^{(m)}}\bigg|\sup_{w\in p_{1}}\frac{|w|^{T}}{|w-x_{1}/T|^{n_{1}+T+1}},
   \end{multline*} by applying \eqref{e:bd3}, there exist $C',c>0$ such that
   \begin{align*}
        \bigg| \int_{p_{1}} \frac{w^{T}}{\left(w-\frac{x_{1}}{T}\right)^{n_{1}+T+1}}e^{TS_{m}(z_{1}^{(m)})-TS_{m}(w)}\bigg| 
        \leq C'\kappa^{-(n_{1}-1)/2} e^{cT\kappa}.
   \end{align*}
   Therefore, applying \cref{r:extra}, for fixed compact $K\subset\mathbb{N}\times \mathbb{R}$ there exist $C'',d>0$ such that for all $(n_{1},x_{1}) \in K$, 
   \begin{multline*}
        \bigg| \int_{p_{1}} \frac{w^{T}}{\left(w-\frac{x_{1}}{T}\right)^{n_{1}+T+1}}e^{TS_{m}(z_{0}^{(m)})-TS_{m}(w)}\bigg| 
        \\ \leq e^{T\Re(S_{m}(z_{0}^{(m)})-S_{m}(z_{1}^{(m)}))}\bigg| \int_{p_{1}} \frac{w^{T}}{\left(w-\frac{x_{1}}{T}\right)^{n_{1}+T+1}}e^{TS_{m}(z_{1}^{(m)})-TS_{m}(w)}\bigg|, 
   \end{multline*}  is bounded by $C''e^{-dT}$.
   The same bound applies over $p_{2}$.
\end{proof}
We also comment on some bounds which apply over the contours inside $\mathcal{R}(c_{0},\kappa)$.
\begin{remark}\label{r:contourminus}
    Assume $\kappa\to 0$ as $m\to\infty$ and that $\log^{2}{(m)}T^{-1/2}\ll \kappa$. Also assume that $\log^{4}{(m)}\ll T\ll m^{2/3}$. By recalling the bounds on $A_{m}(w),B_{m}(z)$ from \cref{r:Rbd} and $L(c,\kappa)$ from \cref{l:lengthbound}, we conclude that for all $\alpha\in (0,2)$ and $K\subset\mathbb{N}\times \mathbb{R}$ compact, there exist $C:=C(K,\alpha),c:=c(K,\alpha)>0$ and $m_{0}:=m_{0}(K,\alpha)\in\mathbb{N}$ such that for all $m>m_{0}$, , $(n_{1},x_{1}),(n_{2},x_{2})\in K$, $x^{(m)}\in\Omega_{m}$, and $X\in (-2+\alpha,2-\alpha),$
    \begin{align*}
        \bigg| \int_{C_{1}^{(m)}}dz z^{-T}\left(z-\frac{x_{2}}{T}\right)^{n_{2}+T-1} e^{TS_{m}(z)- TS_{m}(z_{0}^{(m)})} \bigg| & \leq C\kappa^{-n_{2}-1} e^{c\kappa^{-1}},
        \\ \bigg| \int_{C_{1}^{(m),-}}dz z^{-T}\left(z-\frac{x_{2}}{T}\right)^{n_{2}+T-1} e^{TS_{m}(z)- TS_{m}(\overline{z_{0}^{(m)}})} \bigg| & \leq C\kappa^{-n_{2}-1} e^{c\kappa^{-1}},
        \\ \bigg| \int_{C_{2}^{(m)}}dw w^{T}\left(w-\frac{x_{1}}{T}\right)^{-n_{1}-T-1} e^{TS_{m}(z_{0}^{(m)})- TS_{m}(w)} \bigg|  & \leq C\kappa^{-n_{1}-3}e^{c\kappa^{-1}},
        \\ \bigg| \int_{C_{2}^{(m),-}}dw w^{T}\left(w-\frac{x_{1}}{T}\right)^{-n_{1}-T-1} e^{TS_{m}(\overline{z_{0}^{(m)}})- TS_{m}(w)} \bigg|  & \leq C\kappa^{-n_{1}-3}e^{c\kappa^{-1}}.
    \end{align*}
\end{remark} 

Finally, with all of these estimates now established, we can prove \cref{l:errortermscross}.
\begin{proof}[Proof of \cref{l:errortermscross}]  
In this argument, we will only display the calculations in full for the upper half-plane. The equivalent statements in the lower half-plane follow by essentially the same calculations.

We start with the first expression. We separate the left-hand side into four terms, which we will treat individually, 
\begin{align*}
    & \bigg|\frac{n_{1}+T}{T} \frac{-1}{(2\pi i)^{2}}\int_{\{\Re(z_{4}^{(m)})+iy||y|\geq \kappa^{-1}\} \cup v \cup h_{1}\cup h_{2}} dz \int_{C_{w}^{(m)}}dw \frac{\left(z-\frac{x_{2}}{T}\right)^{n_{2}+T-1}w^{T}}{z^{T}\left(w-\frac{x_{1}}{T}\right)^{n_{1}+T+1}}  \frac{e^{T\left(S_{m}(z) -S_{m}(w)\right) }}{w-z} \bigg|
    \\ & \leq  \bigg|\frac{n_{1}+T}{T} \frac{-1}{(2\pi i)^{2}}\int_{\{\Re(z_{4}^{(m)})+iy||y|\geq \kappa^{-1}\}} dz \int_{C_{2}^{(m)}\cup C_{2}^{(m),-}}dw \frac{\left(z-\frac{x_{2}}{T}\right)^{n_{2}+T-1}w^{T}}{z^{T}\left(w-\frac{x_{1}}{T}\right)^{n_{1}+T+1}}  \frac{e^{T\left(S_{m}(z) -S_{m}(w)\right) }}{w-z} \bigg|
    \\ & + \bigg|\frac{n_{1}+T}{T} \frac{-1}{(2\pi i)^{2}}\int_{\{\Re(z_{4}^{(m)})+iy||y|\geq \kappa^{-1}\}} dz \int_{C_{w}^{(m)}\setminus (C_{2}^{(m)}\cup C_{2}^{(m),-})}dw \frac{\left(z-\frac{x_{2}}{T}\right)^{n_{2}+T-1}w^{T}}{z^{T}\left(w-\frac{x_{1}}{T}\right)^{n_{1}+T+1}}  \frac{e^{T\left(S_{m}(z) -S_{m}(w)\right) }}{w-z} \bigg|
    \\ & + \bigg|\frac{n_{1}+T}{T} \frac{-1}{(2\pi i)^{2}}\int_{ v \cup h_{1}\cup h_{2}} dz \int_{C_{2}^{(m)}\cup C_{2}^{(m),-}}dw \frac{\left(z-\frac{x_{2}}{T}\right)^{n_{2}+T-1}w^{T}}{z^{T}\left(w-\frac{x_{1}}{T}\right)^{n_{1}+T+1}}  \frac{e^{T\left(S_{m}(z) -S_{m}(w)\right) }}{w-z}\bigg|
    \\ & + \bigg|\frac{n_{1}+T}{T} \frac{-1}{(2\pi i)^{2}}\int_{ v \cup h_{1}\cup h_{2}} dz \int_{C_{w}^{(m)}\setminus(C_{2}^{(m)}\cup C_{2}^{(m),-})}dw \frac{\left(z-\frac{x_{2}}{T}\right)^{n_{2}+T-1}w^{T}}{z^{T}\left(w-\frac{x_{1}}{T}\right)^{n_{1}+T+1}}  \frac{e^{T\left(S_{m}(z) -S_{m}(w)\right) }}{w-z}\bigg|.
\end{align*} 
To bound the first term, we note that  there exists $c>0$ such that $|w-z|\geq c^{-1}\kappa^{-1}$ for all $w,z$ on those contours, for some $c>0$ which can be chosen uniformly over $(n_{1},x_{1}),(n_{2},x_{2})\in K$ compact. Therefore, applying \cref{l:zrealbd} and \cref{r:contourminus}, the expression 
\begin{multline*}
    \bigg|\frac{n_{1}+T}{T} \frac{-1}{(2\pi i)^{2}}\int_{\{\Re(z_{4}^{(m)})+iy||y|\geq \kappa^{-1}\}} dz \int_{C_{2}^{(m)}\cup C_{2}^{(m),-}}dw \frac{\left(z-\frac{x_{2}}{T}\right)^{n_{2}+T-1}w^{T}}{z^{T}\left(w-\frac{x_{1}}{T}\right)^{n_{1}+T+1}}  \frac{e^{T\left(S_{m}(z) -S_{m}(w)\right) }}{w-z} \bigg|
    \\ \leq c'\kappa \bigg|\int_{C_{2}^{(m)}}dw \frac{ w^{T}}{ \left(w-\frac{x_{1}}{T}\right)^{n_{1}+T+1}}  e^{T\left(S_{m}(z_{0}^{(m)}) -S_{m}(w)\right) }\bigg|
    \\ \cdot \bigg|\int_{\{\Re(z_{4}^{(m)})+iy|y\geq \kappa^{-1}\}} dz \frac{\left(z-\frac{x_{2}}{T}\right)^{n_{2}+T-1}}{z^{T}}  e^{T\left(S_{m}(z) -S_{m}(z_{0}^{(m)})\right) } \bigg|, 
\end{multline*} is bounded by $C'\kappa^{-n_{1}-2}e^{c\kappa^{-1}}e^{-d'T}\leq C''e^{-d''T}$, with $C'',d''>0$ uniform over $(n_{1},x_{1}),(n_{2},x_{2})\in K$. 

To deal with the second term, we can apply the same reasoning to conclude that $|w-z|\geq c^{-1}\kappa^{-1}$ for some $c>0$ which can be chosen uniformly over all $(n_{1},x_{1}),(n_{2},x_{2})\in K$ compact, and thus, applying \cref{l:zrealbd} and \cref{l:wrealbd}, the expression
\begin{multline*}
    \bigg|\frac{n_{1}+T}{T} \frac{-1}{(2\pi i)^{2}}\int_{\{\Re(z_{4}^{(m)})+iy||y|\geq \kappa^{-1}\}} dz \int_{C_{w}^{(m)}\setminus (C_{2}^{(m)}\cup C_{2}^{(m),-})}dw \frac{\left(z-\frac{x_{2}}{T}\right)^{n_{2}+T-1}w^{T}}{z^{T}\left(w-\frac{x_{1}}{T}\right)^{n_{1}+T+1}}  \frac{e^{T\left(S_{m}(z) -S_{m}(w)\right) }}{w-z} \bigg|
    \\ \leq c'\kappa \bigg| \int_{\{\Re(z_{4}^{(m)})+iy|y\geq \kappa^{-1}\}} dz \frac{\left(z-\frac{x_{2}}{T}\right)^{n_{2}+T-1}e^{TS_{m}(z)-TS_{m}(z_{0}^{(m)})}}{z^{T}} \bigg|
    \\ \cdot \bigg|\int_{C_{w}^{(m)}\setminus (C_{2}^{(m)}\cup C_{2}^{(m),-})}dw \frac{w^{T}e^{TS_{m}(z_{0}^{(m)})-TS_{m}(w)}}{\left(w-\frac{x_{1}}{T}\right)^{n_{1}+T+1}}\bigg|,
\end{multline*} is bounded by $C'\kappa e^{-d'T}$, where $C',d'>0$ are chosen uniformly over $(n_{1},x_{1}),(n_{2},x_{2})\in K$. 

To deal with the third term, we note that we have defined the contours so that for $m$ sufficiently large, $|w-x_{1}/T|>c^{-1}\kappa$ for a constant $c>0$ which can be chosen uniformly over $(n_{1},x_{1})\in K$ compact. Therefore, we apply \cref{l:zrealbd} and \cref{r:contourminus} to conclude that  
 \begin{align*}
     \bigg|\frac{n_{1}+T}{T} \frac{-1}{(2\pi i)^{2}}\int_{ v \cup h_{1}\cup h_{2}} dz \int_{C_{2}^{(m)}\cup C_{2}^{(m),-}}dw \frac{\left(z-\frac{x_{2}}{T}\right)^{n_{2}+T-1}w^{T}}{z^{T}\left(w-\frac{x_{1}}{T}\right)^{n_{1}+T+1}}  \frac{e^{T\left(S_{m}(z) -S_{m}(w)\right) }}{w-z}\bigg|
     \\ \leq c'\kappa^{-1}\bigg|\int_{ v \cup h_{1}\cup h_{2}} dz z^{-T}\left(z-\frac{x_{2}}{T}\right)^{n_{2}+T-1}e^{T\left(S_{m}(z) -S_{m}(z_{0}^{(m)})\right) } \bigg| 
     \\ \cdot \bigg|\int_{C_{2}^{(m)}}dw \frac{w^{T}}{\left(w-\frac{x_{1}}{T}\right)^{n_{1}+T+1}}  e^{T\left(S_{m}(z_{0}^{(m)}) -S_{m}(w)\right) }\bigg|,
 \end{align*} is bounded by $C'\kappa^{-n_{1}-4}e^{c\kappa^{-1}}e^{-d'T}\leq C''e^{-d''T}$, where $C'',d''>0$ are chosen uniformly over $(n_{1},x_{1}),(n_{2},x_{2})\in K$.  
To bound the fourth term, we note that for $m$ sufficiently large, the contours are separated by $c^{-1}\kappa$, where, as before, we can choose this constant uniformly over $(n_{1},x_{1}),(n_{2},x_{2})\in K$ compact, and therefore we apply \cref{l:zrealbd} and \cref{l:wrealbd} to conclude that 
 \begin{align*}
    \bigg|\frac{n_{1}+T}{T} \frac{-1}{(2\pi i)^{2}}\int_{ v \cup h_{1}\cup h_{2}} dz \int_{C_{w}^{(m)}\setminus(C_{2}^{(m)}\cup C_{2}^{(m),-})}dw \frac{\left(z-\frac{x_{2}}{T}\right)^{n_{2}+T-1}w^{T}}{z^{T}\left(w-\frac{x_{1}}{T}\right)^{n_{1}+T+1}}  \frac{e^{T\left(S_{m}(z) -S_{m}(w)\right) }}{w-z}\bigg|
    \\ \leq C'\kappa^{-1} \bigg|\int_{ v \cup h_{1}\cup h_{2}} dz z^{-T}\left(z-\frac{x_{2}}{T}\right)^{n_{2}+T-1}e^{T\left(S_{m}(z) -S_{m}(z_{0}^{(m)})\right) }\bigg| \\ \cdot \bigg| \int_{C_{w}^{(m)}\setminus(C_{2}^{(m)}\cup C_{2}^{(m),-})}dw \frac{w^{T}}{\left(w-\frac{x_{1}}{T}\right)^{n_{1}+T+1}}  e^{T\left(S_{m}(z_{0}^{(m)}) -S_{m}(w)\right) }\bigg|,
\end{align*} is bounded by $C''\kappa^{-1}e^{-d'T}$ where $C'',d'>0$ are chosen uniformly over $(n_{1},x_{1}),(n_{2},x_{2})\in K$. 

Putting this together, we see that there exist $C,d>0,m_{0}\in\mathbb{N}$ such that for all $(n_{1},x_{1}),(n_{2},x_{2})\in K$, $x^{(m)}\in \Omega_{m}$ and $m>m_{0}$, 
\begin{align*}
    \bigg|\frac{n_{1}+T}{T} \frac{-1}{(2\pi i)^{2}}\int_{\{\Re(z_{4}^{(m)})+iy||y|\geq \kappa^{-1}\} \cup v \cup h_{1}\cup h_{2}} dz \int_{C_{w}^{(m)}}dw \frac{\left(z-\frac{x_{2}}{T}\right)^{n_{2}+T-1}w^{T}}{z^{T}\left(w-\frac{x_{1}}{T}\right)^{n_{1}+T+1}}  \frac{e^{T\left(S_{m}(z) -S_{m}(w)\right) }}{w-z} \bigg| \leq Ce^{-dT}.  
\end{align*}

We proceed to study the second expression. We expand the left-hand side into several terms, as before,
\begin{align*}
    & \bigg|\frac{n_{1}+T}{T} \frac{-1}{(2\pi i)^{2}}\int_{C_{z}^{(m)}} dz \int_{C_{w}^{(m)}\setminus (C_{2}^{(m)}\cup C_{2}^{(m),-})}dw R_{m}(z,w)  \frac{e^{T\left(S_{m}(z) -S_{m}(w)\right) }}{w-z}\bigg|
    \\ & \leq \bigg|\frac{n_{1}+T}{T} \frac{-1}{(2\pi i)^{2}}\int_{C_{z}^{(m)}\setminus (C_{1}^{(m)}\cup C_{1}^{(m),-})} dz \int_{v_{1}\cup v_{2}\cup p_{1}\cup p_{2}}dw R_{m}(z,w)  \frac{e^{T\left(S_{m}(z) -S_{m}(w)\right) }}{w-z}\bigg|
    \\ & + \bigg|\frac{n_{1}+T}{T} \frac{-1}{(2\pi i)^{2}}\int_{C_{1}^{(m)}\cup C_{1}^{(m),-}} dz \int_{v_{1}\cup v_{2}\cup p_{1}\cup p_{2}}dw R_{m}(z,w)  \frac{e^{T\left(S_{m}(z) -S_{m}(w)\right) }}{w-z}\bigg|.
\end{align*} The first term is already bounded by $Ce^{-dT}$ by the preceding discussion, and to bound the second term, we note that for $m$ sufficiently large there exists $c>0$ such that $|w-z|\geq c^{-1}\kappa$, and thus 
\begin{multline*}
    \bigg|\frac{n_{1}+T}{T} \frac{-1}{(2\pi i)^{2}}\int_{C_{1}^{(m)}\cup C_{1}^{(m),-}} dz \int_{v_{1}\cup v_{2}\cup p_{1}\cup p_{2}}dw \frac{\left(z-\frac{x_{2}}{T}\right)^{n_{2}+T-1}w^{T}}{z^{T}\left(w-\frac{x_{1}}{T}\right)^{n_{1}+T+1}}  \frac{e^{T\left(S_{m}(z) -S_{m}(w)\right) }}{w-z}\bigg|
    \\ \leq c'\kappa^{-1} \bigg|\frac{n_{1}+T}{T} \frac{-1}{(2\pi i)^{2}}\int_{C_{1}^{(m)}} dz z^{-T}\left(z-\frac{x_{2}}{T}\right)^{n_{2}+T-1}e^{T\left(S_{m}(z) -S_{m}(z_{0}^{(m)})\right) }\bigg| 
    \\ \cdot \bigg|\int_{v_{1}\cup v_{2}\cup p_{1}\cup p_{2}}dw \frac{w^{T}}{\left(w-\frac{x_{1}}{T}\right)^{n_{1}+T+1}}   e^{T\left(S_{m}(z_{0}^{(m)}) -S_{m}(w)\right) }\bigg|,
\end{multline*}
which is bounded by $C'\kappa^{-n_{2}-2} e^{c\kappa^{-1}}e^{-d'T}\leq C''e^{-d'T}$ for $C'',d'>0$ chosen uniformly over $(n_{1},x_{1}),(n_{2},x_{2})\in K$. This concludes the argument.
 \end{proof}

\section{Proof of \cref{t:gaps}}\label{s:carg}

In this section we prove \cref{t:gaps}. The main techniques in this section are inspired by or modified from those employed in other papers, especially \cite{EKYY12,EYY12,T13,TV11}. We outline modifications of the Lindeberg replacement argument in \cite{EYY12} which allow us to show that if two $n\times n$ Wigner matrices have entries matching up to two moments, and share all but $n\times\text{polylog}(n)$ entries, then the difference of the expected value of observables of each of these matrices decays in $n$.

We begin by introducing a helpful notation. 
\begin{definition}\label{d:2momentpair}
    We will say that $A,B$ are an $(m,T)$-pair if $A=(a_{ij})_{i,j=1}^{m+T}$ is an $(m+T)\times (m+T)$ Wigner matrix in the sense of \cref{d:wigner}, and $B=(b_{ij})_{i,j=1}^{m+T}$ is a $(m+T)\times (m+T)$ random matrix such that for all index pairs $(i,j)$ with $i\leq m$ and $j\leq m$, $b_{ij}=a_{ij}$, and for the remaining $(m+T)^{2}-m^{2}$ entries, 
    \begin{align*}
        b_{ij} & :=\begin{cases}
            2^{-1/2}(X_{ij} + i Y_{ij}) & i<j
            \\ 2^{-1/2}(X_{ji} -i Y_{ji}) & i>j
            \\ X_{ij} & i=j
        \end{cases}.
    \end{align*}
    with $\{X_{ij}\}_{1\leq i\leq j\leq m+T},\{Y_{ij}\}_{1\leq i<j\leq m+T}$ mutually independent (and also independent of all $\{a_{ij}\}_{i,j=1}^{m+T}$) standard Gaussian. Under these conditions, all $(m,T)$-pairs $A,B$ have all entries matching moments to order $2$ in the sense of \cref{d:matching}. 
\end{definition}

We obtain a comparison theorem for fixed-energy correlation measures of $(m,T)$-pairs. 
\begin{proposition}\label{t:couplingcost} Fix $\alpha\in (0,2)$ and let $A,B$ be an $(m,T)$-pair and set $n:=m+T$. Let $\widehat{\rho}_{k,X}^{A,(n)},\widehat{\rho}_{k,X}^{B,(n)}$ be the rescaled $k$-point correlation measures of the eigenvalues of $A$ and $B$, respectively (see \cref{d:cormeas}). Then, for all $k\in\mathbb{N}$, and for every compact set $U$ and differentiable compactly supported test function $F\in C^{1}_{c}(\mathbb{R}^{k})$ with $\mathrm{supp}(F)\subset U^{k}$, there exist $C:=C(U,k,\alpha)>0$ and $c:=c(k,\alpha)>0$ such that for all $X\in (-2+\alpha,2-\alpha)$, 
    \begin{align*}
    \left|\int_{\mathbb R^k} F(x_{1},\cdots x_{k})\widehat{\rho}_{k,X}^{A,(n)}(dx_{1},\cdots,dx_{k})-\int_{\mathbb R^k} F(x_{1},\cdots,x_{k}) \widehat{\rho}_{k,X}^{B,(n)}(dx_{1},\cdots, dx_{k})\right|\leq C T n^{-c}\|F\|_{C^{1}} .
    \end{align*}
\end{proposition} 
We outline the proof of \cref{t:couplingcost} in two steps, first showing the same result for Wigner matrices with an added uniform exponential decay condition, and then showing we can remove that condition for the price of a small error. The comparison argument is a standard Green function/Lindeberg
replacement argument. Since the required modifications of the existing proofs
are minor, we
only record the precise comparison statement needed below and indicate the changes from the proof of \cite[Theorem 6.4]{EYY12}. The next lemma is similar to \cite[Theorem 2.3]{EYY12}, except that we require matching only to the second moment, rather than fourth, and we impose the condition of the large shared $m\times m$ submatrix.
\begin{lemma}\label{l:2moment}
    Set $n:=m+T$ and let $\widetilde{A},\widetilde{B}$ be an $(m,T)$-pair, set $n:=m+T$, and define $A:=n^{-1/2}\widetilde{A}$ and $B:=n^{-1/2}\widetilde{B}$, and impose the condition that the entries satisfy a uniform exponential decay condition, meaning that there exist $\alpha,\beta>0$ such that 
    \begin{align*}
        \mathbb{P}\left(|a_{ij}|\geq x^{\alpha}\right)\leq \beta e^{-x}, & & \mathbb{P}\left(|b_{ij}|\geq x^{\alpha}\right) \leq \beta e^{-x}.
    \end{align*} 
    
    If we fix a bijective ordering map on the index set of these elements outside of their $m\times m$ shared submatrix, 
    \begin{align*}
        \phi:\{\{i,j\}|1\leq i\leq j\leq n\}\setminus \{(i,j)|1\leq i\leq j\leq m\} \to \left\{1,\cdots ,\Gamma(T,m)\right\},
    \end{align*} where $\Gamma(T,m):=2^{-1}(n(n+1) - m(m+1)) = mT+T(T+1)/2$ and denote by $A_{\gamma}$ the matrix whose elements take the form $b_{ij}$ for all $\phi(i,j)\leq \gamma$ and $a_{ij}$ otherwise, so that $B=A_{\Gamma(T,m)},A=A_{0}$.

    Let $\kappa>0$ be arbitrary, and suppose that for any small parameter $\tau>0$ and $y\geq n^{\tau-1}$, we have 
    \begin{align}\label{e:loglog}
        \mathbb{P}\left(\max_{0\leq \gamma\leq \Gamma(T,m)}\max_{1\leq k\leq n}\sup_{|X|\leq 2-\kappa}\bigg|G(X+iy,A_{\gamma})_{kk} \bigg| \leq n^{2\tau}\right) \geq 1-Cn^{-c\log{\log{(n)}}}
    \end{align} with some constants $C,c>0$ depending on $\tau,\kappa$, and where $G$ is the Green's function. Furthermore, assume that $a_{ij}$ and $b_{ij}$ match moments to order $2$. 

    Let $\varepsilon_{1}>0$ be arbitrary and choose an $\eta$ so that $n^{-1-\varepsilon_{1}}\leq \eta\leq n^{-1}$. For any sequence of positive integers $k_{1},\cdots,k_{p}$ for some $p\in\mathbb{N}$, set complex parameters $z_{j}^{(r)}=E_{j}^{r}\pm i\eta$, $j\in [[k_{r}]],r\in[[p]]$, with $|E_{j}^{r}|\leq 2-\kappa$ and with arbitrary choice sign $\pm$ for each $z_{j}^{(r)}$. Let $F:\mathbb{C}^{p}\to\mathbb{R}$ be such that for any multi-index $\vec\alpha=(\alpha_{1},\cdots,\alpha_{p})$ with $1\leq |\vec\alpha|\leq 5$ and there exists $c_{0}>0$ such that for any $\varepsilon_{2}>0$ sufficiently small,
    \begin{align*}
        & \max\left\{|\partial^{\alpha}F(x_{1},\cdots,x_{p})|\bigg| \max_{j}|x_{j}|\leq n^{\varepsilon_{2}}\right\}\leq n^{c_{0}\varepsilon_{2}}, \\ & \max\left\{|\partial^{\alpha}F(x_{1},\cdots,x_{p})| \bigg| \max_{j}|x_{j}|\leq n^{2}\right\} \leq n^{c_{0}}.
    \end{align*}
    Then there is a constant $c_{1}$ depending on $\alpha,\beta,\sum_{r=1}^{p}k_{r},c_{0}$ such that for any $\eta$ with $n^{-1-\varepsilon_{1}}\leq \eta\leq n^{-1}$ and for any choices of the signs in the imaginary part of the $z^{(r)}_{j}$,
    \begin{multline*}
        \bigg|\mathbb{E}\left[F\left(n^{-k_{1}}\mathrm{Tr}\left[\prod_{j=1}^{k_{1}}G(z^{(1)}_{j},A)\right],\cdots,n^{-k_{p}}\mathrm{Tr}\left[\prod_{j=1}^{k_{p}}G(z^{(p)}_{j},A)\right]\right)\right] 
        \\ -\mathbb{E}\left[F\left(n^{-k_{1}}\mathrm{Tr}\left[\prod_{j=1}^{k_{1}}G(z^{(1)}_{j},B)\right],\cdots,n^{-k_{p}}\mathrm{Tr}\left[\prod_{j=1}^{k_{p}}G(z^{(p)}_{j},B)\right]\right)\right] \bigg| \\ \leq c_{1}T\left(1-\frac{T-1}{2n}\right)n^{-1/2+C\varepsilon_{1}}.
        \end{multline*}
\end{lemma}
The proof is the standard telescoping Green's function
comparison. The only changes from \cite[Theorem 6.4]{EYY12} are that
the interpolation has $\Gamma(T,m)=mT+T(T+1)/2$ swaps rather than $O(n^{2})$
swaps. Since the swapped entries match to second order, the resolvent expansion cancels to second order, and the remaining terms give a single-swap error of 
$O(n^{-3/2+C\varepsilon_1})$. 
Summing over the $\Gamma(T,m)$ swaps gives the stated $Tn^{-1/2+C\varepsilon_1}$ 
bound. The Green's function bound \eqref{e:loglog} follows from the local semicircle law under the stated tail assumption,
as in \cite[Theorem 3.1]{EYY12}.
\begin{proof}[Proof of \cref{t:couplingcost}]
    By an argument identical to that in the proof of \cite[Theorem 6.4]{EYY12}, which explains how to extract uniform convergence of the $k$-point correlation functions from results of the form of \cref{l:2moment}, we conclude that for all $k\in\mathbb{N}$ and for every compact set $U$ and differentiable compactly supported test function $F\in C^{1}_{c}(\mathbb{R}^{k})$ with $\mathrm{supp}(F)\subset U^{k}$, there exists $C:=C(X,U,k)>0$ and $c:=c(X,k)>0$ such that 
    \begin{align}\label{e:swapproof}
    \left|\int_{\mathbb R^k} F(x_{1},\cdots x_{k})\widehat{\rho}_{k,X}^{A,(n)}(dx_{1},\cdots,dx_{k})-\int_{\mathbb R^k} F(x_{1},\cdots,x_{k}) \widehat{\rho}_{k,X}^{B,(n)}(dx_{1},\cdots, dx_{k})\right|\leq C T n^{-c}\|F\|_{C^{1}} .
    \end{align} If we instead allowed any $F\in C_{c}(\mathbb{R}^{k})$, we would possibly lose the quantitative error, though it is still possible to prove that the left-hand side above would be $o(1)$. 
    Finally, applying \cite[Lemma 7.6]{EKYY12} with  $\lambda=n^{c_{1}}$ for some $c_{1}>0$, we use the notation $\widetilde{a}_{ij},\widetilde{b}_{ij}$ to denote random variables which are equal to $a_{ij},b_{ij}$ except on an event of measure $O(\lambda^{-(4+\varepsilon)})$, and such that $|\widetilde{a}_{ij}|,|\widetilde{b}_{ij}|\leq \lambda$. This implies that there exists $C>0$ such that  
    \begin{align*}
        \mathbb{P}\left(E\right)\leq Cn^{2-c_{1}(4+\varepsilon)}, & & E:= \bigcup_{i,j} \{a_{ij}\neq \widetilde{a}_{ij}\}\cup \bigcup_{i,j}\{b_{ij}\neq \widetilde{b}_{ij}\}.
    \end{align*} We can apply the existing bound \eqref{e:swapproof}, on $E^{c}$. Therefore, since the left-hand side of \eqref{e:swapproof} is bounded by a constant,  all that is required is that $2(4+\varepsilon)^{-1}<c_{1}$, which holds for any $c_{1}\geq 1/2.$  We finish the proof by noting that for any fixed $\alpha\in (0,2)$ we can choose constants $C,c>0$ so that this bound is satisfied uniformly for all $X\in (-2+\alpha,2-\alpha).$
\end{proof}
We are now in a position to prove \cref{t:gaps}. 
\begin{proof} We set $n:=m+T$ for the duration of this argument. 
    Applying \cref{p:unif2jan} to the uniform convergence results in \cref{t:convrate} and \cref{l:gueunif}, and noting that $\mathbb{P}(\Omega_{m})\geq 1-m^{-25}$, we conclude that for any $\alpha\in (0,2)$,  $k\in\mathbb{N}$, $U\subset\mathbb{R}$ compact, and $F\in C^{1}_{c}(\mathbb{R}^{k})$ with $\mathrm{supp}(F)\subset U^{k}$, there exist $C:=C(U,k,\alpha)>0,c:=c(k,\alpha)>0,n_{0}:=n_{0}(U,k,\alpha)\in\mathbb{N}$ such that for all $X\in (-2+\alpha,2-\alpha)$, and $n>n_{0}$,
    \begin{multline*}
        \bigg|\int_{\mathbb R^k} F(x_{1},\ldots x_{k})\widehat{\rho}_{k,X;x^{(m)}}^{\textsc{GUE},(m,T)}(dx_{1},\ldots,dx_{k})
        \\ -\int_{\mathbb R^k} F(x_{1},\ldots,x_{k}) \widehat{\rho}_{k,X}^{\textsc{GUE},(n)}(dx_{1},\ldots, dx_{k})\bigg|\leq C( T^{-1/2}+n^{-1})\|F\|_{C^{1}}.
    \end{multline*}  In this expression $\widehat{\rho}_{k,X;x^{(m)}}^{\textsc{GUE},(m,T)}$ is the correlation measure associated to the kernel in \cref{d:cormeas}. 
    Since the pair of Wigner matrices formed by taking an $m\times m$ Wigner matrix and adding, in one case, $T$ Wigner rows and columns, and, in the other case, $T$ $\textsc{GUE}$ rows and columns, form an $(m,T)$-pair as in \cref{d:2momentpair}, we can apply \cref{t:couplingcost}, taking constants uniformly over $X\in (-2+\alpha,2-\alpha)$, to obtain 
    \begin{multline*}
        \bigg|\int_{\mathbb R^k} F(x_{1},\ldots x_{k})\widehat{\rho}_{k,X;x^{(m)}}^{W,(n)}(dx_{1},\ldots,dx_{k})
        \\ -\int_{\mathbb R^k} F(x_{1},\ldots,x_{k}) \widehat{\rho}_{k,X}^{\textsc{GUE},(n)}(dx_{1},\ldots, dx_{k})\bigg|\leq C( T^{-1/2}+n^{-1} + Tn^{-c})\|F\|_{C^{1}}.
    \end{multline*}
    Finally, taking $T=n^{d}$ for some $d\in (0,\min\{c,2/3\})$, we obtain an overall polynomial error rate in $n$. 
\end{proof}

\appendix

\section{Modified Eynard-Mehta Theorem}\label{a:em}
In the proof below, we use the notion of an $L$-ensemble and conditional $L$-ensemble, which are explained in \cite{Bor11}.
\begin{proof}[Proof of \cref{t:em}] 
    We apply the $L$-ensemble approach to the point process from level $m$ to level $m+L$, following \cite{BR05,Bor11}. 
We note that the weight $W(X,x^{(m)})$ is proportional to
    \begin{align*} 
\begin{split}
\det\begin{bmatrix}
[\upsilon_{i}(x_{j}^{(m)}|m) ]_{i,j=1}^{m}  & \mathbf{0}_{m\times L}\\
\mathbf{0}_{L \times m} & \mathbf{1}_{L\times L}
\end{bmatrix}\prod_{r=m+1}^{m+L}
&\det \begin{bmatrix}
[\varphi (x_i^{(r-1)} , x_j^{(r)}) ]_{
i,j=1}^{r} & \mathbf{0}_{r\times (m+L-r)}\\
\mathbf{0}_{(m+L-r)\times r} & \mathbf{1}_{(m+L-r)\times (m+L-r)} 
\end{bmatrix}
\det[\psi_{i}(x_j^{(m+L)}|m+L)]_{i,j=1}^{m+L}.
\end{split}
\end{align*}
We define
\begin{align*}
    \Upsilon(x^{(m)}) & :=\begin{bmatrix}
[\upsilon_{i}(x_{j}^{(m)}|m) ]_{i,j=1}^{m}  & \mathbf{0}_{m\times L}\\
\mathbf{0}_{L \times m} & \mathbf{1}_{L\times L}
\end{bmatrix} \\
    \Phi_{r}(x^{(r-1)},x^{(r)}) & : = \begin{bmatrix}
[\varphi (x_i^{(r-1)} , x_j^{(r)}) ]_{
i,j=1}^{r} & \mathbf{0}_{r\times (m+L-r)}\\
\mathbf{0}_{(m+L-r)\times r} & \mathbf{1}_{(m+L-r)\times (m+L-r)} 
\end{bmatrix}\\
 \Psi (x^{(m+L)}) & : = [\psi_{i}(x_j^{(m+L)}|m+L)]_{i,j=1}^{m+L}
\end{align*} 
We consider $\mathcal{X}=\{1,\ldots,m+L\}\sqcup \hat{\mathcal{X}}^{(m)}\sqcup\cdots\sqcup\hat{\mathcal{X}}^{(m+L)}$. Where $\hat{\mathcal{X}}^{(r)}:=\{r\}\times \mathcal{X}^{(r)}\sqcup\{\text{virt}_{r,1},\cdots,\text{virt}_{r,m+L-r}\}$ is the state-space of $r$-physical and $m+L-r$ virtual slots, so that elements of $\hat{\mathcal{X}}^{(r)}$ take the form $(r,x)$. Here, 
\begin{align*}
\mathcal{X}^{(m)}:=\{x_{j}^{(m)}\}_{j\in [[m]]},
 & & 
\mathcal{X}^{(r)}:=\mathbb R  \text{ for } r\in[[m+1,m+L]].
\end{align*}
Although the level $m$ configuration $x^{(m)}$ is fixed, we include
$\hat{\mathcal X}^{(m)}$ as an auxiliary layer in the $L$-ensemble
construction. The random particles live on levels \(m+1,\ldots,m+L\). After
computing the conditional $L$-ensemble kernel on the enlarged space, we
restrict the kernel to the coordinates $(r,x)$ with
$r\in[[m+1,m+L]]$.

We can likewise define the L-ensemble operator
\begin{align*}
    L & = \begin{bmatrix}
0 & \Upsilon & 0 & 0 & 0 & \ldots & 0 \\ 0 & 0 & -\Phi_{m+1} & 0 & 0 & \ldots & 0 \\ 0 & 0 & 0 & -\Phi_{m+2} & 0 & \ldots & 0 \\ 0 & 0 & 0 & 0 & \ldots & \ldots & 0 \\ 0 & 0 & 0 & 0 & 0 & \ldots & -\Phi_{m+L} \\ \Psi & 0 & 0 & 0 & 0 & \ldots & 0 
\end{bmatrix}.
\end{align*}
The conditional $L$-ensemble defined by $L$ on $\eta = \hat{\mathcal{X}}^{(m)}\sqcup\cdots\sqcup\hat{\mathcal{X}}^{(m+L)}$ is the point process defined by the weight $W(\cdot,x^{(m)})$ above. Elements of $\eta$ are either physical points $(r,x)$, with
$r\in[[m,m+L]]$ and $x\in\mathcal X^{(r)}$, or virtual points
$\mathrm{virt}_{r,a}$. 

To compute the correlation kernel, we will invert $\mathbf{1}_{\eta}+L$. We know that a conditional $L$-ensemble defines a determinantal point process with the correlation kernel given by $K_{ij}(x,y)=\left[\mathbf{1}_{\eta}-(\mathbf{1}_{\eta}+L)^{-1}|_{\eta\times\eta}\right]_{ij}(x,y)$~\cite[Proposition 1.2]{Bor11}. We will write $\mathbf{1}_{\eta} + L$ as 
\begin{align*}
    \mathbf{1}_{\eta}+L & = \begin{bmatrix}
        A & B \\ C & D
    \end{bmatrix},
\end{align*}
\begin{align*}
    A  := 0 , & & 
    B := \begin{bmatrix}
        \Upsilon & 0 & \ldots & 0
    \end{bmatrix},\\ 
    C  := \begin{bmatrix}
        0 \\ \ldots \\ 0 \\ \Psi
    \end{bmatrix},  & & 
    D  := \begin{bmatrix}
        \mathbf{1} & -\Phi_{m+1} & 0 & 0 & \ldots & 0 \\ 0 & \mathbf{1} & -\Phi_{m+2} & 0 & \ldots & 0 \\ 0 & 0 & \mathbf{1} & -\Phi_{m+3} & \ldots & 0 \\ 0 & 0 & 0 & \mathbf{1} & \ldots & 0 \\ \ldots & \ldots & \ldots & \ldots & \ldots & \ldots \\ 0 & 0  & 0 & 0 & \ldots & \mathbf{1}
    \end{bmatrix}.
\end{align*} In particular, we find that 
\begin{align*}
    [D^{-1}]_{ij} & = \begin{cases}
        0 & i>j 
        \\ \mathbf{1} & i=j
        \\ \prod_{\ell=0}^{j-i-1} \Phi_{m+i+\ell} & i< j
    \end{cases}.
\end{align*} To compute the inverse of $\mathbf{1}_{\eta}+L$ we define $M=BD^{-1}C-A = \Upsilon\left(\prod_{i=1}^{L}\Phi_{m+i}\right)\Psi$ (which we will label by $G$, to recognize that this is exactly the Gram matrix). By expanding this product, we find that 
\begin{align*}
    G_{ij} & = \begin{cases}
        \upsilon_{i}(\cdot|m)\star\underset{L \text{ times}}{\varphi\star\cdots \star \varphi}\star \psi_{j}(\cdot|m+L) &  i\in [[1,m]] \\ 
        \left(\underset{m+L-i+1\text{ times}}{\varphi \star \cdots \star \varphi} \star \psi_{j}(\cdot|m+L)\right)(\text{virt}) & i\in [[m+1,m+L]]
    \end{cases}
\end{align*} Then, assuming that the Gram matrix is
invertible, the inverse is given by 
\begin{align*}
    (\mathbf{1}_{\eta}+L)^{-1} & = \begin{bmatrix}
        -M^{-1} & M^{-1}BD^{-1} \\ D^{-1}CM^{-1} & D^{-1} - D^{-1}CM^{-1}BD^{-1}
    \end{bmatrix}
\end{align*} and thus $(\mathbf{1}_{\eta}+L)^{-1}|_{\eta\times\eta} = D^{-1}-D^{-1}CM^{-1}BD^{-1}$. 
Re-indexing the blocks of $D$ by their corresponding levels
$r,s\in[[m,m+L]]$, we have
$$
[D^{-1}]_{rs}
=
\begin{cases}
0, & r>s,\\
\mathbf 1, & r=s,\\
\Phi_{r+1}\Phi_{r+2}\cdots\Phi_s, & r<s.
\end{cases}
$$
Thus, setting $i=n_{1},j=n_{2}$ we can compute the kernel $K_{ij}(x_{1},x_{2})$ as 
\begin{align*}
    K(n_{1},x_{1};n_{2},x_{2}) & = \left(\mathbf{1}_{\eta}-(\mathbf{1}_{\eta}+L)^{-1}|_{\eta\times\eta}\right)_{n_{1}n_{2}}(x_{1},x_{2}) 
    \\ & = \left(\prod_{r=n_{1}+1}^{m+L}\Phi_{r}\right)\Psi M^{-1} \Upsilon \left(\prod_{r=m+1}^{n_{2}}\Phi_{r}\right)(x_{1},x_{2}) - \mathbf{1}_{n_{2}>n_{1}}\left(\prod_{r=n_{1}+1}^{n_{2}}\Phi_{r}\right)(x_{1},x_{2}). 
\end{align*}
Which we can rewrite as 
\begin{align*}
    K(n_{1},x_{1};n_{2},x_{2}) & = -\mathbf{1}_{n_{2}>n_{1}}\left(\underset{n_{2}-n_{1}\text{ times}}{\varphi\star\cdots \star\varphi}\right)(x_{1},x_{2}) \\ & + \sum_{i=1}^{m}\sum_{j=1}^{m+L}[G^{-t}]_{ij}\left(\upsilon_{i}(\cdot|m)\star\underset{n_{2}-m\text{ times}}{\varphi \star \cdots \star \varphi}\right)(x_{2})\cdot \left(\underset{m+L-n_{1}\text{ times}}{\varphi \star\cdots\star \varphi} \star \psi_{j}(\cdot|m+L)\right)(x_{1}) 
    \\ & + \sum_{i=m+1}^{n_{2}}\sum_{j=1}^{m+L}[G^{-t}]_{ij}\left(\underset{n_{2}-i+1\text{ times}}{\varphi \star\cdots \star \varphi}\right)(\text{virt},x_{2})\cdot\left(\underset{m+L-n_{1}\text{ times}}{\varphi\star\cdots \star\varphi}\star \psi_{j}(\cdot|m+L)\right)(x_{1}).
\end{align*} This concludes the proof. 
\end{proof}

\section{Convergence of the Polygon Lozenge Kernel}\label{s:leolimit}
In this section, we demonstrate that, for an appropriate choice of parameters, the kernel constructed by Petrov \cite{Pet14} converges to the kernel that we construct in \cref{s:kernel}. When $m=0$ (for the standard rising $\textsc{GUE}$ eigenvalue process), an existing result of Aggarwal and Gorin \cite{AG22} would allow us to construct the $\textsc{GUE}$ kernel as the appropriate limit of the kernel of the uniform measure on lozenge tilings of a hexagonal domain. 
In this appendix, we demonstrate a variant of that result: that the kernel of the rising $\textsc{GUE}$ eigenvalue process started from a fixed configuration arises under a similar scaling limit to that in \cite{AG22}, using instead Petrov's general kernel for lozenge tilings of fairly general polygon domains. We will show pointwise convergence, but this could be upgraded to uniform-on-compact convergence without much additional work.

To set up this result, we briefly discuss the kernel in \cite[Theorem 1]{Pet14}. For the duration of this section, we always take $N$ to be even, and when we take $N\to\infty$ we do so along a sequence comprised of even integers.

For $1\leq n_{1}\leq N+m$, and $1\leq n_{2}\leq N+m-1$ and $x_{1},x_{2}\in\mathbb{Z}$, Petrov constructs the correlation kernel of the uniform measure on tilings of a polygon parameterized by $N,k\in\mathbb{N}$ and a collection of half-integers $A_{1}<B_{1}<\cdots <A_{k}<B_{k}$ with $A_{i},B_{i}\in \mathbb{Z}':=\mathbb{Z}+\frac{1}{2}$ under the constraint that $\sum_{i=1}^{k}(B_{i}-A_{i})=N+m$. When we take asymptotics in $N$ later in this section, we will leave $m$ fixed, so it is equivalent to scale by either $N$ or $N+m$. For simplicity, we will simply scale by $N$. 

We will set $k=m+2$ and we will consider the polygon embedded in $\mathbb{R}^{2}$ such that one face lies on the horizontal axis $n=0$, and such that $k-1$ faces lie along the line $n=N+m$ between each $B_{i}$ and $A_{i+1}$, see \cite[Fig. 2, Fig. 3]{Pet14}. We embed the polygon so that it is symmetric about $x=0$, and assume $A_{1}=-N-1/2, B_{1}=-N/2-1/2, A_{m+2}=N/2-1/2,B_{m+2}=N-1/2$. Between these endpoints there are $(A_{i},B_{i})_{i\in [[2,m+1]]}$  defined by 
\begin{align*}
A_i=\left\lfloor \sqrt{N/2}x_{m+2-i}^{(m)}\right\rfloor-\frac{1}{2},
 & & 
B_i=A_i+1.
\end{align*}We refer to the kernel corresponding to this domain as $K_{\mathbf{P}(N)}(\cdot)$. Applying these specifications to the expression in \cite[Theorem 1]{Pet14}, we obtain
\begin{multline*} 
     K_{\mathbf{P}(N)}(n_1,x_1;n_2,x_2)  =
    -1_{n_2<n_1}1_{x_2\leq x_1}\frac{\prod_{\ell=1}^{n_{1}-n_{2}-1}(x_1-x_2+\ell)}{(n_1-n_2-1)!}
    + \frac{(N-n_1)!}{(N-n_2-1)!} \\ \cdot
    \frac1{(2\pi i)^{2}}
    \oint\limits_{\{z\}}dz\oint\limits_{\{w\}}dw
    \frac{\prod_{\ell=1}^{N-n_{2}-1}(z-x_2+\ell)}{\prod_{\ell=0}^{N-n_{1}}(w-x_1+\ell)}
    \frac{1}{w-z}
    \prod_{r=1}^{m}
    \frac{\lfloor \sqrt{N/2}x_{r}^{(m)}\rfloor -w}{\lfloor \sqrt{N/2} x_{r}^{(m)}\rfloor -z} 
    \frac{\left(-N -w\right)_{\frac{N}{2}}\left(\frac{N}{2}  -w\right)_{\frac{N}{2}}}{\left(-N -z\right)_{\frac{N}{2}}\left(\frac{N}{2}-z\right)_{\frac{N}{2}}},
  \end{multline*} where we have represented each of the Pochhammer symbols in terms of their product definition for the sake of convenience. In that expression, the contour in $z$ contains all poles on the real line which are greater than or equal to $x_{2}$. The contour in $w$ contains the interval $[-N,N]$ on the real line (and thus all of the $\sqrt{N/2} x_{r}^{(m)}$). We assume that neither contour crosses the real axis at integer values. 

  Before we prove the convergence to \cref{t:kernel}, we state a pointwise convergence result and a bound which will be used in the argument. 
  \begin{remark}\label{l:fasm} For all $w\in\mathbb{C}$ and $z =a+it$ for $t\in\mathbb{R}$, it follows immediately from the product formula for the Pochhammer sybmol that
      \begin{align}\label{e:al1}
          \lim_{N\to\infty} \frac{\left(-N -w\sqrt{N/2}\right)_{\frac{N}{2}}\left(\frac{N}{2}  -w\sqrt{N/2}\right)_{\frac{N}{2}}}{\left(-N -z\sqrt{N/2}\right)_{\frac{N}{2}}\left(\frac{N}{2}-z\sqrt{N/2}\right)_{\frac{N}{2}}} = e^{-w^{2}/2+z^{2}/2}.
      \end{align}
      We write \begin{align*}
        P_{N}(z) & := \prod_{k=0}^{N/2-1}(-N -z\sqrt{N/2}+k)\prod_{k=0}^{N/2-1}(N/2-z\sqrt{N/2}+k).
    \end{align*} Let $K\subset\mathbb{C}$ be compact.  From this representation we can also conclude that for some $p,C>0$
    \begin{align*}
        \bigg| \frac{P_{N}(a)}{P_{N}(a+it)}\bigg| \leq \left(1+\frac{p t^{2}}{N/2}\right)^{-N/2} & & \sup_{w\in K}\bigg|\frac{P_{N}(w)}{P_{N}(a)} \bigg|\leq C.
    \end{align*} As a consequence of these bounds, for all $R>0$ there exist $C,c>0$ such that for all $N$ sufficiently large, and for  all $w\in K$ and $z=a+it$ for $t\in\mathbb{R}$, 
      \begin{align}\label{e:al2}
          \bigg|\frac{\left(-N -w\sqrt{N/2}\right)_{\frac{N}{2}}\left(\frac{N}{2}  -w\sqrt{N/2}\right)_{\frac{N}{2}}}{\left(-N -z\sqrt{N/2}\right)_{\frac{N}{2}}\left(\frac{N}{2}-z\sqrt{N/2}\right)_{\frac{N}{2}}}\bigg| \leq C\left(e^{-ct^{2}}+(1+|t|)^{-R}\right).
      \end{align}
  \end{remark}
\begin{proposition}\label{p:asymeq}
Up to gauge transformation, for all $n_{1},n_{2}>m$, $x_{1},x_{2}\in\mathbb{R}$, and $x^{(m)}\in\mathrm{Conf}_{m}^{0}(\mathbb{R})$,
\begin{align*}
      \lim_{N\to\infty} \sqrt{N/2} K_{\mathbf{P}(N)}(N+m-n_{1},\lfloor x_{1}\sqrt{N/2}\rfloor;N+m-n_{2},\lfloor x_{2}\sqrt{N/2}\rfloor)= K_{x^{(m)}}^{\textsc{GUE}}(n_{1},x_{1};n_{2},x_{2}).
     \end{align*} 
\end{proposition}
\begin{proof} We take $K_{\mathbf{P}(N)}(n_{1},x_{1};n_{2},x_{2})$ and apply the change of variables \begin{align*}z \mapsto \sqrt{N/2}z,
     & & w \mapsto \sqrt{N/2}w,
    & & x_{1} \mapsto \lfloor \sqrt{N/2} x_{1}\rfloor ,
     & & x_{2}\mapsto \lfloor \sqrt{N/2} x_{2}\rfloor ,
     \end{align*}\vspace{-2mm}
\begin{align*}
    n_{1} \mapsto N+m-n_{1}, & & n_{2}\mapsto N + m -n_{2}.
\end{align*} After making a gauge transformation by $(N/2)^{(n_{1}-n_{2})/2}$ and multiplying by $\sqrt{N/2}$, we obtain the following expression for the kernel
    \begin{align*}
        & =
    -1_{n_2>n_1}1_{x_2\leq x_1}\frac{ \prod_{\ell=1}^{n_{2}-n_{1}-1}(\lfloor \sqrt{N/2}x_{1}\rfloor-\lfloor \sqrt{N/2}x_{2}\rfloor+\ell)}{(n_2-n_1-1)!}
      \\ &   +\frac{(n_{1}-m)!}{(n_{2}-m-1)!}\frac{1}{(2\pi i)^{2}}
    \oint\limits_{c(x_{j}^{(m)}\geq x_{2})}dz\oint\limits_{c(\text{large})}dw
    \frac{\prod_{\ell=1}^{n_{2}-m-1}(\sqrt{N/2}z-\lfloor \sqrt{N/2}x_{2}\rfloor+\ell)}{\prod_{\ell=0}^{n_{1}-m}(\sqrt{N/2}w-\lfloor \sqrt{N/2}x_{1}\rfloor+\ell)}
    \frac{1}{w-z}
    \\ & \cdot  \prod_{r=1}^{m}
    \frac{\sqrt{N/2}w- \lfloor \sqrt{N/2}x_{r}^{(m)}\rfloor  }{\sqrt{N/2}z-  \lfloor \sqrt{N/2}x_{r}^{(m)} \rfloor }
    \frac{\left(-N -w\sqrt{N/2}\right)_{\frac{N}{2}}\left(\frac{N}{2}  -w\sqrt{N/2}\right)_{\frac{N}{2}}}{\left(-N -z\sqrt{N/2}\right)_{\frac{N}{2}}\left(\frac{N}{2}-z\sqrt{N/2}\right)_{\frac{N}{2}}}.
    \end{align*} 
    The contour in $z$ contains all poles on the real line which are greater than or equal to $x_{2}$. The contour in $w$ contains the interval $[-\sqrt{2N},\sqrt{2N}]$ on the real line (and thus all of the $\sqrt{2/N}\lfloor\sqrt{N/2} x_{r}^{(m)}\rfloor$).
    
    We deform the contour in $w$ so that it does not enclose the contour in $z$ but instead simply encircles  the set $\{\sqrt{2/N}\lfloor \sqrt{N/2}x_{1}\rfloor -\ell\sqrt{2/N}|\ell\in [[0,n_{1}-m]]\}$ and no other $w$ poles. The residue that we pick up through this deformation cancels with the first term. Then we deform the $z$ contour (we are still treating $N$ as finite) so that it becomes an open contour, picking up a factor of $-1$ to account for orientation. What remains is 
    \begin{align*}
        & \frac{(n_{1}-m)!}{(n_{2}-m-1)!}\frac{-1}{(2\pi i)^{2}}
    \int\limits_{a-i\infty}^{a+i\infty}dz\oint\limits_{c(x_{1})}dw
    \frac{\prod_{\ell=1}^{n_{2}-m-1}(\sqrt{N/2}z-\lfloor \sqrt{N/2}x_{2}\rfloor+\ell)}{\prod_{\ell=0}^{n_{1}-m}(\sqrt{N/2}w-\lfloor \sqrt{N/2}x_{1}\rfloor+\ell)}
    \frac{1}{w-z}
    \\ & \cdot  \prod_{r=1}^{m}
    \frac{\sqrt{N/2}w- \lfloor \sqrt{N/2}x_{r}^{(m)}\rfloor  }{\sqrt{N/2}z-  \lfloor \sqrt{N/2}x_{r}^{(m)} \rfloor }
    \frac{\left(-N -w\sqrt{N/2}\right)_{\frac{N}{2}}\left(\frac{N}{2}  -w\sqrt{N/2}\right)_{\frac{N}{2}}}{\left(-N -z\sqrt{N/2}\right)_{\frac{N}{2}}\left(\frac{N}{2}-z\sqrt{N/2}\right)_{\frac{N}{2}}},
    \end{align*} where $a>x_{2}$ and for all $x_{j}^{(m)}>x_{2}$, then  $x_{j}^{(m)}>a$, and similarly, if $x_{1}>x_{2},$ then $x_{1}>a$. 
    
    We can write this as 
    \begin{align*}
        & \frac{(n_{1}-m)!}{(n_{2}-m-1)!}\frac{-1}{(2\pi i)^{2}}
    \int\limits_{a-i\infty}^{a+i\infty}dz\oint\limits_{c(x_{1})}dw
    \frac{(z-x_2)^{n_{2}-m-1}}{(w-x_{1})^{n_{1}-m+1}}
    \frac{1}{w-z}
    \prod_{r=1}^{m}
    \frac{w- x_{r}^{(m)}  }{z-  x_{r}^{(m)} }
    \\ & \cdot  \frac{\left(-N -w\sqrt{N/2}\right)_{\frac{N}{2}}\left(\frac{N}{2}  -w\sqrt{N/2}\right)_{\frac{N}{2}}}{\left(-N -z\sqrt{N/2}\right)_{\frac{N}{2}}\left(\frac{N}{2}-z\sqrt{N/2}\right)_{\frac{N}{2}}}
     \textsc{ERR}^{1}(z,w)\textsc{ERR}^{2}(z,w),
    \end{align*} 
    where
  \begin{align*}
      \textsc{ERR}^{1}(z,w)    := \frac{\prod_{\ell=1}^{n_{2}-m-1}\left(1+\frac{O(1)+\ell}{\sqrt{N/2}(z-x_{2})}\right)}{\prod_{\ell=0}^{n_{1}-m}\left(1+\frac{O(1)+\ell}{\sqrt{N/2}(w-x_{1})} \right)}, & &  \textsc{ERR}^{2}(z,w) := \prod_{r=1}^{m}
    \frac{w-\sqrt{2/N}\lfloor\sqrt{N/2}x_{r}^{(m)}\rfloor }{z-\sqrt{2/N}\lfloor \sqrt{N/2} x_{r}^{(m)}\rfloor } \prod_{r=1}^{m}\frac{z-x_{r}^{(m)}}{w-x_{r}^{(m)}}. 
  \end{align*} 
   Since $x_{1},x_{2},x^{(m)}$ are all fixed, we can assume that their separation from the contours of integration is bounded below by a fixed positive constant. Therefore, we have the following pointwise convergence for all $z,w$ in the contours of integration,
  \begin{align*}
       \lim_{N\to\infty}\textsc{ERR}^{1}(z,w) =1,
      &  & \lim_{N\to\infty}\textsc{ERR}^{2}(z,w) = 1. 
    \end{align*} We also have pointwise convergence of the integrand due to \eqref{e:al1}. Finally, noting that for fixed $m, x_{1},x_{2},n_{1},n_{2},x^{(m)}$, the expression
    \begin{align*}
        \frac{(z-x_2)^{n_{2}-m-1}}{(w-x_{1})^{n_{1}-m+1}}
    \frac{1}{w-z}
    \prod_{r=1}^{m}
    \frac{w- x_{r}^{(m)}  }{z-  x_{r}^{(m)} },
    \end{align*} at most polynomial in $|t|$ for all $z=a+it$, we can bound this factor by, for instance, the expression $C(1+|t|)^{d}$ for some $C,d>0$. Then we apply \eqref{e:al2} for any $R>d+2$ and dominated convergence to finish the proof. 
     \end{proof}

\section{Conditioning on the Top Level}\label{s:compare}
In this section, we prove the formula developed by Metcalfe for the kernel of the $\textsc{GUE}$ corners process with fixed top level, using the same methods we developed in \cref{s:kernel}.
\begin{proposition}[Proposition 2.4, \cite{Met13}]\label{p:closedkernel}
    The kernel of the rising $\textsc{GUE}$ eigenvalue process with a fixed top configuration $x^{(m)}\in\mathrm{Conf}^{0}_{m}(\mathbb{R})$, for $n_{1}\in [[m-1]],n_{2}\in[[m]]$ and $x_{1},x_{2}\in\mathbb{R}$, is given by 
    \begin{align*}
        \widehat{K}_{x^{(m)}}^{\textsc{GUE}}(n_{1},x_{1};n_{2},x_{2}) & = \frac{(m-n_{2})!}{(m-n_{1}-1)!}\frac{1}{(2\pi i)^{2}}\oint_{c(x_{1}\leq x_{k}^{(m)})}dz \oint_{c(x_{2})}dw \frac{(z-x_{1})^{m-n_{1}-1}}{(w-x_{2})^{m-n_{2}+1}}\frac{1}{w-z}\prod_{r=1}^{m}\frac{w-x_{r}^{(m)}}{z-x_{r}^{(m)}},
    \end{align*} where $c(x_{1}\leq x_{k}^{(m)})$ is the contour which contains all $x_{k}^{(m)}$ $k\in[[m]]$ such that $x_{1}\leq x_{k}^{(m)}$, and where $c(x_{2})$ is the contour which contains $x_{2}$ and no points $x_{k}^{(m)}$ unless $x_{2}=x_{k}^{(m)}$. Both contours are positively oriented.
\end{proposition}
\begin{proof}
The product determinant measure takes the form
\begin{align*}
    \mathbb{P}\left(\emptyset\prec x^{(1)}\prec\cdots\prec x^{(m)}\big|x^{(m)}\right) & = \mathbf{1}_{\emptyset\prec x^{(1)}\prec \cdots \prec x^{(m)}}\prod_{i<j}(x_{i}^{(m)}-x_{j}^{(m)})^{-1}
    \\ & = \prod_{k=1}^{m} \mathrm{det}\left[\varphi (x_{i}^{(k-1)},x_{j}^{(k)})\right]_{i,j=1}^{k} \det\left[\rho_{i}(x_{j}^{(m)}|m)\right]_{i,j=1}^{m}.
\end{align*} This is the uniform distribution over the Gelfand-Tsetlin pattern with top row given by $x^{(m)}$. The functions $\varphi(\cdot,\cdot),\varphi(\text{virt},\cdot)$ are the same as those defined in \cref{s:kernel}, and we define 
\begin{align*}
    \rho_{i}(\cdot |m) :\{x_{j}^{(m)}\}_{j\in [[m]]}\to\mathbb{C}, & & i\in [[m]],
\end{align*} by 
\begin{align*}\rho_{i}(x_{j}^{(m)}|m) := [h(x^{(m)})^{-T}]_{ij} \text{ for all $i,j\in[[m]]$.}
\end{align*} We likewise define the $\star$-convolution operation for a function $f:\mathbb{R}^{2}\to\mathbb{R}$ and $h:S\to\mathbb{R}$ for $S\subset\mathbb{R}$, $|S|\in\mathbb{N}$, as
\begin{align*}
    (f\star h)(x) & = \sum_{y\in S}f(x,y)h(y).
\end{align*}
We will use the expression 
\begin{align*}
    \rho_{i}(x_{j}^{(m)}|m)  & = \frac{1}{\prod_{r\neq j}(x_{j}^{(m)}-x_{r}^{(m)})}\sum_{l=1}^{m} (-1)^{l-1}e_{l-1}(x^{(m)}\setminus x_{j}^{(m)}) \frac{(m-l)!2^{-\frac{i-l}{2}}}{((i-l)/2)!(m-i)!}\mathbf{1}_{i-l\text{ is even}, i\geq l}.
\end{align*}
We again change to use the functions $\widetilde{\varphi}(\cdot,\cdot)$ for the sake of eliminating divergences. We find 
\begin{align*}
    G_{ij} & = \widetilde{\varphi}\star\widetilde{\varphi}^{(i,m)}(\text{virt},y) \star \rho_{j}(y|m) = \frac{1}{\sqrt{2\pi}(m-i)!}\sum_{k=1}^{m} h_{m-i}(x_{k}^{(m)})[(h^{(m)})^{-T}]_{jk} 
    = \frac{1}{\sqrt{2\pi}(m-i)!}\mathbf{1}_{ij}.
\end{align*}This Gram matrix is invertible, and $G^{-t}_{ij}=\sqrt{2\pi}(m-i)!\mathbf{1}_{ij}.$ Furthermore, we note the following transformation, 
\begin{align*}
    \widetilde{\rho}_{i}(x_{1}|n_{1}) & = \sum_{j=1}^{m} \widetilde{\varphi}^{(n_{1},m)}(x_{1},x_{j}^{(m)})\rho_{i}(x_{j}^{(m)}|m) 
    \\ & = \sum_{j=1}^{m}\varphi^{(n_{1},m)}(x_{1},x_{j}^{(m)})\rho_{i}(x_{j}^{(m)}|m)  - \mathbf{1}_{n_{1}<m}\sum_{j=1}^{m}\sum_{k=n_{1}+1}^{m}\frac{h_{m-k}(x_{j}^{(m)})\psi_{k}(x_{1}|n_{1})}{\sqrt{2\pi}(m-k)!}\rho_{i}(x_{j}^{(m)}|m) 
    \\ & = \rho_{i}(x_{1}|n_{1})  - \mathbf{1}_{n_{1}<m}\sum_{j=1}^{m}\sum_{k=n_{1}+1}^{m}\frac{h_{m-k}(x_{j}^{(m)})\psi_{k}(x_{1}|n_{1})}{\sqrt{2\pi}(m-k)!}\rho_{i}(x_{j}^{(m)}|m).
\end{align*}
We can apply the standard version of the rising Eynard-Mehta theorem \cite[Theorem 4.4]{Bor11} to see that 
\begin{align*}
    \widehat{K}^{\textsc{GUE}}_{x^{(m)}}(n_{1},x_{1};n_{2},x_{2}) & = -\widetilde{\varphi}^{(n_{1},n_{2})}(x_{1},x_{2}) + \sum_{i=1}^{n_{2}}\sum_{j=1}^{m} [G^{-t}]_{ij}\left(\widetilde{\varphi}\star \widetilde{\varphi}^{(i,n_{2})}\right)(\text{virt},x_{2})\cdot \widetilde{\rho}_{j}(x_{1}|n_{1}). 
\end{align*}
Noting that 
\begin{align*}
    \widetilde{\varphi}\star\widetilde{\varphi}^{(i,n_{2})}(\text{virt},x_{2}) & = \frac{h_{n_{2}-i}(x_{2})}{\sqrt{2\pi}(n_{2}-i)!},
    \\  -\widetilde{\varphi}^{(n_{1},n_{2})}(x_{1},x_{2}) & = -\varphi^{(n_{1},n_{2})}(x_{1},x_{2}) + \mathbf{1}_{n_{1}<n_{2}}\sum_{k=n_{1}+1}^{n_{2}}\frac{h_{n_{2}-k}(x_{2})\psi_{k}(x_{1}|n_{1})}{\sqrt{2\pi}(n_{2}-k)!},
\end{align*}
the kernel can be expressed as
\begin{align*}
    \widehat{K}^{\textsc{GUE}}_{x^{(m)}}(n_{1},x_{1};n_{2},x_{2}) & = -\varphi^{(n_{1},n_{2})}(x_{1},x_{2}) + \mathbf{1}_{n_{1}<n_{2}}\sum_{k=n_{1}+1}^{n_{2}}\frac{h_{n_{2}-k}(x_{2})\psi_{k}(x_{1}|n_{1})}{\sqrt{2\pi}(n_{2}-k)!}
    \\ & + \sum_{i=1}^{n_{2}}\sqrt{2\pi}(m-i)!\frac{h_{n_{2}-i}(x_{2})}{\sqrt{2\pi}(n_{2}-i)!}\sum_{l=1}^{m}\varphi^{(n_{1},m)}(x_{1},x_{l}^{(m)})\rho_{i}(x_{l}^{(m)}|m) 
    \\ & - \mathbf{1}_{n_{1}<n_{2}}\sum_{i=1}^{n_{2}}\sqrt{2\pi}(m-i)!\frac{h_{n_{2}-i}(x_{2})}{\sqrt{2\pi}(n_{2}-i)!} \sum_{l=1}^{m}\sum_{k=n_{1}+1}^{m}\frac{h_{m-k}(x_{l}^{(m)})\psi_{k}(x_{1}|n_{1})}{\sqrt{2\pi}(m-k)!}\rho_{i}(x_{l}^{(m)}|m) .
\end{align*}
We deal first with the first and third sums. In particular, in the third sum, after we take the sum over $l\in[[m]]$, what remains is
\begin{align*}
    - \mathbf{1}_{n_{1}<n_{2}}\sum_{i=n_{1}+1}^{n_{2}}\frac{h_{n_{2}-i}(x_{2})\psi_{i}(x_{1}|n_{1})}{\sqrt{2\pi}(n_{2}-i)!}, 
\end{align*} which exactly cancels with the first sum which appears in the kernel. We must still deal with the remaining terms
\begin{align*}
    \widehat{K}^{\textsc{GUE}}_{x^{(m)}}(n_{1},x_{1};n_{2},x_{2}) & = -\varphi^{(n_{1},n_{2})}(x_{1},x_{2})  + \sum_{i=1}^{n_{2}}\sqrt{2\pi}(m-i)!\frac{h_{n_{2}-i}(x_{2})}{\sqrt{2\pi}(n_{2}-i)!}\sum_{k=1}^{m}\varphi^{(n_{1},m)}(x_{1},x_{k}^{(m)})\rho_{i}(x_{k}^{(m)}|m). 
\end{align*}
We focus on the sum, 
and expand the contour integral expression for $\rho_{i}(x_{l}^{(m)}|m)$ and substitute the expression for $\varphi^{(a,b)}(x,y)$,
\begin{align*}
    & =\sum_{k=1}^{m} \frac{\varphi^{(n_{1},m)}(x_{1},x_{k}^{(m)})}{\prod_{r\neq k}(x_{k}^{(m)}-x_{r}^{(m)})} \sum_{j=1}^{n_{2}}(-1)^{j-1}e_{j-1}(x^{(m)}\setminus x_{k}^{(m)})(m-j)!\sum_{i=1}^{n_{2}} \frac{h_{n_{2}-i}(x_{2})2^{-(i-j)/2}}{((i-j)/2)!(n_{2}-i)!}\mathbf{1}_{i-j\text{ even};i\geq j}
    \\ & =\sum_{k=1}^{m} \frac{\varphi^{(n_{1},m)}(x_{1},x_{k}^{(m)})}{\prod_{r\neq k}(x_{k}^{(m)}-x_{r}^{(m)})} \sum_{j=1}^{n_{2}}(-1)^{j-1}e_{j-1}(x^{(m)}\setminus x_{k}^{(m)})(m-j)!\sum_{i=0}^{\lfloor \frac{n_{2}-j}{2}\rfloor} \frac{h_{n_{2}-j-2i}(x_{2})2^{-i}}{i!(n_{2}-j-2i)!}
    \\ & = \sum_{k=1}^{m} \frac{\varphi^{(n_{1},m)}(x_{1},x_{k}^{(m)})}{\prod_{r\neq k}(x_{k}^{(m)}-x_{r}^{(m)})} \sum_{j=1}^{n_{2}}(-1)^{j-1}e_{j-1}(x^{(m)}\setminus x_{k}^{(m)})\frac{(m-j)!}{(n_{2}-j)!}x_{2}^{n_{2}-j}
    \\ & = \sum_{k=1}^{m} \frac{\varphi^{(n_{1},m)}(x_{1},x_{k}^{(m)})}{\prod_{r\neq k}(x_{k}^{(m)}-x_{r}^{(m)})} \frac{d^{m-n_{2}}}{(dw)^{m-n_{2}}}\prod_{r\neq k}(w-x_{r}^{(m)})\bigg|_{w=x_{2}}
    \\ & =\sum_{k=1}^{m} \frac{\varphi^{(n_{1},m)}(x_{1},x_{k}^{(m)})}{\prod_{r\neq k}(x_{k}^{(m)}-x_{r}^{(m)})}\frac{(m-n_{2})!}{2\pi i} \oint_{c(\infty)}dw\frac{\prod_{r\neq k}(w-x_{r}^{(m)})}{(w-x_{2})^{m-n_{2}+1}}
    \\ & =\sum_{k|x_{1}\leq x_{k}^{(m)}}  \frac{(m-n_{2})!}{(m-n_{1}-1)!}\frac{1}{2\pi i} \oint_{c(\infty)}dw\frac{(x_{k}^{(m)}-x_{1})^{m-n_{1}-1}}{(w-x_{2})^{m-n_{2}+1}}\prod_{r\neq k}\frac{w-x_{r}^{(m)}}{x_{k}^{(m)}-x_{r}^{(m)}}
    \\ & = \frac{(m-n_{2})!}{(m-n_{1}-1)!}\frac{1}{(2\pi i)^{2}} \oint_{c(\infty)}dw\oint_{c(x_{1}\leq x_{k}^{(m)})}dz \frac{(z-x_{1})^{m-n_{1}-1}}{(w-x_{2})^{m-n_{2}+1}}\frac{1}{w-z}\prod_{r=1}^{m}\frac{w-x_{r}^{(m)}}{z-x_{r}^{(m)}}.
\end{align*}
Finally, we change the $w$ contour to $c(x_{2}),$ which cancels with the $-\varphi^{(n_{1},n_{2})}(x_{1},x_{2})$. Thus,
\begin{align*}
    \widehat{K}^{\textsc{GUE}}_{x^{(m)}}(n_{1},x_{1};n_{2},x_{2}) & = \frac{(m-n_{2})!}{(m-n_{1}-1)!}\frac{1}{(2\pi i)^{2}} \oint_{c(x_{1}\leq x_{k}^{(m)})}dz \oint_{c(x_{2})}dw\frac{(z-x_{1})^{m-n_{1}-1}}{(w-x_{2})^{m-n_{2}+1}}\frac{1}{w-z}\prod_{r=1}^{m}\frac{w-x_{r}^{(m)}}{z-x_{r}^{(m)}}.
\end{align*}
\end{proof}
\section{Proof of \cref{p:steepest}} \label{s:steepest}
 
\begin{proof}[Proof of \cref{p:steepest}]
We begin by proving part $(1)$. The quantitative Morse lemma (see, for example, \cite[Theorem B.1]{DGBLR25}) tells us that, under the assumption of the Lipschitz bound on $g_{n}''$, then the size of the neighborhood $U_{n}$ is bounded below by a function which is proportional to $|g''(z_{0}^{(n)})|$. Since this latter expression is uniformly bounded below in our setting, we conclude that we can choose a sequence of $U_{n}$ with radius uniformly bounded below by some fixed $r>0$. Therefore, for all $r'\in (0,r)$ there exists $n_{r}\in\mathbb{N}$ and open set $U$ of radius $r'$ such that $U\subset U_{n}$ and $z_{0}^{(n)}\in U$ for all $n>n_{r}$.

We proceed to prove part $(2)$. We use the set $U$ to separate the integral into four terms
\begin{align*}
    \int_{C_{1}^{(n)}}dz\int_{C_{2}^{(n)}}dw f_{n}(z,w)\frac{e^{n(g_{n}(z)-g_{n}(w))}}{w-z} & = \int_{C_{1}^{(n)}\cap U }dz\int_{C_{2}^{(n)}\cap U}dw f_{n}(z,w)\frac{e^{n(g_{n}(z)-g_{n}(w))}}{w-z} 
    \\ & + \int_{C_{1}^{(n)}\cap U }dz\int_{C_{2}^{(n)}\setminus U}dw f_{n}(z,w)\frac{e^{n(g_{n}(z)-g_{n}(w))}}{w-z} 
    \\ & + \int_{C_{1}^{(n)}\setminus U }dz\int_{C_{2}^{(n)}\cap U }dw f_{n}(z,w)\frac{e^{n(g_{n}(z)-g_{n}(w))}}{w-z} 
    \\ & + \int_{C_{1}^{(n)}\setminus U }dz\int_{C_{2}^{(n)}\setminus U }dw f_{n}(z,w)\frac{e^{n(g_{n}(z)-g_{n}(w))}}{w-z}.
\end{align*} 
 
Using the assumptions of part $(2)$, we conclude that 
\begin{align*}
    \bigg| \int_{C_{1}^{(n)}\cap U }dz\int_{C_{2}^{(n)}\setminus U}dw f_{n}(z,w)\frac{e^{n(g_{n}(z)-g_{n}(w))}}{w-z}\bigg|  &  \leq  L Mc_{4}^{-1}e^{-nc_{3}},
    \\   \bigg|\int_{C_{1}^{(n)}\setminus U }dz\int_{C_{2}^{(n)}\cap U }dw f_{n}(z,w)\frac{e^{n(g_{n}(z)-g_{n}(w))}}{w-z} \bigg| & \leq LMc_{4}^{-1}e^{-nc_{3}},
    \\  \bigg| \int_{C_{1}^{(n)}\setminus U }dz\int_{C_{2}^{(n)}\setminus U }dw f_{n}(z,w)\frac{e^{n(g_{n}(z)-g_{n}(w))}}{w-z}\bigg| &  \leq  L^{2}Mc_{4}^{-1}e^{-2nc_{3}} . 
\end{align*}
We now appeal to the definition 
\begin{align*}
    \int_{C_{1}^{(n)}\cap U }dz\int_{C_{2}^{(n)}\cap U}dw f_{n}(z,w)\frac{e^{n(g_{n}(z)-g_{n}(w))}}{w-z} & = \lim_{\varepsilon\to 0}\int_{C_{1}^{(n)}\cap U \setminus B_{\varepsilon,n}}dz\int_{C_{2}^{(n)}\cap U\setminus B_{\varepsilon,n}}dw f_{n}(z,w)\frac{e^{n(g_{n}(z)-g_{n}(w))}}{w-z}, 
\end{align*} where $B_{\varepsilon,n}$ is a ball of radius $\varepsilon>0$ around $z_{0}^{(n)}$, 

There exists $\delta>0$ such that
    \begin{multline*}
    \lim_{\varepsilon\to 0}\bigg|\int_{C_{1}^{(n)}\cap U \setminus B_{\varepsilon,n}}dz\int_{C_{2}^{(n)}\cap U\setminus B_{\varepsilon,n}}dw f_{n}(z,w)\frac{e^{n(\zeta_{n}(z)^{2}-\zeta_{n}(w)^{2})/2}}{w-z}\bigg| 
    \\ \leq  \lim_{\varepsilon\to 0}M\int_{(-\delta,\delta)  \setminus \zeta_{n}(B_{\varepsilon,n})}dt\int_{(-\delta,\delta)\setminus \zeta_{n}(B_{\varepsilon,n})}ds e^{-n(t^{2}+s^{2})/2}\bigg|\frac{(\zeta^{-1}_{n})'(it)(\zeta^{-1}_{n})'(s)}{\zeta_{n}^{-1}(s)-\zeta_{n}^{-1}(it)}\bigg| 
\end{multline*}
Using the assumptions that there exists $c_{1}>0$ such that $\inf_{n}|g_{n}''(z_{0}^{(n)})|>c_{1}$ and that there exist $c_{2}>0$ and for all $n>n_{0}$ and $\sup_{n}\sup_{|z-z_{0}^{(n)}|<r}|g_{n}'''(z)|<c_{2}$, there exist $C,C',\varepsilon_{0}>0$ such that for all $n>n_{0}$ and $\varepsilon\in (0,\varepsilon_{0}]$,
\begin{multline*}
    \lim_{\varepsilon\to 0}M\int_{(-\delta,\delta)  \setminus \zeta_{n}(B_{\varepsilon,n})}dt\int_{(-\delta,\delta)\setminus \zeta_{n}(B_{\varepsilon,n})}ds e^{-n(t^{2}+s^{2})/2}\bigg|\frac{(\zeta^{-1}_{n})'(it)(\zeta^{-1}_{n})'(s)}{\zeta_{n}^{-1}(s)-\zeta_{n}^{-1}(it)}\bigg| 
    \\ \leq \lim_{\varepsilon\to 0}C\cdot M\int_{(-\delta,\delta)  \setminus \zeta_{n}(B_{\varepsilon,n})}dt\int_{(-\delta,\delta)\setminus \zeta_{n}(B_{\varepsilon,n})}ds \frac{e^{-n(t^{2}+s^{2})/2}}{\sqrt{s^{2}+t^{2}}} \leq C'\cdot Mn^{-1/2}.
\end{multline*}
Thus, there exists $c>0$ such that the entire integral is bounded by $cML^{2}c_{4}^{-1}n^{-1/2}.$
The proof of \cref{r:steepest} is identical. 
\end{proof}

\section{Generalization of \cref{t:convrate}}\label{a:fullgen}
For the purpose of this section, we set $T(m)$ to be a sequence of integers such that $T(m)\ll m$ and $T(m)\to\infty$ as $m\to\infty$. As before, the notation  $x^{(m)}$ will denote an element of $\mathrm{Conf}_{m}^{0}(\mathbb{R})$.
\begin{assumption}[Local Density]\label{a:ld} We say $x^{(m)}\in\mathrm{Conf}_{m}^{0}(\mathbb{R})$ satisfies the \textit{local density assumption} if there exist scales $D(m),Q(m)$ satisfying $ D(m)\ll T(m)\ll Q(m)\ll m$; and $\rho_{*},\rho^{*}>0$ satisfying $\rho_{*}\leq \rho^{*}<\infty$ such that 
\begin{enumerate}
    \item For any interval $I\subset (X-Q(m)/m,X+Q(m)/m)$ of length $D(m)/m$, $$N_{I}(x^{(m)}/\sqrt{m})\in [\rho_{*}D(m),\rho^{*}D(m)].$$ 
    \item For any interval $I$ of length $Q(m)/m$, $N_{I}(x^{(m)}/\sqrt{m})\leq \rho^{*} Q(m)$. 
\end{enumerate} 
\end{assumption}

\begin{assumption}[Intermediate Scale]\label{a:is} We say $x^{(m)}\in\mathrm{Conf}_{m}^{0}(\mathbb{R})$ satisfies the \textit{intermediate scale assumption} if for all $X\in(-2,2)$ and $R,\delta>0$ there exist $A_{R,\delta}:=A_{R,\delta}(X)>0$ such that
\begin{align*}\bigg|\frac{1}{m}\sum_{RT/m\le|X-x_{r}^{(m)}/\sqrt{m}|\le \delta }\frac{1}{X - x_{r}^{(m)}/\sqrt{m}}\bigg|\leq A_{R,\delta}.
 \end{align*} 
\end{assumption} 
It is a straightforward consequence of the local semicircle law and eigenvalue rigidity that \cref{a:ld} and \cref{a:is} are satisfied for eigenvalues $x^{(m)}$ of Wigner matrices satisfying \cref{d:wigner} on a high probability event in $\mathbb{P}_{m}^{W}$ for all $\log^{2(4+\varepsilon)/\varepsilon}{(m)}\ll T(m)\ll m$, where $\varepsilon>0$ is as in \cref{d:wigner}.

In the case that the starting configuration $x^{(m)}$ cannot be assumed to satisfy a rigidity condition like that in \cref{thm:rigid}, we can still obtain a result like \cref{t:convrate} using \cref{a:ld} and \cref{a:is}. First we state the following result, analogous to \cite[Proposition 2.6]{GP19}. 
\begin{proposition}\label{p:unique2}
    Under \cref{a:ld} and \cref{a:is} there exists $m_{0}\in\mathbb{N}$ such that for all $m>m_{0}$, $S'_{m}(z)=0$, \eqref{e:sderiv}, has a unique root $z_{0}^{(m)}$ in the upper half-plane. Moreover, there exists a compact set $K$ such that $z_{0}^{(m)}\in K$ for all $m>m_{0}$. 
\end{proposition}
There is also an analogous statement to \cite[Theorem 2.7]{GP19}.
\begin{theorem}\label{t:27}
There exists $m_{0}\in\mathbb{N}$ and $r(n)=o(1)$ such that for all $K\subset\mathbb{N}\times \mathbb{R}$ compact, there exists $C:=C(K)>0$ such that for all $m>m_{0}$, and for all $(n_{1},x_{1}),(n_{2},x_{2})$ and $x^{(m)}$ satisfying \cref{a:ld} and \cref{a:is},
\begin{align*}
    \bigg|\frac{1}{\sqrt{m}}\widetilde{K}_{x^{(m)}}^{\textsc{GUE}}\left(n_{1}+T,X\sqrt{m}+\frac{x_{1}}{\sqrt{m}};n_{2}+T,X\sqrt{m}+\frac{x_{2}}{\sqrt{m}}\right) -K^{\text{sine}}_{-1/z_{0}^{(m)}}(n_{1},x_{1};n_{2},x_{2}) \bigg| \leq C\cdot r(n),
\end{align*} where $z_{0}^{(m)}$ is the root from \cref{p:unique2}.
\end{theorem} We also have the following result, analogous to \cite[Theorem 2.10]{GP19}, with additional conditions which ensure that the sequence of $z_{0}^{(m)}$ converge as $m\to\infty$.  
\begin{theorem}\label{t:210}
    Suppose that for a sequence $\{x^{(m)}\}_{m\in\mathbb{N}}$, \cref{a:ld} and \cref{a:is} hold at each step, and further assume that there exists a $\sigma$-finite measure $\mu_{\text{loc}}$ such that the following vague limit holds, 
    \begin{align}\label{e:weak}
        \lim_{m\to\infty}\frac{1}{T}\sum_{r=1}^{m}\delta_{(x_{r}^{(m)}/\sqrt{m}-X)/(T/m)} = \mu_{\text{loc}},
    \end{align} and such that for all $R>0,$ there exists a limit $d(R)$ such that 
    \begin{align}\label{e:dlimit}
        \lim_{m\to\infty}\frac{1}{m}\sum_{|x_{r}^{(m)}/\sqrt{m}-X|\geq RT/m}\frac{1}{X-x_{r}^{(m)}/\sqrt{m}} = d(R).
    \end{align}
    Then, for all $n_{1},n_{2}>0,x_{1},x_{2}\in\mathbb{R},$
    \begin{align*}
        \lim_{m\to\infty} \frac{1}{\sqrt{m}}\widetilde{K}_{x^{(m)}}^{\textsc{GUE}}\left(n_{1}+T,X\sqrt{m}+\frac{x_{1}}{\sqrt{m}};n_{2}+T,X\sqrt{m}+\frac{x_{2}}{\sqrt{m}}\right) = K^{\text{sine}}_{-1/z_{0}}(n_{1},x_{1};n_{2},x_{2}),
    \end{align*} 
    where $z_{0}$ is the unique root in the upper-half plane of the equation
    \begin{align*}
       X+\frac{1}{z} -d(R) -\int_{\mathbb{R}}\left(\frac{1}{z-u}+\frac{\mathbf{1}_{|u|>R}}{u}\right)\mu_{\text{loc}}(du)=0.
   \end{align*}
\end{theorem}

The proof of \eqref{t:210} is substantially simpler than that of \eqref{t:27}, because the additional assumptions \eqref{e:weak} and \eqref{e:dlimit} allow us to prove a version of \cref{l:sconv}, and from there the argument in \cref{s:convrate} follows with only small modifications. In the statement of this theorem, we have omitted the rate, since it will depend on the rates of convergence in \eqref{e:weak} and \eqref{e:dlimit}. In the case of \cref{l:sconv}, the rate is determined by the bounds in statement of the eigenvalue rigidity in \cref{thm:rigid}. 

We discuss briefly the additional work needed to obtain a result like \cref{t:27}. The main challenge is just that it is more involved to verify the conditions for \cref{p:steepest} under these weaker assumptions. The argument after that point is identical, so we must focus on how to modify the results in \cref{s:zeros} and \cref{s:contours}. 

To show \cref{p:unique2}, we would replicate the approach of Gorin and Petrov \cite[Section 5]{GP19}, with some simplifications. The result showing that there are at most two complex conjugate critical points away from the real line (\cref{l:maximaginary}) would be unchanged. To show the existence of a critical point in the upper half-plane as in \cref{l:unique}, we would define $\mathcal{R}(c,\kappa_{2})$ identically as in \cref{s:convrate},
\begin{align*}
    \mathcal{R}(\kappa_{1},\kappa_{2}) := \{x+iy| |x|< \kappa_{1}^{-1}, \kappa_{2}<y<\kappa_{2}^{-1}\},
\end{align*} and construct a continuous curve $\gamma^{(m)}\subset\mathcal{R}(\kappa_{1},\kappa_{2})$ such that for all $z\in\gamma^{(m)}$, $\Im(S'_{m}(z))=0$. We then show that the real part of $S'_{m}(z)$ must change signs along $\gamma^{(m)}$. 

We outline the key steps of this argument. We start with the imaginary part of $S'_{m}(z)$,
\begin{align*}
    \Im(S_{m}'(x+iy)) & =-\frac{y}{x^{2}+y^{2}} + \frac{T}{m}y + \frac{1}{T} \sum_{r=1}^{m}\frac{y}{(x-u_{r}^{(m)})^{2}+y^{2}}.
\end{align*}
\begin{lemma}\label{l:im1} For all $X\in (-2,2)$, there exist $\kappa_{a},\kappa_{b}>0,m_{0}\in\mathbb{N}$ such that for all $\kappa_{1}\in(0,\kappa_{a}),\kappa_{2}\in(0,\kappa_{b}),m>m_{0}$ and $x^{(m)}$ satisfying \cref{a:ld} and \cref{a:is}, there exist $a_{m}(\kappa_{2}), b_{m}(\kappa_{2}) \in\mathbb{R}_{>0}$ and $C_{1}:=C_{1}(X),C_{2}:=C_{2}(X)>0$ such that  $C_{1}\kappa_{2}^{1/2}<a_{m}(\kappa_{2})<b_{m}(\kappa_{2})<C_{2}\kappa_{2}^{1/2}$ and such that
\begin{enumerate}
    \item $\Im(S_{m}'(x+i\kappa_{2}))>0$ for all  $|x|\leq \kappa_{1}^{-1}$ such that $x\in \mathbb{R}\setminus(-b_{m}(\kappa_{2}),b_{m}(\kappa_{2}))$. 
    \item  $\Im(S_{m}'(x+i\kappa_{2}))<0$ for all $|x|\leq \kappa_{1}^{-1}$ such that $x\in (-a_{m}(\kappa_{2}),a_{m}(\kappa_{2}))$.
    \item For all $(x,y)$ such that $\sqrt{x^{2}+y^2} = \kappa_{2}^{-1}$ and $y\geq \kappa_{2}$, $\Im(S_{m}'(x+iy))>0$. 
\end{enumerate} 
\end{lemma}
This is proved similarly to \cite[Lemma 5.2]{GP19}, using the assumptions to produce an estimate on the sum which appears in $\Im(S'_{m}(x+iy))$. 
 
We continue by constructing a curve $\gamma^{(m)}\subset\mathcal{R}(\kappa_{1},\kappa_{2})$ along which $\Im(S_{m}'(z))=0$. 

\begin{lemma}\label{l:curve}
    For all $X\in (-2,2)$ there exist $\kappa_{a},\kappa_{b}>0,m_{0}\in\mathbb{N}$ such that for all $\kappa_{1}\in (0,\kappa_{a}),\kappa_{2}\in (0,\kappa_{b}),m>m_{0}$ and $x^{(m)}$ satisfying \cref{a:ld} and \cref{a:is}, there exists a curve $\gamma^{(m)}\subset\mathcal{R}(\kappa_{1},\kappa_{2})$ such that
    \begin{enumerate}
        \item For all $z\in \gamma^{(m)},$ $\Im(S_{m}'(z))=0$, $\Im(z)\in [\kappa_{2},\kappa_{2}^{-1}]$ and $|\Re(z)|\leq \kappa_{1}^{-1}$.
        \item The curve $\gamma^{(m)}$ has an endpoint in each of the sets 
        \begin{align*}
            \{x+i\kappa_{2}|x\in [-b_{m}(\kappa_{2}),-a_{m}(\kappa_{2})]\}, & & \{x+i\kappa_{2}|x\in [a_{m}(\kappa) , b_{m}(\kappa_{2}) ]\}.
        \end{align*} 
    \end{enumerate}
\end{lemma}
 
 We then deal with the real part of $S_{m}'(z)$,
\begin{align*}
    \Re(S_{m}'(x+iy)) & =  X + \frac{x}{x^{2}+y^{2}} +  \frac{T}{m}x - \frac{1}{T} \sum_{r=1}^{m}\frac{x-u_{r}^{(m)}}{(x-u_{r}^{(m)})^{2}+y^{2}}.
\end{align*}
\begin{lemma}\label{l:re} For all $X\in(-2,2)$, there exist $\kappa_{0}>0,m_{0}\in\mathbb{N}$ such that for all $\kappa_{2}\in (0,\kappa_{0}),$ $m>m_{0}$ and $x^{(m)}$ satisfying \cref{a:ld} and \cref{a:is}, $\mathrm{sgn}(\Re(S_{m}'(x_{1}+i\kappa_{2})))\neq \mathrm{sgn}(\Re(S_{m}'(x_{2}+i\kappa_{2})))$ for all $x_{1}\in (-b_{m}(\kappa_{2}),-a_{m}(\kappa_{2}))$, $x_{2}\in (a_{m}(\kappa_{2}),b_{m}(\kappa_{2}))$.  
\end{lemma}
As before, this is proved analogously to a simpler case of \cite[Lemma 5.5]{GP19}, using \cref{a:ld} and \cref{a:is} to estimate the sum which appears in $\Re(S_{m}'(x+iy))$. 

We can also prove versions of the statements in \cref{s:contours} with only \cref{a:ld} and \cref{a:is}. We modify the proofs of \cref{l:conditions1} and \cref{l:lengthbound} by, along the same lines as the proofs of \cite[Lemma 6.5, Lemma 6.8]{GP19}, noting that if, by way of contradiction, one of the statements fails, then there must exist a subsequence $\{m_{k}\}_{k\in\mathbb{N}}$ along which convergence results like \eqref{e:weak} and \eqref{e:dlimit} hold, but such that $S_{m}''(z_{m_{k}})$ diverges or converges to $0$, and then replicating the arguments in the proofs of those results to produce a contradiction along that subsequence. 

Results analogous to \cref{l:conditions1}, \cref{l:conditions3}, and remarks \cref{r:Rbd} and \cref{r:extra} also follow from \cref{a:ld} and \cref{a:is} by arguments analogous to those in \cite[Section 6]{GP19}, in particular \cite[Lemma 6.1, Lemma 6.2, Lemma 6.4]{GP19}. There is substantial simplification in this argument compared to the case that \cite{GP19} considers, reflecting the fact that $S_{m}$ has simpler structure in the case considered in this paper).  
\bibliographystyle{amsplain0}
\bibliography{main.bib}
\end{document}